\documentclass[nointlimits,11pt,oneside]{amsart}
\usepackage{amssymb,cases,enumitem,verbatim,mathrsfs}
\usepackage{xcolor}
\usepackage{hyperref}
\usepackage[T1]{fontenc}
\usepackage[utf8]{inputenc}
\usepackage[normalem]{ulem}
\usepackage{todonotes}
\usepackage[foot]{amsaddr}
\usepackage{orcidlink}

\hypersetup{
	colorlinks=true,
	linkcolor=blue,
	citecolor=blue,
	filecolor=green,
	urlcolor=cyan,
}

\makeatletter
\renewcommand*{\eqref}[1]{%
	\hyperref[{#1}]{\textup{\tagform@{\ref*{#1}}}}%
}
\makeatother

\setlist[enumerate,1]{label={\textup{(\roman*)}}}

\usepackage[%
	a4paper,
	total={16cm,23cm},
	left=3cm, top=3cm,
	marginparsep=2pt]
{geometry}

\theoremstyle{plain}
\newtheorem{theorem}{Theorem}[section]
\newtheorem{lemma}[theorem]{Lemma}
\newtheorem{corollary}[theorem]{Corollary}
\newtheorem{proposition}[theorem]{Proposition}
\theoremstyle{definition}
\newtheorem{remark}[theorem]{Remark}

\newtheorem{definition}[theorem]{Definition}
\newtheorem{example}[theorem]{Example}

\numberwithin{equation}{section}

\def\M{\mathcal M}

\def\N{\mathbb N}
\def\R{\mathbb R}
\def\Z{\mathbb Z}

\DeclareMathOperator\supp{supp}
\DeclareMathOperator{\sgn}{sgn}

\def\L1loc{L^1_{\text{loc}}}

\newcommand{\newtopology}{\xi}
\newcommand{\ACR}{(ACR) }

\begin{document}

\title[Mean and pointwise ergodicity for composition operators on r.i.~spaces]{Mean and pointwise ergodicity for composition operators on rearrangement-invariant spaces}
\author{Thomas Kalmes$^1$\orcidlink{0000-0001-7542-1334} and Dalimil Pe{\v s}a$^{1,2}$\orcidlink{0000-0001-6638-0913}}

\address{$^1$Chemnitz University of Technology, Faculty of Mathematics, 09107 Chemnitz, Germany}
\email{thomas.kalmes@math.tu-chemnitz.de}

\address{$^2$Department of Mathematical Analysis, Faculty of Mathematics and Physics, Charles University, Sokolovsk\'a~83, 186~75 Praha~8, Czech Republic}
\email{dalimil.pesa@mathematik.tu-chemnitz.de}


\thanks{The second author was supported by the grant 23-04720S of the Czech Science Foundation.}

\begin{abstract}
    We study ergodicity of composition operators on rearrangement-invariant Banach function spaces. More precisely, we give a natural and easy-to-check condition on the symbol of the operator which entails mean ergodicity on a very large class of rearrangement-invariant Banach function spaces. Further, we present some natural additional assumptions that allow us to obtain pointwise ergodicity. The class of spaces covered by our results contains many non-reflexive spaces, such as the Lorentz spaces $L^{p, 1}$ and $L^{p,\infty}$, $p \in (1, \infty)$, Orlicz spaces $L \log^{\alpha} L$ and $\exp L^{\alpha}$, $\alpha > 0$, and the spaces $L^1$ and $L^{\infty}$ over measure spaces of finite measure. The main novelty in our approach is the application of a new locally convex topology which we introduce and which lies strictly between the norm topology and the weak topology induced by the associate space. Throughout, we give several examples which illustrate the applicability of our results as well as highlight the necessity and optimality of our assumptions.\\

    \noindent Keywords: re\-arrange\-ment-invariant (quasi-)Banach function space, composition operator, mean ergodic operator, power-bounded operator.\\

    \noindent MSC 2020: 47A35, 47B33, 46E30.
\end{abstract}

\date{\today}

\maketitle

\makeatletter
   \providecommand\@dotsep{2}
\makeatother

\section{Introduction} \label{SectionIntroduction}

For Banach spaces, or more generally locally convex spaces, $X$ a continuous linear operator $T$ (briefly \emph{continuous operator}; we denote the space of all continuous operators on $X$ by $\mathscr{L}(X)$) is said to be \emph{mean ergodic} if the limits of its Ces\`aro means,
\begin{align*}
\lim_{n\rightarrow\infty} &\frac{1}{n} \sum_{i=0}^{n-1} T^i f,  &f \in X,
\end{align*}
exist in the topology of $X$. One of the first results on mean ergodicity was von Neumann’s seminal work \cite{vonNeumann32} in 1932, where he proved that unitary operators on Hilbert spaces are mean ergodic. Since then, many authors investigated various aspects of mean ergodic behaviour in operator theory; for an overview of the period up to the 1980s, see for example \cite{Krengel85} and the references therein. A central notion closely tied to mean ergodicity is that of \emph{power-boundedness}: an operator  $T \in \mathscr{L}(X)$ is \emph{power-bounded} if the set of its iterates $\{T^n; \; n \in \mathbb{N}\}$ is equicontinuous, which in the Banach space setting means $\sup_{n \in \mathbb{N}} \lVert T^n \rVert < \infty$. For such operators, it is well-known that mean ergodicity is equivalent to the decomposition
\[
X = \operatorname{Ker}(I - T) \oplus \overline{\operatorname{Im}(I - T)}.
\]
A Banach space $X$ is called \emph{mean ergodic} if every power-bounded operator on it is mean ergodic. Notable early results on mean ergodic spaces include the Hilbert spaces and the $L^p$-spaces for $1 < p < \infty$, as shown by Riesz \cite{Riesz38} in 1938, while Kakutani \cite{Kakutani1938}, Lorch \cite{Lorch39}, and Yosida \cite{Yosida1938} proved (independently) that all reflexive Banach spaces are mean ergodic. Concerning the converse, in 1976, Sucheston \cite[Problem 1, p.~285]{Sucheston1976} posed the following question: \emph{If every contraction in a Banach space $X$ is mean ergodic, is $X$ reflexive?} (By a simple renorming argument, it is easily seen that this problem is equivalent to the one whether every mean ergodic Banach space $X$ is reflexive.) In the celebrated article \cite{FoLiWo01} from 2001, Fonf, Lin and Wojtaszczyk proved (among other things) that for a Banach space $X$ \emph{with a basis}, being mean ergodic indeed implies reflexivity.  

A class of operators which plays a prominent role in functional analysis and operator theory is the class of composition operators. Given a set $\mathcal{R}$ and a self mapping $\phi$ on $\mathcal{R}$, beyond the fundamental question under which conditions the composition operator induced by $\phi$ acts on a given space of functions $X$ defined on $\mathcal{R}$, i.e.~whether the operation $T_\phi f = f\circ \phi$ results in a function belonging to $X$ whenever $f$ does, it is a natural task to study properties of $T_\phi$ on $X$ in terms of its \emph{symbol} $\phi$. As general references for composition operators on $L^p$-spaces and spaces of continuous, differentiable, or holomorphic functions, we only mention the textbooks \cite{SinghManhas93, Shapiro93, CowenMacCluer95}, respectively.

In recent years, mean ergodicity and power-boundedness of composition operators on various spaces of (generalised) functions have attracted the attention of a large number of authors. We give only a sample of articles (and refer to references therein); see e.g.\ \cite{Albanese25, AlJoMe22, AsJoKa25, BeGoJoJo16, BeGoJoJo16b, BeJo21, BernardesBonillaPinto25, BoDo11, BoDo11b, FeGaJo18, FeGaJo20, GoJoJo16, JoRo20, Kalmes19, Kalmes20, KaSa22, KaSa23, Santacreu24}.

A further natural aspect to consider in this setting is pointwise ergodicity. While mean ergodicity concerns convergence in norm (or, more generally, a topology) of the function space, the question of almost everywhere convergence of the Ces\`aro means is at the heart of classical ergodic theory and goes back to Birkhoff’s fundamental paper \cite{Birkhoff31}. This naturally raises the problem of determining under which additional assumptions normwise convergence implies pointwise convergence almost everywhere. We only mention the classical generalisation of Birkhoff's result by Dunford and Schwartz \cite{DunfordSchwartz56}, by Calder{\' o}n \cite{Calderon68}, as well as the recent article by Mart\'in-Reyes and de la Rosa \cite{Martin-ReyesDelaRosa22}.

In this paper, we study composition operators on \emph{rearrangement-invariant (quasi-)Banach function spaces} (briefly, r.i.~(quasi-)Banach function spaces). These are (quasi-)Banach spaces of (equivalence classes of) measurable functions on a fixed measure space; the archetypal example are the classical $L^p$-spaces but the class contains many other important function spaces, such as the Lorentz spaces $L^{p,q}$ (and their various generalisations such as Lorentz--Zygmund and --Karamata spaces), Orlicz spaces, classical Lorentz spaces, and more. The precise definitions and properties of these spaces, including the notion of associate spaces, will be introduced in Section~\ref{section: preliminaries} below. The same is true for our standing assumption on the underlying measure space, i.e.~that it is resonant. Our main goal is to establish mean ergodic and pointwise ergodic theorems for composition operators in this setting, under conditions that are significantly weaker than reflexivity.

In particular, we identify a mild and widely satisfied decay condition---namely, the \emph{\ACR property} (see Definition~\ref{DefACR})---for both the space $X$ and its associate space $X'$ and a natural growth condition on the symbol $\phi$---namely, the \emph{power-measure-boundedness} (see Definition~\ref{DefPowerMeasureBounded}---which together ensure that a composition operator $T_\phi$ is mean ergodic. Additionally, we prove that power-measure-boundedness of $\phi$ is equivalent to $T_\phi$ being power-bounded on every r.i.~quasi-Banach function space. This property is also necessary for power-boundedness in many specific r.i.~quasi-Banach function spaces, with some exceptions (we provide a non-trivial counterexample in Example~\ref{CounterExSVPower}). We note that power-measure-boundedness is a much weaker condition than assuming that the symbol is a measure-preserving transformation, which is a common assumption in this context (see e.g.~\cite{Calderon68, DurcikSlavikova25, EinsiedlerWard11, KrauseMirekTao22, MirekSzarek22, Petersen89, Tao08}), and that we do not require invertibility of the symbol. Our results also require only the most natural assumptions on the underlying measure space, i.e.~that it is $\sigma$-finite and resonant (see Definition~\ref{DefResonant}); notably, we do not require finiteness of the measure.

It is important to note, that the \ACR property is very mild. Not only is it satisfied by every r.i.~quasi-Banach function space with absolutely continuous quasinorm (see Proposition~\ref{PropACNimpliesACR}), it turns out that it is satisfied by every r.i.~quasi-Banach function space that is not equal to $L^{\infty}$ ``near infinity'', see \cite[Theorem~4.16]{MusilovaNekvinda24} for the precise formulation of this statement and its proof. Furthermore, if the measure of the underlying measure space is finite, then the \ACR property is trivially satisfied for every r.i.~quasi-Banach function space. Hence, it is clear that the assumptions that the \ACR property holds for both $X$ and $X'$ is much weaker than assuming reflexivity of $X$ (which is equivalent to both $X$ and $X'$ having absolutely continuous norm, see Corollary~\ref{CorReflexBFS}). Some concrete examples of non-reflexive spaces covered by our results are given below.

Conversely, it turns out that the \ACR property, both for an r.i.~Banach function space $X$ on the one hand and for its associate space $X'$ on the other hand, is necessary for a general power-bounded composition operator on $X$ to be mean ergodic (see Examples \ref{CounterExampleL1} and \ref{CounterExampleLInfty}, respectively).

Let us now present a special case of our main results that can be formulated easily without any new concepts. Here, we assume absolute continuity of the norm for $X$ (see Definition \ref{DefACqN}) and the \ACR property for $X'$, which is still much weaker than reflexivity. The most important examples of non-reflexive r.i.~Banach function spaces covered by this theorem are the Lorentz spaces $L^{p, 1}$, $p \in (1, \infty)$, which play an important role in the theories of interpolation and Sobolev embeddings, and, when the underlying measure space is of finite measure, also the space $L^1$ and the Orlicz spaces $L \log^{\alpha} L$, $\alpha > 0$; here, the restriction on the measure is fundamental in the case of $L^1$, as we show in Example~\ref{CounterExampleL1}, while in the case of the spaces $L \log^{\alpha} L$ it is merely a consequence of the fact their definition itself assumes that the measure is finite. The last paragraph of the theorem then outlines what can be shown under a weaker assumption on $X$ than absolute continuity of the norm; Theorems~\ref{ThmErgodic(w')*} and \ref{ThmErgodicNormLocalised} contain the precise formulation. Spaces covered by those theorems, i.e.~satisfying the \ACR property but without absolutely continuous norm, include the weak Lebesgue spaces $L^{p, \infty}$, $p \in (1, \infty)$ and, when the underlying measure space is of finite measure, the space $L^{\infty}$ and the Orlicz spaces $\exp L^{\alpha}$, $\alpha > 0$; again, this restriction on the measure is fundamental in the case of $L^{\infty}$ (see Example~\ref{CounterExampleLInfty}), while for $\exp L^{\alpha}$ it is a mere consequence of the way these spaces are defined.

\begin{theorem} \label{ThmErgodicLite}
	Let $X$ be an r.i.~Banach function space over a $\sigma$-finite and resonant measure space $(\mathcal{R},\mu)$, and let $X'$ be the corresponding associate space. Assume that $X$ has absolutely continuous norm and $X'$ has the \ACR property. Assume that the measurable mapping $\phi \colon \mathcal{R} \to \mathcal{R}$ has the property that there is some $A \in (0, \infty)$ such that it holds for every $n \in \mathbb{N}$, $n\geq 1$, that
	\begin{align*}
		\mu(\phi^{-n}(E)) &\leq A \mu(E) &\text{for every measurable } E \subseteq \mathcal{R}.
	\end{align*}    
	Then the operator $T_{\phi}$ is mean ergodic on $(X, \lVert \cdot \rVert_X)$, i.e.~there is some continuous operator $T \colon X \to X$ such that we have for every $f \in X$ that
	\begin{align} \label{ThmErgodicLite:E0}
		\frac{1}{n}\sum_{i = 0}^{n-1} T_{\phi}^i f &\to Tf & \text{in } (X, \lVert \cdot \rVert_X)  \text{ as } n \to \infty.
	\end{align}
	
	Furthermore, the above defined operator $T$ is positive and there exists the associate operator $T'\colon X' \to X'$ which is also continuous.
	Finally, the operator $T$ is a projection onto the space $\operatorname{Ker} (I - T_{\phi})$ while the operator $I - T$ is a projection onto $\overline{\operatorname{Im} (I - T_{\phi})}^{\lVert \cdot \rVert_X}$, and we have
	\begin{align}
		X &= \operatorname{Ker} (I - T_{\phi}) \bigoplus \overline{\operatorname{Im} (I - T_{\phi})}^{\lVert \cdot \rVert_X} \label{ThmErgodicLite:E0.2}.
	\end{align}

    Finally, if the space $X$ does not have absolutely continuous norm but it has the \ACR property, than the limit operator $T$ still exists, however the convergence in \eqref{ThmErgodicLite:E0} holds only for those functions that themselves have absolutely continous norm; for other functions the  Ces\`aro means converge to the limit operator with respect to a locally convex topology on $X$ that we denote $\newtopology$ and that is weaker than the norm topology but stronger than the weak topology $\sigma(X, X')$. Also, in this case, $I - T$ is a projection onto the a~priori larger set $\overline{\operatorname{Im}(I - T_{\phi})}^{\newtopology}$ and the closure on the right-hand side of \eqref{ThmErgodicLite:E0.2} must likewise be taken with respect to the $\newtopology$ topology.
\end{theorem}

The key to our mean ergodicity results is the new locally convex topology $\newtopology$ for r.i.~quasi-Banach function spaces which we introduce and study in Section~\ref{SecRITopo}. This topology, defined via a family of rearrangement-invariant norms, is crucial for our work, because it serves as a middle ground between the norm topology and the weak topology induced by $X'$. It has several strongly desirable properties, chief among which is that it is complete (not just sequentially- or quasi-complete) and that the topological dual of $X$ when equipped with this topology is equal to the associate space $X'$, provided that $X$ has the \ACR property (this is, in fact, the reason why we need the \ACR property in our results).

As for pointwise ergodicity, it turns out that with additional assumptions on the symbol $\phi$, the norm convergence of the  Ces\`aro means implies pointwise convergence. We have found two distinct assumptions that guarantee this behaviour which we present in Definition~\ref{Def(I)}; the full version of the pointwise ergodic result is presented as Theorem~\ref{ThmErgodicPointwise}. However, as before, we present a simpler version of our theorem that is well compatible with Theorem~\ref{ThmErgodicLite}:

\begin{theorem} \label{ThmErgodicPointwiseLite}
	Let $X$ be an r.i.~Banach function space over a $\sigma$-finite and resonant measure space $(\mathcal{R},\mu)$, and let $X'$ be the corresponding associate space. Assume that $X$ has absolutely continuous norm and $X'$ has the \ACR property. Assume that the measurable mapping $\phi \colon \mathcal{R} \to \mathcal{R}$ has the property that there is some $A \in (0, \infty)$ such that it holds for every $n \in \mathbb{N}$, $n\geq 1$, that
	\begin{align}
		A^{-1} \mu(E) \leq \mu(\phi^{-n}(E)) &\leq A \mu(E) &\text{for every measurable } E \subseteq \mathcal{R}. \label{ThmErgodicPointwiseLite:E1}
	\end{align}    
	Then the operator $T_{\phi}$ is pointwise ergodic on $X$, i.e.~we have for every $f \in X$ that
	\begin{align} \label{ThmErgodicLite:E0.1}
		\frac{1}{n}\sum_{i = 0}^{n-1} T_{\phi}^i f &\to Tf & \mu\text{-a.e.~as } n \to \infty,
	\end{align}
    where $T \colon X \to X$ is the continuous linear operator whose existence is asserted in Theorem~\ref{ThmErgodicLite}.

    In the case when $X$ does not have absolutely continuous norm but it has only the \ACR property, we get the pointwise convergence for those functions that have absolutely continuous norm in $X$.
\end{theorem}

Here, \eqref{ThmErgodicPointwiseLite:E1} is one of the two possible sufficient conditions for pointwise ergodicity presented in Definition~\ref{Def(I)}. The key to this result is the r.i.~version of the Maximal Ergodic Theorem that we present as Corollary~\ref{CorMaxErgodicX}; this result itself follows via interpolation from a weak-type estimate for the maximal ergodic operator on $L^1$ (Theorem~\ref{ThmMaxErgodicL1}).

It is worth noting that in the case when $\mu(\mathcal{R}) < \infty$, then every r.i.~Banach function space is embedded into $L^1$, whence it is pointless to examine any other space and our contribution lies in weakening the known sufficient conditions onto $\phi$ (cf.~\cite{Birkhoff31} or \cite[Theorem~2.30]{EinsiedlerWard11}). Our main contribution to the pointwise ergodic theory is therefore the case when $\mu(\mathcal{R}) = \infty$, where our result shows that the pointwise convergence holds for every function that has absolutely continuous norm in any given r.i.~Banach function space.

The paper is organised as follows. First, in Section~\ref{section: preliminaries} we recall the necessary background on r.i.~(quasi-)Banach function spaces, including their representation, duality, and norm properties. In order to keep the paper self-contained and since a significant part of the required background is rather recent and not yet covered by any standard books, this is a rather extensive section. Section~\ref{section: composition operators on r.i.} studies composition operators on  r.i.~(quasi-)Banach function spaces, culminating in our results on power-boundedness for composition operators. Section~\ref{SecRITopo} is devoted to the construction and study of the locally convex topology mentioned above. This serves as a preparation for our main ergodic results, but given how nicely the new topology combines the properties of the norm topology (e.g.~net-completeness) with those of the weak topology $\sigma(X, X')$ (e.g.~that the topological dual coincides with the associate space, assuming the \ACR condition), we believe it to be of independent interest. In Section~\ref{SectionErgodic}, we present our mean ergodicity results for composition operators on r.i.~Banach function spaces. Additionally, throughout, we give several examples which illustrate the applicability of our results and highlight the necessity and optimality of our assumptions. Finally, Section~\ref{SectionPointwise} is devoted to pointwise ergodicity. There we investigate when the norm convergence of Ces\`aro means established earlier can in fact be strengthened to almost everywhere convergence. The argument relies on a suitable maximal inequality and shows that pointwise ergodic theorems can be obtained under rather natural and comparatively mild assumptions on the underlying transformation. It is worth noting, that in Sections~\ref{section: composition operators on r.i.}, \ref{SecRITopo}, \ref{SectionErgodic}, and \ref{SectionPointwise} we work under the standing assumption that the underlying measure space is $\sigma$-finite and resonant. 

\section{Preliminaries}\label{section: preliminaries}

The purpose of this section is to establish the theoretical background that serves as the foundation for our research. The definitions and notation is intended to be as standard as possible. The usual references for most of this theory are \cite{BennettSharpley88} and \cite{MeiseVogt97}.

Throughout this paper we will denote by $(\mathcal{R}, \mu)$, and occasionally by $(\mathcal{S}, \nu)$, some arbitrary (totally) $\sigma$-finite measure space. Given a $\mu$-measurable set $E \subseteq \mathcal{R}$ we will denote its characteristic function by $\chi_E$. By $\mathcal{M}(\mathcal{R}, \mu)$ we will denote the set of all extended complex-valued $\mu$-measurable functions defined on $\mathcal{R}$. As is customary, we will identify functions that coincide $\mu$-almost everywhere. We will further denote by $\mathcal{M}_0(\mathcal{R}, \mu)$ and $\mathcal{M}_+(\mathcal{R}, \mu)$ the subsets of $\mathcal{M}(\mathcal{R}, \mu)$ containing, respectively, the functions finite $\mu$-almost everywhere and the non-negative functions.

When there is no risk of confusing the reader, we will abbreviate $\mu$-almost everywhere, $\mathcal{M}(\mathcal{R}, \mu)$, $\mathcal{M}_0(\mathcal{R}, \mu)$, and $\mathcal{M}_+(\mathcal{R}, \mu)$ to $\mu$-a.e., $\mathcal{M}$, $\mathcal{M}_0$, and $\mathcal{M}_+$, respectively.

By the support of a given measurable function $f$ on $(\mathcal{R}, \mu)$ (i.e.~not an equivalence class) we mean the set where it is non-zero, i.e.
\begin{equation*}
	\supp f =\left \{x \in \mathcal{R}; \; \lvert f \rvert> 0  \right \}.
\end{equation*}

We will, generally, be treating functions and their respective equivalence classes interchangeably whenever the distinction is not important. However, to avoid confusion, let us clarify some statements we may make when talking about $g \in \mathcal{M}$ (i.e.~an equivalence class of functions):
\begin{enumerate}
    \item By saying that $g$ is a simple function, we mean that there is a representative that is a simple function.
    \item By saying that the support of $g$ is included in some set $E$, or writing $\supp g \subseteq E$, we mean that $g$ is equal to $0$ a.e.~outside of $E$.
\end{enumerate}

When $X$ is a set and $f, g \colon X \to \mathbb{C}$ are two maps satisfying that there is some positive and finite constant $C$, depending only on $f$ and $g$, such that $\lvert f(x) \rvert \leq C \lvert g(x) \rvert$ for all $x \in X$, we will denote this by $f \lesssim g$. We will also write $f \approx g$, or sometimes say that $f$ and $g$ are equivalent, whenever both $f \lesssim g$ and $g \lesssim f$ are true at the same time. We choose this general definition because we will use the symbols ``$\lesssim$'' and ``$\approx$'' with both functions and functionals. 

When $X, Y$ are two topological linear spaces, we will denote by $Y \hookrightarrow X$ that $Y \subseteq X$ and that the identity mapping $I \colon Y \rightarrow X$ is continuous.

As for some special cases, we will denote by $\lambda^n$ the classical $n$-dimensional Lebesgue measure on $\mathbb{R}^n$, with the exception of the $1$-dimensional case where we simply write $\lambda$. We will further denote by $m$ the counting measure over $\mathbb{N}$ or $\mathbb{Z}$. When $p \in (0, \infty]$ we will denote by $L^p$ the classical Lebesgue space (of functions in $\mathcal{M}(\mathcal{R}, \mu)$), defined for finite $p$ as the set
\begin{equation*}
	L^p = \left \{ f \in \mathcal{M}(\mathcal{R}, \mu); \; \int_{\mathcal{R}} \lvert f \rvert^p \: d\mu < \infty \right \},
\end{equation*}
equipped with the customary (quasi)norm
\begin{equation*}
	\lVert f \rVert_p = \left( \int_ {\mathcal{R}} \lvert f \rvert^p \: d\mu \right)^{\frac{1}{p} },
\end{equation*}
and through the usual modifications for $p=\infty$. In the special case when $(\mathcal{R}, \mu) = (\mathbb{N}, m)$ we will denote this space by $\ell^p$. Note that in this paper we consider $0$ to be an element of $\mathbb{N}$.

When $\phi \colon \mathcal{R} \to \mathcal{R}$ is measurable and $n \in \mathbb{N}$ then $\phi^n$ is the composition
\begin{equation*}
	\phi^n = \underbrace{\phi \circ \cdots \circ \phi}_{\text{$n$-times}}.
\end{equation*}
Specifically, $\phi^1 = \phi$ and $\phi^0$ is the identity mapping on $\mathcal{R}$. The same notation will also be used for self-mappings on $\mathcal{M}$. When further $E \subseteq \mathcal{R}$ is measurable, $\phi^{-n}(E)$ denotes the preimage of $E$ with respect to the mapping $\phi^n$.

\subsection{Non-increasing rearrangement} \label{SectionNon-increasingRearrangement}
We now present the crucial concept of the non-in\-crea\-sing rearrangement of a function and state some of its properties that will be important for our work. We proceed in accordance with \cite[Chapter~2]{BennettSharpley88}.

\begin{definition} \label{DefNIR}
	The \emph{distribution function} $f_*$ of $f \in \mathcal{M}$ is defined for $s \in [0, \infty)$ by
	\begin{equation*}
		f_*(s) = \mu(\{ x \in \mathcal{R}; \; \lvert f(x) \rvert > s \}).
	\end{equation*}	
	The \emph{non-increasing rearrangement} $f^*$ of the said function is then defined for $t \in [0, \infty)$ by
	\begin{equation*}
		f^*(t) = \inf \{ s \in [0, \infty); \; f_*(s) \leq t \}.
	\end{equation*}
\end{definition}

For the basic properties of the distribution function and the non-increasing rearrangement, with proofs, see \cite[Chapter~2, Propositions~1.3 and 1.7]{BennettSharpley88}, respectively. We consider those properties to be classical and well known and we will be using them without further explicit reference; let us, however, just point out one critical property we will use several times in the paper, that is
\begin{equation} \label{EqAlmostSubAdd*}
	(f+g)^* (t) \leq f^* \left( \frac{t}{2} \right) + g^* \left( \frac{t}{2} \right)
\end{equation}
for all $f,g \in \mathcal{M}$ and all $t \in [0, \infty)$

An important concept used in the paper is that of equimeasurability.

\begin{definition} \label{DEM}
	We say that $f \in \mathcal{M}(\mathcal{R}, \mu)$ and $g \in \mathcal{M}(\mathcal{S}, \nu)$ are \emph{equimeasurable} if $f^* = g^*$, where the non-increasing rearrangements are computed with respect to the corresponding measures.
\end{definition}

It is not hard to show that two functions are equimeasurable if and only if their distribution functions coincide too (with each distribution function being considered with respect to the corresponding measure).

A very significant classical result is the Hardy--Littlewood inequality. For a proof, see for example \cite[Chapter~2, Theorem~2.2]{BennettSharpley88}.

\begin{theorem} \label{THLI}
	It holds for all $f, g \in \mathcal{M}$ that
	\begin{equation*}
		\int_\mathcal{R} \lvert fg \rvert \: d\mu \leq \int_0^{\infty} f^*g^* \: d\lambda.
	\end{equation*}
\end{theorem}

It follows directly from this result that
\begin{equation} \label{HLI_sup}
	\sup_{\substack{\tilde{g} \in \mathcal{M} \\ \tilde{g}^* = g^*}} \int_{\mathcal{R}} \lvert f \tilde{g} \rvert \: d\mu \leq \int_0^{\infty} f^*g^* \: d\lambda
\end{equation}
holds for every $f,g \in \mathcal{M}$. This motivates the definition of resonant measure spaces as those spaces where we have equality in \eqref{HLI_sup}.

\begin{definition} \label{DefResonant}
	A $\sigma$-finite measure space $(\mathcal{R}, \mu)$ is said to be \emph{resonant} if it is true for all $f, g \in \mathcal{M}(\mathcal{R}, \mu)$ that
	\begin{equation} \label{DefResonant:E1}
		\sup_{\substack{\tilde{g} \in \mathcal{M} \\ \tilde{g}^* = g^*}} \int_\mathcal{R} \lvert f \tilde{g} \rvert \: d\mu = \int_0^{\infty} f^* g^* \: d\lambda.
	\end{equation}
\end{definition}

\begin{remark} \label{RemResonant}
	Note that $\tilde{g}^*$ does not depend on the sign of $\tilde{g}$, hence for every admissible $\tilde{g}_1$ there is some admissible $\tilde{g}_2$ such that
	\begin{equation*}
		\int_\mathcal{R} \lvert f \tilde{g}_1 \rvert \: d\mu = \int_\mathcal{R} f \tilde{g}_2 \: d\mu.
	\end{equation*}
	Hence, the modulus on the left-hand side of \eqref{DefResonant:E1} could as well be placed outside of the integral or omitted altogether, provided that the functions for which the resulting integrals fail to be defined would not be considered in the supremum. 
\end{remark}

The property of being resonant is a crucial one. Luckily, there is a straightforward characterisation of resonant measure spaces which we present below. For its proof and further details see \cite[Chapter~2, Theorem~2.7]{BennettSharpley88}.

\begin{theorem} \label{TheoremCharResonance}
	A $\sigma$-finite measure space is resonant if and only if it is either non-atomic or completely atomic with all atoms having equal measure.
\end{theorem}

Besides the non-increasing rearrangement, we will also need the concept of elementary maximal function, sometimes also called the maximal non-increasing rearrangement, which is defined as the Hardy transform of the non-increasing rearrangement.

\begin{definition}
	The \emph{elementary maximal function} $f^{**}$ of $f \in \mathcal{M}$ is defined for $t \in (0, \infty)$ by 
	\begin{equation*}
		f^{**}(t) = \frac{1}{t} \int_0^{t} f^*(s) \: ds.
	\end{equation*} 
\end{definition}

For the basic properties of the elementary maximal function, we refer the reader to \cite[Chapter~2, Proposition~3.2]{BennettSharpley88}. The usefulness of this concept lies in the fact, that unlike the non-increasing rearrangement, the elementary maximal function is subadditive. For proof, see \cite[Chapter~2, Theorem~3.4]{BennettSharpley88}.

\begin{proposition} \label{PropSubAdd**}
	Let $(\mathcal{R}, \mu)$ be resonant and let $f,g \in \mathcal{M}_0(\mathcal{R}, \mu)$. Then 
    \begin{align*}
        (f + g)^{**}(t) &\leq f^{**}(t) + g^{**}(t) &\text{for all } t \in (0, \infty).
    \end{align*}
\end{proposition}

The subadditivity of the elementary maximal function closely related to the very useful Hardy's lemma, which we will also need directly. The proof can be found in \cite[Chapter~2, Proposition~3.6]{BennettSharpley88}.

\begin{lemma} \label{LemmaHardy}
	Let $h_1, h_2 \in \mathcal{M}_+([0, \infty), \lambda)$ satisfy
	\begin{align*}
		\int_0^{t} h_1 \: d\lambda &\leq \int_0^{t} h_2 \: d\lambda &\text{for all } t \in (0, \infty).
	\end{align*}
	Let further $\varphi \in \mathcal{M}_+([0, \infty), \lambda)$ be non-increasing. Then
	\begin{equation*}
		\int_0^{\infty} h_1 \varphi \: d\lambda \leq \int_0^{\infty} h_2 \varphi \: d\lambda.
	\end{equation*}
\end{lemma}

\subsection{Banach function norms and quasinorms} \label{SectionFunctionNormsQuasinorms}

\begin{definition}
	Let $\lVert \cdot \rVert \colon \mathcal{M}(\mathcal{R}, \mu) \rightarrow [0, \infty]$ be a mapping satisfying $\lVert \, \lvert f \rvert \, \rVert = \lVert f \rVert$ for all $f \in \mathcal{M}$. We say that $\lVert \cdot \rVert$ is a \emph{Banach function norm} if its restriction to $\mathcal{M}_+$ satisfies the following axioms:
	\begin{enumerate}[label=\textup{(P\arabic*)}, series=P]
		\item \label{P1} it is a norm, in the sense that it satisfies the following three conditions:
		\begin{enumerate}[ref=(\theenumii)]
			\item \label{P1a} it is absolutely homogeneous, i.e.\ $\forall a \in \mathbb{C} \; \forall f \in \mathcal{M}_+ : \lVert a f \rVert = \lvert a \rvert \lVert f \rVert$,
			\item \label{P1b} it satisfies $\lVert f \rVert = 0 \Leftrightarrow f = 0$  $\mu$-a.e.,
			\item \label{P1c} it is subadditive, i.e.\ $\forall f,g \in \mathcal{M}_+  :  \lVert f+g \rVert \leq \lVert f \rVert + \lVert g \rVert$,
		\end{enumerate}
		\item \label{P2} it has the lattice property, i.e.\ if some $f, g \in \mathcal{M}_+$ satisfy $f \leq g$ $\mu$-a.e., then also $\lVert f \rVert \leq \lVert g \rVert$,
		\item \label{P3} it has the Fatou property, i.e.\ if  some $f_n, f \in \mathcal{M}_+$ satisfy $f_n \uparrow f$ $\mu$-a.e., then also $\lVert f_n \rVert \uparrow \lVert f \rVert $,
		\item \label{P4} $\lVert \chi_E \rVert < \infty$ for all $E \subseteq \mathcal{R}$ satisfying $\mu(E) < \infty$,
		\item \label{P5} for every $E \subseteq \mathcal{R}$ satisfying $\mu(E) < \infty$ there exists some finite constant $C_E$, dependent only on $E$, such that the inequality $ \int_E f \: d\mu \leq C_E \lVert f \rVert $ is true for all $f \in \mathcal{M}_+$.
	\end{enumerate} 
\end{definition}

\begin{definition}
	Let $\lVert \cdot \rVert \colon \mathcal{M}(\mathcal{R}, \mu) \rightarrow [0, \infty]$ be a mapping satisfying $\lVert \, \lvert f \rvert \, \rVert = \lVert f \rVert$ for all $f \in \mathcal{M}$. We say that $\lVert \cdot \rVert$ is a \emph{quasi-Banach function norm} if its restriction to $\mathcal{M}_+$ satisfies the axioms \ref{P2}, \ref{P3}, and \ref{P4} of Banach function norms together with a weaker version of axiom \ref{P1}, namely
	\begin{enumerate}[label=\textup{(Q\arabic*)}]
		\item \label{Q1} it is a quasinorm, in the sense that it satisfies the following three conditions:
		\begin{enumerate}[ref=(\theenumii)]
			\item \label{Q1a} it is absolutely homogeneous, i.e.\ $\forall a \in \mathbb{C} \; \forall f \in \mathcal{M}_+ : \lVert af \rVert = \lvert a \rvert \lVert f \rVert$,
			\item \label{Q1b} it satisfies  $\lVert f \rVert = 0 \Leftrightarrow f = 0$ $\mu$-a.e.,
			\item \label{Q1c} there is a constant $C\geq 1$, called the modulus of concavity of $\lVert \cdot \rVert$, such that it is subadditive up to this constant, i.e.
			\begin{equation*}
				\forall f,g \in \mathcal{M}_+ : \lVert f+g \rVert \leq C(\lVert f \rVert + \lVert g \rVert).
			\end{equation*}
		\end{enumerate}
	\end{enumerate}
\end{definition}

Usually, it is assumed that the modulus of concavity is the smallest constant for which the part \ref{Q1c} of \ref{Q1} holds. We will follow this convention, even though the value will be of little consequence for our results. 

\begin{definition}
	Let $\lVert \cdot \rVert$ be a (quasi-)Banach function norm. We say that $\lVert \cdot \rVert$ is \emph{rearrangement-invariant}, abbreviated r.i., if $\lVert f\rVert = \lVert g \rVert$ whenever $f, g \in \mathcal{M}$ are equimeasurable (in the sense of Definition~\ref{DEM}).
\end{definition}

\begin{definition}
	Let $\lVert \cdot \rVert_X$ be a (quasi-)Banach function norm. We then define the corresponding \emph{(quasi-)Banach function space} $X$ as the set
	\begin{equation*}
		X = \left \{ f \in \mathcal{M};  \; \lVert f \rVert_X < \infty \right \}.
	\end{equation*}
	
	Furthermore, we will say that $X$ is rearrangement-invariant whenever $\lVert \cdot \rVert_X$ is.
\end{definition}

For a detailed treatment of (r.i.)~Banach function spaces we refer the reader to \cite[Chapters~1 and 2]{BennettSharpley88}; for an overview of (r.i.)~quasi-Banach function spaces we recommend \cite{LoristNieraeth23, MusilovaNekvinda24, NekvindaPesa24}, and the references therein. Here we focus exclusively on the properties that are directly related to our work.

Firstly, we will need the result that embeddings of quasi-Banach function spaces are automatically continuous. This result has long been known for Banach function spaces (see e.g.~\cite[Chapter~1, Theorem~1.8]{BennettSharpley88}), but the extension to the wider class of quasi-Banach function spaces has only been obtained quite recently in \cite[Corollary~3.10]{NekvindaPesa24}.

\begin{theorem} \label{TEQBFS}
	Let $\lVert \cdot \rVert_X$ and $\lVert \cdot \rVert_Y$ be quasi-Banach function norms and let $X$ and $Y$ be the corresponding quasi-Banach function spaces. If $X \subseteq Y$ then also $X \hookrightarrow Y$.
\end{theorem}

Secondly, an important property of r.i.~quasi-Banach function spaces over $([0, \infty), \lambda)$ is that the dilation operator is continuous on those spaces, as stated in the following theorem. This is a classical result in the context of r.i.~Banach function spaces which has been recently extended to r.i.~quasi-Banach function spaces in \cite[Section~3.4]{NekvindaPesa24} (for the classical version see e.g.~\cite[Chapter~3, Proposition~5.11]{BennettSharpley88}).

\begin{definition} \label{DDO}
	Let $t \in (0, \infty)$. The dilation operator $D_t$ is defined on $\mathcal{M}([0, \infty), \lambda)$ by the formula
	\begin{equation*}
		D_tf(s) = f(ts),
	\end{equation*}
	where $f \in \mathcal{M}([0, \infty), \lambda)$, $s \in (0, \infty)$.
\end{definition}

\begin{theorem} \label{TDRIS}
	Let $X$ be an r.i.~quasi-Banach function space over $([0, \infty), \lambda)$ and let $t \in (0, \infty)$. Then $D_t \colon X \rightarrow X$ is a continuous operator.
\end{theorem}

It is worth pointing out that the dilation operator satisfies $(D_tf)^* = D_tf^*$ for every function $f \in \mathcal{M}([0, \infty), \lambda)$ (and every $t \in (0, \infty)$), as follows from a simple calculation.

Thirdly, an extremely important result in the theory of r.i.~Banach function spaces is the Luxemburg representation theorem, which allows for a space defined over an arbitrary resonant measure space to be represented by a space defined over $([0,\infty), \lambda)$. Recently, it has been shown that this result extends to the wider class of r.i.~quasi-Banach function spaces, first in \cite[Section~3]{MusilovaNekvinda24} where the representation quasinorm is defined over $([0,\mu(\mathcal{R})), \lambda)$ and then in \cite[Section~3]{PesaRepreACqN} where the original construction was slightly modified for the cases when $\mu(\mathcal{R})<\infty$ in order to always provide a quasinorm over $([0,\infty), \lambda)$).

\begin{theorem} \label{TheoremRepresentation}
	Assume that $(\mathcal{R},\mu)$ is resonant and let $\lVert \cdot \rVert_X$ be an r.i.~quasi-Banach function norm on $\mathcal{M}(\mathcal{R},\mu)$. Then there is an r.i.~quasi-Banach function norm $\lVert \cdot \rVert_{\overline{X}}$ on $\mathcal{M}([0,\infty), \lambda)$ such that for every $f \in \mathcal{M}(\mathcal{R},\mu)$ it holds that $\lVert f \rVert_X=\lVert f^* \rVert_{\overline{X}}$. Furthermore, when $(\mathcal{R},\mu)$ is non-atomic and $\mu(\mathcal{R}) = \infty$, then $\lVert \cdot \rVert_{\overline{X}}$ is uniquely determined.
\end{theorem}

When the space is completely atomic then the representation is not unique. The reason for this is well explained in \cite[Section~3]{MusilovaNekvinda24} but it can be summarised as follow: The original quasinorm only measures the decay of a function, as the functions on atomic spaces do not have blow-ups. Since the functions on $([0,\mu(\mathcal{R})), \lambda)$ do in general have blow-ups, the representation must measure them, but how it does so cannot be uniquely determined by the original quasinorm. Hence, a choice must be made during the construction. (The authors of \cite{MusilovaNekvinda24} chose $L^1$ as the local component, because for this choice their construction, when applied on an r.i.~Banach function norm, yields the same representation functional as the classical construction of Luxemburg, as presented in \cite{Luxemburg67} or \cite[Chapter~2, Theorem~4.10]{BennettSharpley88}).

Similarly, the lack of uniqueness in the case when $\mu(\mathcal{R})< \infty$ is due to the fact that the representation space is considered as over $([0,\infty), \lambda)$ and the argument is almost identical, only with the roles of blow-ups and decay exchanged. Were we to follow the approach of \cite[Section~3]{MusilovaNekvinda24} and consider representation over $([0,\mu(\mathcal{R})), \lambda)$, then it would always be unique for non-atomic spaces.

For the cases when the representation quasinorm is not uniquely determined, it has been observed in \cite{PesaRepreACqN} that some of the properties of the representation correspond nicely to the properties of the original quasinorm only when a proper choice of the representation quasinorm is made. The paper also provides a particular construction for which all the characterisations so far considered are valid. Notably, this construction coincides to that provided in \cite[Section~3]{MusilovaNekvinda24} provided that $\mu(\mathcal{R}) = \infty$ or to the classical construction from \cite{Luxemburg67} provided that $\lVert \cdot \rVert_X$ is an r.i.~Banach function norm. We will therefore always work with a representation quasinorm constructed in this way. However, we will never need to work with the explicit construction, we only need the information that the representation is constructed in such a way that we have the required characterisations of its properties (that is, Propositions~\ref{PropRepreP5} and \ref{PropRepreHLPP} and Theorem~\ref{ThmRepreACqN}). Hence, we choose not to provide the full construction here; we refer the interested reader to the presentation in \cite[Section~3]{PesaRepreACqN}.

\begin{definition}
	Assume that $(\mathcal{R},\mu)$ is resonant and let $\lVert \cdot \rVert_X$ be an r.i.~quasi-Banach function norm on $\mathcal{M}(\mathcal{R},\mu)$. By $\lVert \cdot \rVert_{\overline{X}}$ we always denote the representation quasinorm on $\mathcal{M}([0,\infty), \lambda)$ that has been constructed in \cite[Section~3]{PesaRepreACqN}, which we call the \emph{canonical representation quasinorm}. Similarly, the corresponding r.i.~quasi-Banach function space $\overline{X}$ will be called the \emph{canonical representation space}.
\end{definition}

Given that this construction coincides with the older ones in the relevant cases, we believe the name ``canonical'' to be justified. First of all the useful characterisation which we obtain is the following:

\begin{proposition} \label{PropRepreP5}
	Assume that $(\mathcal{R},\mu)$ is resonant, let $\lVert \cdot \rVert_X$ be an r.i.~quasi-Banach function norm on $\mathcal{M}(\mathcal{R},\mu)$, and let $\lVert \cdot \rVert_{\overline{X}}$ be the canonical representation quasinorm. Then $\lVert \cdot \rVert_{\overline{X}}$ has the property \ref{P5} if and only if $\lVert \cdot \rVert_X$ does.
\end{proposition}

A useful corollary of the representation theorem is the following statement that has been obtained in \cite{MusilovaNekvinda24}. It is worth noting that the result follows from the representation theorem, but the actual representation is not part of the statement.

\begin{corollary} \label{Corollary*<*=>||<||}
	Assume that $(\mathcal{R},\mu)$ is resonant, let $\lVert \cdot \rVert_X$ be an r.i.~quasi-Banach function norm on $\mathcal{M}(\mathcal{R},\mu)$, and let $g_1, g_2 \in \mathcal{M}(\mathcal{R},\mu)$. If $g_1^* \leq g_2^*$ then also $\lVert g_1 \rVert_X \leq \lVert g_2 \rVert_X$.
\end{corollary}

Finally, we will also work with the so called Hardy--Littlewood--P\'{o}lya relation and the related Hardy--Littlewood--P\'{o}lya principle.

\begin{definition} \label{DefHLPR}
	Let $f, g \in \mathcal{M}$. We say that $g$ majorises $f$ with respect to the \emph{Hardy--Littlewood--P\'{o}lya relation}, denoted by $f \prec g$, if $f^{**} \leq g^{**}$, i.e.~if 
	\begin{equation*}
		\int_0^{t} f^* \: d\lambda \leq \int_0^{t} g^* \: d\lambda 
	\end{equation*}
	for all $t \in (0, \infty)$.
\end{definition}

\begin{definition} \label{DHLP}
	Let $\lVert \cdot \rVert$ be an r.i.~quasi-Banach function norm over a resonant measure space. We say that the \emph{Hardy--Littlewood--P\'{o}lya principle} holds for $\lVert \cdot \rVert$ if the estimate $\lVert f \rVert \leq \lVert g \rVert$ is true for any pair of functions $f, g \in \mathcal{M}$ such that $f \prec g$.
\end{definition}

Let us put this property into context:
\begin{enumerate}
	\item The property \ref{P2} of quasi-Banach function norms states that $\lVert f \rVert \leq \lVert g \rVert$ whenever $\lvert f \rvert \leq \lvert g \rvert$ $\mu$-a.e.
	\item The property that the quasinorm is rearrangement-invariant is equivalent to saying that $\lVert f \rVert \leq \lVert g \rVert$ whenever $f^* \leq g^*$ $\lambda$-a.e.~(as discussed above, see Corollary~\ref{Corollary*<*=>||<||}).
	\item The property that the Hardy--Littlewood--P\'{o}lya principle holds for the quasinorm states that $\lVert f \rVert \leq \lVert g \rVert$ whenever $f \prec g$.	
\end{enumerate}
Hence, the three properties each require that the quasinorm in question is monotone with respect to some partial ordering on $\mathcal{M}_0$ and they get progressively stronger as the partial ordering gets weaker.

Furthermore, it has been shown in \cite[Proposition~3.3]{PesaRepreACqN} that the Hardy--Littlewood--P\'{o}lya principle can be characterised using the canonical representation quasinorm:

\begin{proposition} \label{PropRepreHLPP}
	Assume that $(\mathcal{R},\mu)$ is resonant, let $\lVert \cdot \rVert_X$ be an r.i.~quasi-Banach function norm on $\mathcal{M}(\mathcal{R},\mu)$, and let $\lVert \cdot \rVert_{\overline{X}}$ be the canonical representation quasinorm. Then the Hardy--Littlewood--P\'{o}lya principle holds for $\lVert \cdot \rVert_{\overline{X}}$ if and only if it holds for $\lVert \cdot \rVert_X$.
\end{proposition}

Finally, it is well known that the Hardy--Littlewood--P\'{o}lya principle holds for every r.i.~Banach function norm (see e.g.~\cite[Chapter~2, Theorem~4.6]{BennettSharpley88}). In the more general context of r.i.~quasi-Banach function spaces, there is (as far as we are aware) no known characterisation of this property; for some necessary conditions, see \cite[Lemma~2.24 and Theorem~5.9]{Pesa22}.

\subsection{Absolute continuity of the quasinorm} \label{SecACqN}

Absolute continuity of the quasinorm is an important concept in the theory of quasi-Banach function spaces because of its relation to separability (see e.g.~\cite[Section~3.3]{NekvindaPesa24}) and reflexivity (see Section~\ref{Section:AssociateSpaces} below). In our work, it is important because it determines whether the Ces\`aro means converge with respect to the norm or with respect to the weaker topology $\newtopology$ we introduce in Section~\ref{SecRITopo} (as we show in Theorem~\ref{ThmErgodicNormLocalised}).

\begin{definition} \label{DefACqN}
	Let $\lVert \cdot \rVert_X$ be a quasi-Banach function norm and let $X$ be the corresponding quasi-Banach function space. We say that $f \in X$ \emph{has absolutely continuous quasinorm} if it holds that $\lVert f \chi_{E_k} \rVert_X \rightarrow 0$ as $k \to \infty$ whenever $E_k$ is a sequence of $\mu$-measurable subsets of $\mathcal{R}$ such that $\chi_{E_k} \rightarrow 0$ $\mu$-a.e.~as $k \to \infty$. The set of all such functions is denoted $X_a$.
	
	If $X_a = X$, i.e.~every $f \in X$ has absolutely continuous quasinorm, we further say that the space $X$ itself has absolutely continuous quasinorm.
\end{definition}

Since quasi-Banach function norms are monotone with respect to the a.e.~pointwise order, it holds that $X_a$ is an order ideal with respect to the same order. This is easily verified, in a manner virtually identical to that of \cite[Chapter~1, Proof of Theorem~3.8]{BennettSharpley88}.

\begin{proposition} \label{PropXaOrdId}
	The set $X_a$ is an order ideal in $X$, i.e.~it is a closed (with respect to the quasinorm $\lVert \cdot \rVert_X$) linear subspace satisfying that $f \in X_a$ and $\lvert g \rvert \leq \lvert f \rvert$ $\mu$-a.e.~implies $g \in X_a$.
\end{proposition}

We will need the following abstract version of the Lebesgue dominated convergence theorem. Its proof is precisely the same as that of \cite[Chapter~1, Proposition~3.6]{BennettSharpley88}.
\begin{proposition} \label{PropDomConv}
	Let $\lVert \cdot \rVert_X$ be a quasi-Banach function norm and let $X$ be the corresponding quasi-Banach function space. Fix $f \in X$. Then the following statements are equivalent:
	\begin{enumerate}
		\item $f \in X_a$.
		\item It holds for every sequence $g_n \in \mathcal{M}$ and every function $g \in \mathcal{M}$ such that $g_n \to g$ $\mu$-a.e.~and $\lvert g_n \rvert \leq \lvert f \rvert$ that all $g_n, g \in X$ and $g_n \to g$ in $X$. 
	\end{enumerate}
\end{proposition}

Proposition~\ref{PropDomConv} has a special case that is worth pointing out.

\begin{corollary}
	Let $\lVert \cdot \rVert_X$ be a quasi-Banach function norm and let $X$ be the corresponding quasi-Banach function space. Assume that $f \in X_a$, $f \geq 0$, and that we have some sequence of functions $f_n \in \mathcal{M}_0$ such that $0 \leq f_n \uparrow f$ $\mu$-a.e. Then $f_n \to f$ in $X$.
\end{corollary}

Let us now move to the setting of r.i.~quasi-Banach function norms. It has been shown in \cite[Theorem~4.2]{PesaRepreACqN} that absolute continuity of the quasinorm can be characterised using the canonical representation quasinorm.

\begin{theorem} \label{ThmRepreACqN}
	Assume that $(\mathcal{R},\mu)$ is resonant, let $\lVert \cdot \rVert_X$ be an r.i.~quasi-Banach function norm on $\mathcal{M}(\mathcal{R},\mu)$, let $\lVert \cdot \rVert_{\overline{X}}$ be the canonical representation quasinorm, and let $X$ and $\overline{X}$ be the corresponding r.i.~quasi-Banach function spaces. Then $f \in X_a$ if and only if $f^* \in \left( \overline{X} \right)_a$.
\end{theorem}

Since r.i.~quasi-Banach function norms are monotone with respect to the ordering or non-increasing rearrangements and may also be monotone with respect to the Hardy--Littlewood--P\'{o}lya relation, it is natural to ask whether the set $X_a$ is in the relevant cases still an order ideal. The answer is positive, as has been shown in \cite{PesaRepreACqN}, and we present it in a precise form as the next theorem. It is worth noting, that while the first part (originally \cite[Corollary~4.3]{PesaRepreACqN}) is a direct consequence of Theorem~\ref{ThmRepreACqN}, the second part (originally \cite[Theorem~4.4]{PesaRepreACqN}) is rather non-trivial.

\begin{theorem} \label{ThmInheritanceACqN}
	Assume that $(\mathcal{R},\mu)$ is resonant, let $\lVert \cdot \rVert_X$ be an r.i.~quasi-Banach function norm on $\mathcal{M}(\mathcal{R},\mu)$, let $X$ be the corresponding r.i.~quasi-Banach function space, and let $f \in X_a$. Then:
	\begin{enumerate}
		\item If $g \in \mathcal{M}(\mathcal{R}, \mu)$ satisfies $g^* \leq f^*$ then $g \in X_a$.
		\item Assume further that the Hardy--Littlewood--P\'{o}lya principle holds for $\lVert \cdot \rVert_{X}$. Then we have for every function $h \in \mathcal{M}(\mathcal{R}, \mu)$ which satisfies $h \prec f$ that $h \in X_a$.
	\end{enumerate}
\end{theorem}

Another useful fact is the following alternative concerning the subspace $X_a$.

\begin{proposition} \label{PropXaAlternative}
    Let $\lVert \cdot \rVert_X$ be an r.i.~quasi-Banach function norm and let $X$ be the corresponding r.i.~quasi-Banach function space. Then we either have that $X_a = \{0 \}$ or that the simple functions are both contained and dense in $X_a$. 
\end{proposition}

\begin{proof}
	That we either have that $X_a = \{0 \}$ or that it contains the simple functions has been proven in \cite[Theorem~4.15]{MusilovaNekvinda24} (the same statement for r.i.~Banach functions spaces have been known earlies, see e.g.~\cite[Chapter~2, Theorems~5.4 and 5.5]{BennettSharpley88}). The density then follows from Proposition~\ref{PropDomConv} once one recalls that any function $f \in X$ is $\mu$-a.e.~the pointwise limit of some sequence of simple functions $s_n$ such that $\lvert s_n \rvert \leq \lvert f \rvert \in X_a$.
\end{proof}

Let us now introduce a property that we call absolute continuity of rearrangement of a function, as well as the related \ACR property that will play a crucial role in most of our main results.

\begin{definition} \label{DefMACR}
	We say, that $f \in \mathcal{M}_0$ has \emph{absolutely continuous rearrangement} if it holds that
	\begin{equation*} \label{PropACR:E1}
		\lim_{t \to \infty} f^*(t) = 0.
	\end{equation*}
	The set of all such functions is denoted $\mathcal{M}_{(ACR)}$.
\end{definition}

\begin{definition} \label{DefACR}
	Let $\lVert \cdot \rVert_X$ be an r.i.~quasi-Banach function norm and let $X$ be the corresponding r.i.~quasi-Banach function space. Then we say that $X$ has the \emph{\ACR property} if $X \subseteq \mathcal{M}_{(ACR)}$.
\end{definition}

The following result relates absolute continuity of the quasinorm with the \ACR property. For proof, see \cite[Proposition~4.1]{PesaRepreACqN}.

\begin{proposition} \label{PropACNimpliesACR}
	Let $\lVert \cdot \rVert_X$ be an r.i.~quasi-Banach function norm and let $X$ be the corresponding r.i.~quasi-Banach function space. If $f \in X$ has an absolutely continuous quasinorm (i.e.~$f \in X_a$), then it satisfies
	\begin{equation*}
		\lim_{t \to \infty} f^*(t) = 0.
	\end{equation*}

    Consequently, if $X$ has absolutely continuous quasinorm, then it has the \ACR property.
\end{proposition}

Let us recall what we already mentioned in the introduction, i.e.~that the \ACR property is very weak: it is trivially satisfied for every r.i.~quasi-Banach function space when the underlying measure space is of finite measure, and even in the infinite-measure case there is in a sense a single counterexample, the spaces that are equal to $L^{\infty}$ ``near infinity''; see \cite[Theorem~4.16]{MusilovaNekvinda24} for the precise formulation of this statement and its proof. 

Finally, the following observation will be of use. Its proof is a simple exercise.
\begin{lemma} \label{LemDilPresXa}
	Let $X$ be an r.i.~quasi-Banach function space over $([0, \infty), \lambda)$, let $E \subseteq [0, \infty)$, and let $t \in (0, \infty)$. Then
	\begin{equation*}
		D_t \chi_E = \chi_{\{s \in [0, \infty); \; ts \in E \}},
	\end{equation*} 
	and consequently $D_t$ maps $X_a$ into itself, i.e.~$D_t(X_a) \subseteq X_a$.
\end{lemma}

\subsection{Associate spaces} \label{Section:AssociateSpaces}
Another important concept is that of an associate space. The detailed study of associate spaces of Banach function spaces can be found in \cite[Chapter 1, Sections 2, 3 and 4]{BennettSharpley88}. We will approach the issue in a slightly more general way. The definition of an associate space requires no assumptions on the functional that defines the original space.

\begin{definition} \label{DAS}
	Let $\lVert \cdot \rVert_X \colon  \mathcal{M} \to [0, \infty]$ be some non-negative functional and put
	\begin{equation*}
		X = \{ f \in \mathcal{M}; \; \lVert f \rVert_X < \infty \}.
	\end{equation*} 
	Then the functional $\lVert \cdot \rVert_{X'}$ defined for $f \in \mathcal{M}$ by 
	\begin{equation*}
		\lVert f \rVert_{X'} = \sup_{g \in X} \frac{1}{\lVert g \rVert_X} \int_\mathcal{R} \lvert f g \rvert \: d\mu, \label{DAS1}
	\end{equation*}
	where we interpret $\frac{0}{0} = 0$ and $\frac{a}{0} = \infty$ for any $a>0$, will be called the associate functional of $\lVert \cdot \rVert_X$ while the set
	\begin{equation*}
		X' = \left \{ f \in \mathcal{M}; \; \lVert f \rVert_{X'} < \infty \right \}
	\end{equation*}
	will be called the associate space of $X$.
\end{definition}

As suggested by the notation, we will be interested mainly in the case when  $\lVert \cdot \rVert_X$ is at least a quasinorm, but we wanted to indicate that such assumption is not necessary for the definition. In fact, it is not even required for the following result, which is the Hölder inequality for associate spaces.

\begin{theorem} \label{THAS}
	Let $\lVert \cdot \rVert_X \colon \mathcal{M} \to [0, \infty]$ be some non-negative functional and denote by $\lVert \cdot \rVert_{X'}$ its associate functional. Then it holds for all $f,g \in \mathcal{M}$ that
	\begin{equation*}
		\int_\mathcal{R} \lvert f g \rvert \: d\mu \leq \lVert g \rVert_X \lVert f \rVert_{X'}
	\end{equation*}
	provided that we interpret $0 \cdot \infty = -\infty \cdot \infty = \infty$ on the right-hand side.
\end{theorem}

The convention concerning the products at the end of this theorem is necessary precisely because we put no restrictions on $\lVert \cdot \rVert_X$ and thus there occur some pathological cases which need to be taken care of. Specifically, $0 \cdot \infty = \infty$ is required because we allow $\lVert g \rVert_X = 0$ even for non-zero $g$ while $-\infty \cdot \infty = \infty$ is required because Definition~\ref{DAS} allows $X = \emptyset$ which implies $\lVert f \rVert_{X'} = \sup \emptyset = -\infty$.

In order for the associate functional to be well behaved some assumptions on $\lVert \cdot \rVert_X$ are needed. The following result, due to Gogatishvili and Soudsk{\'y} in \cite{GogatishviliSoudsky14}, provides a sufficient condition for the associate functional to be a Banach function norm. It also provides a characterisation of functionals that are equivalent to Banach function norms. We note that the final stronger statement for Banach function norms is a classical result of Lorentz and Luxemburg, proof of which can be found for example in \cite[Theorem~4]{Luxemburg55} or \cite[Chapter~1, Theorem~2.7]{BennettSharpley88}.

\begin{theorem} \label{TFA}
	Let $\lVert \cdot \rVert_X \colon \mathcal{M} \to [0, \infty]$ be a functional that satisfies the axioms \ref{P4} and \ref{P5} from the definition of Banach function spaces and which also satisfies for all functions $f \in \mathcal{M}$ that $\lVert f \rVert_X$ = $\lVert \, \lvert f \rvert \, \rVert_X$. Then the functional $\lVert \cdot \rVert_{X'}$ is a Banach function norm. In addition, $\lVert \cdot \rVert_X$ is equivalent to a Banach function norm if and only if $\lVert \cdot \rVert_X \approx \lVert \cdot \rVert_{X''}$, where $\lVert \cdot \rVert_{X''}$ denotes the associate functional of $\lVert \cdot \rVert_{X'}$.
	
	Furthermore, when $\lVert \cdot \rVert_X$ is a Banach function norm then in fact $\lVert \cdot \rVert_X = \lVert \cdot \rVert_{X''}$.
\end{theorem}

As a special case, we get that the associate functional of any quasi-Banach function space that also satisfies the axiom \ref{P5} is a Banach function norm. This has been observed earlier in \cite[Remark~2.3.(iii)]{EdmundsKerman00}.

Let us point out that even in the case when $\lVert \cdot \rVert_X$, satisfying the assumptions of Theorem~\ref{TFA}, is not equivalent to any Banach function norm we still have the embedding of the space into its second associate space, as formalised in the following statement. The proof is an easy exercise.

\begin{proposition} \label{PESSAS}
	Let $\lVert \cdot \rVert_X$ satisfy the assumptions of Theorem~\ref{TFA}. Then it holds for all $f \in \mathcal{M}$ that
	\begin{equation*}
		\lVert f \rVert_{X''} \leq  \lVert f \rVert_X,
	\end{equation*}
	where $\lVert \cdot \rVert_{X''}$ denotes the associate functional of $\lVert \cdot \rVert_{X'}$. Consequently, the respective spaces satisfy $X \hookrightarrow X''$.
\end{proposition}

Next, we include two simple but useful observations. The first one describes how embeddings of spaces translate to embeddings of the corresponding associate spaces.  We formulate it in its full generality to showcase that it does not require any assumptions on the functionals, but we are of course mostly interested in the case when they are quasi-Banach function norms. The proof is an easy modification of \cite[Chapter~2, Proposition~2.10]{BennettSharpley88}.

\begin{proposition} \label{PEASG}
	Let $\lVert \cdot \rVert_X \colon  \mathcal{M} \to [0, \infty]$ and $\lVert \cdot \rVert_Y \colon  \mathcal{M} \to [0, \infty]$ be two non-negative functionals satisfying that there is a constant $C>0$ such that it holds for all $f \in \mathcal{M}$ that
	\begin{equation*}
		\lVert f \rVert_X \leq C \lVert f \rVert_Y.
	\end{equation*}
	Then the associate functionals $\lVert \cdot \rVert_{X'}$ and $\lVert \cdot \rVert_{Y'}$ satisfy, with the same constant $C$,
	\begin{equation*}
		\lVert f \rVert_{Y'} \leq C \lVert f \rVert_{X'}
	\end{equation*}
	for all $f \in \mathcal{M}$.
\end{proposition}


The second statement shows, that in the case when the underlying measure space is resonant, the associate functional of an r.i.~quasi-Banach function norm can be expressed in terms of non-increasing rearrangement. The proof is the same as in \cite[Chapter~2, Proposition~4.2]{BennettSharpley88}.

\begin{proposition} \label{PAS}
	Let $\lVert \cdot \rVert_X$ be an r.i.~quasi-Banach function norm over a resonant measure space. Then its associate functional $\lVert \cdot \rVert_{X'}$ satisfies
	\begin{equation*}
		\lVert f \rVert_{X'} = \sup_{\lVert g \rVert_X \leq 1} \int_0^{\infty} f^* g^* \: d\lambda.
	\end{equation*}
\end{proposition}

An obvious consequence of Proposition~\ref{PAS} is that an associate space of an r.i.~quasi-Banach function space (over a resonant measure space) is also rearrangement-invariant.

We will also need the following characterisation that follows immediately by combining \cite[Theorem~3.11]{NekvindaPesa24} with the Hardy--Littlewood inequality (Theorem~\ref{THLI}) and Proposition~\ref{PAS}.

\begin{proposition} \label{PropLandauRes}
	Let $\lVert \cdot \rVert_X$ be an r.i.~quasi-Banach function norm over a resonant measure space, let $X$ be the corresponding r.i.~quasi-Banach function space and $X'$ the respective associate space. Then $f \in \mathcal{M}$ belongs to $X'$ if and only if it satisfies
	\begin{equation*}
		\int_0^{\infty} f^* g^* \: d\lambda < \infty
	\end{equation*}
	for all $g \in X$.
\end{proposition}

Another fact that is important for our purposes is the following result from \cite[Proposition~3.4]{PesaRepreACqN} that shows that the concept of associate spaces interacts reasonably with the canonical representation functional.

\begin{proposition} \label{PropRepreAS}
	Assume that $(\mathcal{R},\mu)$ is resonant, let $\lVert \cdot \rVert_X$ be an r.i.~quasi-Banach function norm on $\mathcal{M}(\mathcal{R},\mu)$, and let $X$ be the corresponding r.i.~quasi-Banach function space. Let further $\lVert \cdot \rVert_{\overline{X}}$ be the canonical representation quasinorm and let $\overline{X}$ be the corresponding r.i.~quasi-Banach function space. Finally, denote by $\lVert \cdot \rVert_{X'}$ and $\lVert \cdot \rVert_{\left( \overline{X} \right)'}$ the respective associate norms (on $\mathcal{M}(\mathcal{R},\mu)$ and $\mathcal{M}([0,\infty), \lambda)$, respectively). Then it holds for every $f \in \mathcal{M}(\mathcal{R},\mu)$ that
	\begin{equation*} \label{PropRepreAS:E1}
		\lVert f \rVert_{X'} = \lVert f^* \rVert_{\left( \overline{X} \right)'}.
	\end{equation*}
\end{proposition}

Note that this does not imply anything about the functional $\lVert \cdot \rVert_{\left( \overline{X} \right)'}$ being the canonical representation of $\lVert \cdot \rVert_{X'}$. Indeed, it is generally a different representation functional; see the commentary after \cite[Proposition~3.4]{PesaRepreACqN} for details.

Let us now consider the relationship between the associate space $X'$ and the classical topological dual $X^*$ (of a given quasi-Banach function space $X$ equipped with the quasi-Banach function norm $\lVert \cdot \rVert_X$), which will play an important role in our work. It is quite clear that $X' \hookrightarrow X^*$, in the sense that each $\varphi \in X'$ naturally induces a continuous linear functional on $X$ via the formula
\begin{align} \label{EqFunctionToFuntional}
	f &\to \int_{\mathcal{R}} \varphi f \: d\mu &f \in X,
\end{align}
with the norm of this functional in $X^*$ being equal to $\lVert \varphi \rVert_{X'}$ and with distinct functions in $X'$ inducing distinct functionals in $X^*$ (since $X$ contains the simple functions).

The converse embedding does not hold in general and its validity is tightly tied to the absolute continuity of the quasinorm. We present this relationship in the next theorem for the wider context of quasi-Banach function spaces, and note that it is more complicated than in the classical Banach function space setting. The result is not exactly classical, but follows by a straightforward modification of known methods. We will discuss this in more detail after the statement of the theorem.

\begin{theorem} \label{ThmACNDual}
	Let $\lVert \cdot \rVert_X$ be a quasi-Banach function norm and let $X$ be the corresponding quasi-Banach function space. Denote by $X'$ and $X^*$, respectively, the corresponding associate space and the topological dual space of $X$. When $X$ has absolutely continuous quasinorm, then $X' = X^*$, i.e.~every linear functional $\alpha$ on $X$ that is continuous (i.e.~bounded) with respect to the quasinorm $\lVert \cdot \rVert_X$ can be uniquely represented by some $\varphi \in X'$ via the formula
	\begin{align*}
		\alpha(f) &= \int_{\mathcal{R}} \varphi f \: d\mu &f \in X
	\end{align*}
	and we have $\lVert \alpha \rVert_{X^*} = \lVert \varphi \rVert_{X'}$.
	
	Furthermore, when $X$ is a Banach function space, then the converse statement also holds, i.e.~$X' = X^*$ (in the same sense as above) implies that $X$ has absolutely continuous norm.
	
	Finally, consider the set $X_a$ equipped with the appropriate restriction of the quasinorm $\lVert \cdot \rVert_X$, which is itself a quasi-Banach space. When $X_a$ contains the simple functions, then we have $X_a^* = X'$ in the same sense as above.
\end{theorem}

The proof of sufficiency (for both the characterisations of $X^*$ and $X_a^*$) is basically identical to \cite[Chapter~1, Theorem~4.1]{BennettSharpley88} (see also \cite[Proposition~3.15]{LoristNieraeth23}, we further present a similar argument in Theorem~\ref{ThmACRDual}). The classical proof of necessity (again, \cite[Chapter~1, Theorem~4.1]{BennettSharpley88}) uses $X = X''$, i.e.~it works precisely for Banach function spaces (see Theorem~\ref{TFA}). The necessity fails for quasi-Banach function spaces in general (consider $L^{p,\infty}$, $p \in (0,1)$); as far as we know, no full characterisation is available in this wider context.

Furthermore, we note that the claim that $X_a$ equipped with the appropriate restriction of the quasinorm $\lVert \cdot \rVert_X$ is a quasi-Banach space follows from it being closed, see Proposition~\ref{PropXaOrdId}. We also stress that $X_a$ is in general not a quasi-Banach function space.

For Banach function spaces, it follows that reflexivity can be characterised via the absolute continuity of norm. However, this conclusion is not immediate; see \cite[Chapter~1, Corollary~4.4]{BennettSharpley88} for details.

\begin{corollary} \label{CorReflexBFS}
	A Banach function space $X$ is reflexive if and only if both itself and the associate space $X'$ have absolutely continuous norm.
\end{corollary}

\subsection{Classical locally convex topologies on quasi-Banach function spaces} \label{SecTopologies}

The default topology on a quasi-Banach function space is the one induced by the quasinorm that defines it. Hence, if we make a statement about convergence, continuity or boundedness without specifying topology, then such statement is always to be interpreted as referring to this topology.

\begin{definition}
    Let $\lVert \cdot \rVert_X$ be a quasi-Banach function norm and let $X$ be the corresponding quasi-Banach function space. Then the topology on $X$ induced by $\lVert \cdot \rVert_X$ will, when necessary, also be denoted by $\lVert \cdot \rVert_X$. Hence, the space $X$ equipped with this topology will be denoted $(X, \lVert \cdot \rVert_X)$.
\end{definition}

In addition to the default topology on a quasi-Banach function space, there are several additional natural locally convex topologies. When $\lVert \cdot \rVert_X$ is a quasi-Banach function norm satisfying \ref{P5} then the corresponding associate space $X'$ is an r.i.~Banach function space, specifically it contains $\chi_{E}$ for every $E \subseteq \mathcal{R}$ with $\mu(E) < \infty$. Since each function $\varphi \in X'$ naturally induces a continuous linear functional on $X$ via the formula \eqref{EqFunctionToFuntional}, it is clear that the topological dual $X^*$ is in this case also non-trivial and separates points. Hence, it makes good sense to consider the weak topology on $X$ and the weak$^*$-topology on $X^*$, as well as the weak topology on $X$ induced by $X'$.

\begin{definition}
	Let $\lVert \cdot \rVert_X$ be a quasi-Banach function norm satisfying \ref{P5} and let $X$ be the corresponding quasi-Banach function space. Denote by $X'$ and $X^*$, respectively, the corresponding associate space and the topological dual space of $X$.
	\begin{enumerate}
		\item The classical weak topology on $X$ (i.e.~the one induced by $X^*$) will be denoted $w$.
		\item The weak topology on $X^*$ induced by $X$ will be denoted $w^*$.
		\item The weak topology on $X$ induced by $X'$ will be denoted $w'$.
	\end{enumerate}
\end{definition}

Since the three topologies introduced above are abstract weak topologies, alternative notation is $w=\sigma(X, X^*)$, $w^* = \sigma(X^*, X)$, and $w' = \sigma(X, X')$. Additionally, we may also consider the respective topologies on the dual and associate spaces, e.g.~$(X^*, w) = (X^*, \sigma(X^*, X^{**}))$ or $(X', w') = (X', \sigma(X', X''))$. For these locally convex topologies and the respective topological duals we have $(X, w)^* = X^* (= (X, \lVert \cdot \rVert_X)^*)$, $(X^*, w^*)^* = X$, $(X, w')^* = X'$, see e.g.~\cite[Chapter~23]{MeiseVogt97}.

For the spaces $X$ that are also r.i., we introduce another locally convex topology in Section~\ref{SecRITopo} which, for infinite-dimensional $X$, in general lies strictly between $w'$ and $\lVert \cdot \rVert_X$. Under the mild additional assumption that $X$ has the \ACR property, it turns out that the topological dual space of $X$, when equipped with this new topology, equals its associate space $X'$ (see Theorem~\ref{ThmACRDual}).

We will also need to work with the space $\mathcal{M}_0$ equipped with the topology of convergence in measure on sets of finite measure.

\begin{definition}
    We shall denote by $(\mathcal{M}_0, \mu_{\textup{loc}})$ the space $\mathcal{M}_0$ of all (equivalence classes of) complex valued $\mu$-a.e.~finite functions equipped with the topology of convergence in measure on sets of finite measure.
\end{definition}

The following proposition is well known, see e.g.~\cite[Theorem~2.30]{Folland1999}.

\begin{proposition} \label{PropM0CompMetr}
    The space $(\mathcal{M}_0, \mu_{\textup{loc}})$ is a completely metrisable topological vector space.
\end{proposition}

It is important for our purposes that every quasi-Banach function space equipped with the default quasinormed topology is continuously embedded into $(\mathcal{M}_0, \mu_{\textup{loc}})$. This is a well known result for Banach function spaces (see e.g.~\cite[Chapter~1, Theorem~1.4]{BennettSharpley88}) which has recently been expanded to the wider class in \cite[Theorem~3.4]{NekvindaPesa24}.

\begin{theorem} \label{ThmEmbM0}
	Let $X$ be an quasi-Banach function space. Then $(X, \lVert \cdot \rVert_X) \hookrightarrow (\mathcal{M}_0, \mu_{\textup{loc}})$.
\end{theorem}

Let us stress that this result fails in general when $X$ is equipped with the $w$ or $w'$ topologies, as can be seen by considering the Hilbert space $L^2$ over $([0,1], \lambda)$ for which there is an orthonormal basis consisting of functions $e_n$ such that $\lvert e_n \rvert = 1$ $\lambda$-a.e.~for every $n$.

Let us conclude this section with a useful result characterising $\lVert \cdot \rVert_X$-boundedness in Banach function spaces; the proof can be found in \cite[Chapter~1, Lemma~5.1]{BennettSharpley88}.

\begin{proposition} \label{PropEquivBoundedness}
	Let $\lVert \cdot \rVert_X$ be an r.i.~Banach function norm, let $X$ be the corresponding r.i.~Banach function space. Then a set $A \subseteq X$ is $\lVert \cdot \rVert_X$-bounded if and only if it is $w'$-bounded.
\end{proposition}

\subsection{Linear operators on quasi-Banach function spaces}

We will work exclusively with linear operators whose domain and range reside in the same space.

\begin{definition}
    Let $X$ be a quasi-Banach function space. We say that a linear operator $T$ defined on a subset of $\mathcal{M}$ and with values in $\mathcal{M}$ is an \emph{operator on $X$} if it is well-defined on $X$ and we have for every $f \in X$ that $Tf \in X$.
\end{definition}

When we consider several topologies on a quasi-Banach function space $X$ it will always be clear from the context to which of the topologies we refer when we consider properties like continuity, power-boundedness, or mean ergodicity of an operator $T$ on $X$ (recall that power-boundedness means that the set of the iterates $\{ T^n; \; n \in \mathbb{N} \}$ is equicontinuous).

We will work extensively with the so-called associate operators. It is convenient for our purposes to work with a somewhat more restricted approach than is usuall; we will discuss the differences and our justifications for this approach in Remark~\ref{RemAssociateOperatorOnAssociateSpace} at the end of this section.

\begin{definition} \label{DefAssocOperator}
	Let $\lVert \cdot \rVert_X$ be a quasi-Banach function norm satisfying \ref{P5}, $X$ be the corresponding quasi-Banach function space, and let $X'$ be the corresponding associate space. Let $T$ be an operator defined on $X$ and $T'$ be an operator defined on $X'$ (both with values in $\mathcal{M}$). We say that $T$ and $T'$ are mutually associate if it holds for every $f \in X$ and every $\varphi \in X'$ that
	\begin{equation*}
		\int_{\mathcal{R}} T(f) \varphi \: d\mu = \int_{\mathcal{R}} f T'(\varphi) \: d\mu.
	\end{equation*}
\end{definition}

\begin{remark}
    Let $T$ be an operator on a quasi-Banach function space $X$. Let $X^+$ be the algebraic dual space of $X$ and $T^+$ the algebraic transpose of $T$. If $T^+(X')\subset X'$, clearly $T$ and $T':=T^+|_{X'}$ are mutually associate and we call $T'$ the associate operator of $T$.
\end{remark}

The following proposition is a special case of \cite[Lemma 23.28]{MeiseVogt97}.

\begin{proposition} \label{PropW'T'qBFS}
	Let $\lVert \cdot \rVert_X$ be a quasi-Banach function norm satisfying \ref{P5}, $X$ be the corresponding quasi-Banach function space, and let $X'$ be the corresponding associate space. Let $T$ be an operator on $X$. Then $T \colon (X, w') \to (X, w')$ is continuous if and only if there exists the associate operator $T'$ on $X'$.
\end{proposition}

	

The associate operator $T'$ on $X'$ is obviously continuous with respect to the topology induced by $X$ (i.e.~$\sigma(X', X)$); however, this topology may in general be distinct from the $w'$ topology (which is induced by $X''$, i.e.~$\sigma(X', X'')$). However, when $X$ is a Banach function space, we have $X'' = X$ (see Theorem~\ref{TFA}) and we thus obtain the following stronger statements. We include the proofs for the sake of completeness.

\begin{proposition} \label{PropT'toNormBFS}
	Let $\lVert \cdot \rVert_X$ be an r.i.~Banach function norm, let $X$ be the corresponding r.i.~Banach function space, and let $X'$ be the corresponding associate space. Let $T$ be an operator defined on $X$ such that the associate operator $T'$ exists as an operator on $X'$. Then $T \colon X \to X$ is continuous, $T' \colon X' \to X'$ is continuous, and $\lVert T \rVert_{X \to X} = \lVert T' \rVert_{X' \to X'}$. 
\end{proposition}

\begin{proof}
	Lef $\varphi \in X'$. Then we have for every $f \in X$ such that $\lVert f \rVert_X \leq 1$ that
	\begin{equation*}
		\left \lvert \int_{\mathcal{R}} \varphi T(f) \: d\mu \right \rvert = \left \lvert \int_{\mathcal{R}} T'(\varphi) f \: d\mu \right \rvert \leq \lVert T'(\varphi) \rVert_{X'} < \infty.
	\end{equation*}
	Whence, $T(B_X)$, the image of the unit ball of $X$, is $w'$-bounded. Proposition~\ref{PropEquivBoundedness} thus yields that it is also $\lVert \cdot \rVert_X$-bounded, i.e.~that $T \colon X \to X$ is continuous. It follows that we have for every $\varphi \in X'$ that
	\begin{equation*}
		\lVert T'(\varphi) \rVert_{X'} = \sup_{\lVert f \rVert_X \leq 1} \left \lvert \int_{\mathcal{R}} T'(\varphi) f \: d\mu \right \rvert = \sup_{\lVert f \rVert_X \leq 1} \left \lvert \int_{\mathcal{R}} \varphi T(f) \: d\mu \right \rvert \leq \lVert \varphi \rVert_{X'} \lVert T \rVert_{X \to X}.
	\end{equation*}
	Whence, $T' \colon X' \to X'$ is also continuous and we have $\lVert T \rVert_{X \to X} \geq \lVert T' \rVert_{X' \to X'}$. The converse inequality is shown analogously.
\end{proof}

\begin{proposition} \label{PropW'T'BFS}
	Let $\lVert \cdot \rVert_X$ be a Banach function norm, $X$ be the corresponding Banach function space, and let $X'$ be the corresponding associate space. Let $T$ be an operator defined on $X$. Then $T \colon (X, w') \to (X, w')$ is continuous if and only if the associate operator $T' \colon (X', w') \to (X', w')$ is continuous.
\end{proposition}

\begin{proof}
	As for the sufficiency, it follows from Proposition~\ref{PropT'toNormBFS} that $T$ is an operator on $X$ and thus the rest of the statement follows from Proposition~\ref{PropW'T'qBFS}.
	
	As for the necessity, the only claim not covered by Proposition~\ref{PropW'T'qBFS} is the continuity of $T' \colon  (X', w') \to (X', w')$, which is, however, proved in exactly the same way as the continuity of $T \colon (X, w') \to (X, w')$ since $X'' = X$, i.e.~the $w'$ topology on $X'$ is in fact induced by $X$ (i.e.~it is equal to $\sigma(X', X)$).
\end{proof}

\begin{remark} \label{RemAssociateOperatorOnAssociateSpace}
    Associate operators are usually defined for $T$ and $T'$ on $\mathcal{M}_0$, our approach is somewhat more restricted. However, it is well compatible, since every quasi-Banach function space contains the simple functions. Indeed, whenever $T$ is defined on some subset of $\mathcal{M}_0$ containing some Banach function space and $T'(f)$ is defined for some function $f$ via Definition~\ref{DefAssocOperator}, i.e.~through some pair of quasi-Banach function spaces $X$ and $X'$, then $T'(f)$ is uniquely determined $\mu$-a.e.~by $T$ and $f$ only, i.e.~it does not depend on the spaces $X$ and $X'$. We have opted for this approach because it is both more flexible and very well suited for our purposes, as we will always consider the associate operator in situations when $X$ and $X'$ are given.
\end{remark}

\subsection{Sublinear operators quasi-Banach function spaces} \label{SecSublinearOps}

In Section~\ref{SectionPointwise} we will also need to work with sublinear operators.

\begin{definition}
    An operator $S \colon \mathcal{M} \to \mathcal{M}$ is sublinear, if it satisfies
    \begin{align*}
        \lvert S(f+g) \rvert &\leq \lvert Sf \rvert + \lvert Sg \rvert & \mu\text{-a.e.~for every } f,g \in \mathcal{M}, \\
        \lvert S(af) \rvert &= \lvert a \rvert \, \lvert Sf \rvert &\mu\text{-a.e.~for every } f \in \mathcal{M} \text{ and } a \in \mathbb{C}.
    \end{align*}

    If $A, B \subseteq \mathcal{M}$ are some Banach spaces, then we say that a sublinear operator $S\colon A \to B$ is bounded provided that
    \begin{align*}
        \lVert Sf \rVert_B &\lesssim \lVert f \rVert_A &\text{for every } f \in A.
    \end{align*}
\end{definition}

We recall that a continuous sublinear operator $T \colon A \to B$ is always bounded, but the converse implication does not hold in general, as boundedness only implies continuity at $0$.

\subsection{Fundamental function and Marcinkiewicz endpoint spaces}\label{Sec:FundamentalFunction} 

In this section, we give a very brief overview of the concepts of fundamental function and Marcinkiewicz endpoint spaces (both strong and weak). We chose to omit the full presentation as it is of considerable length and the matter is not too deeply related to the topic at hand; deep understanding is definitely not required in order to read the paper. For this reason, we also restrict ourselves for the most part to the case when the underlying measure space is $([0, \infty), \lambda)$; only at the very end will we mention how to extend some of the concepts to the general case.

Let us start by defining the fundamental function of an r.i.~quasi-Banach function space.

\begin{definition}
    Let $\lVert \cdot \rVert_X$ be an r.i.~quasi-Banach function norm and $X$ the corresponding r.i.~quasi-Banach function space. We define $\varphi_X$, \emph{the fundamental function of $X$}, by
    \begin{align*}
        \varphi_X(t) &= \lVert \chi_{E_t} \rVert_X &\text{for } t \in [0, \infty),
    \end{align*}
    where $E_t$ is an arbitrary subset of $[0, \infty)$ satisfying $\lambda(E_t) = t$.
\end{definition}

The properties of fundamental functions of r.i.~quasi-Banach function spaces have been studied recently in \cite[Section~4]{MusilovaNekvinda24}, while the same topic for r.i.~Banach function spaces is classical, see e.g.~\cite[Chapter~2, Section~5]{BennettSharpley88}. For our purposes, we shall be content with stating that the fundamental function of an arbitrary r.i.~quasi-Banach function space is \emph{admissible} (see \cite[Definition~4.2]{MusilovaNekvinda24}) while that of an r.i.~Banach function space is \emph{quasiconcave} (see e.g.~\cite[Definition~5.6]{BennettSharpley88}).

On the other hand, let us consider the following functionals:
\begin{definition}
    Let $\Phi\colon [0, \infty) \to [0, \infty)$ be a function. We define the functionals $\lVert \cdot \rVert_{m_{\Phi}}$ and $\lVert \cdot \rVert_{M_{\Phi}}$ for arbitrary $f \in \mathcal{M}([0, \infty), \lambda)$ by
    \begin{align*}
        \lVert f \rVert_{m_{\Phi}} &= \sup_{t \in [0, \infty)} \Phi(t) f^*(t), \\
        \lVert f \rVert_{M_{\Phi}} &= \sup_{t \in [0, \infty)} \Phi(t) f^{**}(t).
    \end{align*}
\end{definition}

As it turns out, if the function $\Phi$ is admissible, then $\lVert \cdot \rVert_{m_{\Phi}}$ is an r.i.~quasi-Banach function norm satisfying $\varphi_{m_{\Phi}} \approx \Phi$ and every r.i.~quasi-Banach function space $X$ such that $\varphi_X \approx \Phi$ satisfies $X \hookrightarrow m_{\Phi}$. Similarly, if the function $\Phi$ is quasiconcave, then $\lVert \cdot \rVert_{M_{\Phi}}$ is an r.i.~Banach function norm satisfying $\varphi_{M_{\Phi}} \approx \Phi$ and every r.i.~Banach function space $X$ such that $\varphi_X \approx \Phi$ satisfies $X \hookrightarrow M_{\Phi}$. For proofs, see e.g.~\cite[Propositions~4.3 and 4.4]{MusilovaNekvinda24} and \cite[Chapter~2, Propositions~5.8 and 5.9]{BennettSharpley88}, respectively.

Hence, the property of being admissible characterises fundamental functions of r.i.~quasi-Banach function spaces and for every such function $m_{\Phi}$ is the largest such space. Similarly, the property of being quasiconcave characterises fundamental functions of r.i.~Banach function spaces and for every such function $M_{\Phi}$ is the largest such space. For this reason, the spaces $m_{\Phi}$ and $M_{\Phi}$ are called the weak and strong Marcinkiewicz endpoint spaces, respectively (with the word ``strong'' often being omitted in the case of $M_{\Phi}$ when there is no need to distinguish it from $m_{\Phi}$).

Finally, let us note that while the discussion above was restricted to the measure space $([0, \infty), \lambda)$, the spaces themselves, once constructed, can be easily lifted to an arbitrary resonant measure space $(\mathcal{R}, \mu)$, by simply applying their definitions onto $f^*$ for arbitrary $f \in \mathcal{M}(\mathcal{R}, \mu)$. Denoting the such constructed spaces also $m_{\Phi}$ and $M_{\Phi}$, we get that $m_{\Phi}$ is still an r.i.~quasi-Banach function space (see \cite[Proposition~3.3]{MusilovaNekvinda24}) while $M_{\Phi}$ remains an r.i.~Banach function space (see \cite[Chapter~2, Theorem~4.9]{BennettSharpley88}). Similarly, the definition of a fundamental function can also be applied to r.i.~quasi-Banach function spaces over arbitrary resonant measure spaces; the only difference is that the function will then be defined only for $t$ in the range of $\mu$.

\section{Composition operators on r.i.~quasi-Banach function spaces}\label{section: composition operators on r.i.}

From now on, our standing assumption is that the underlying measure space $(\mathcal{R}, \mu)$ is $\sigma$-finite and resonant, as is usual when working with rearrangement-invariant spaces.

\subsection{Elementary properties of composition operators}

Let us begin by presenting some basic definitions related to the composition operators in the context of r.i.~quasi-Banach function spaces as well as some easy observations about their properties. The proofs are, at best, simple exercises (see also \cite[Section~2.1]{SinghManhas93}).

\begin{definition} \label{DefCompo}
	Let $\phi \colon \mathcal{R} \to \mathcal{R}$ be measurable. We then define the composition operator $T_{\phi} \colon \mathcal{M}(\mathcal{R}, \mu) \to \mathcal{M}(\mathcal{R}, \mu)$ for every $f \in \mathcal{M}(\mathcal{R}, \mu)$ by
	\begin{equation*}
		T_{\phi}f = f \circ \phi.
	\end{equation*}
\end{definition}

\begin{remark}
	Let $E \subseteq \mathcal{R}$ be measurable. Then
	\begin{equation*}
		T_{\phi} \chi_E = \chi_{\phi^{-1}(E)}.
	\end{equation*}
\end{remark}

\begin{definition}
	A measurable map $\phi \colon \mathcal{R} \to \mathcal{R}$ is said to be \emph{non-singular} provided that
	\begin{align*}
		\mu(\phi^{-1}(E)) &= 0 &\text{for every measurable } E \subseteq \mathcal{R} \text{ satisfying } \mu(E) = 0.
	\end{align*}
    If for every measurable $E \subseteq \mathcal{R}$ it holds $\mu(E)=0$ precisely when $\mu(\phi^{-1}(E))=0$ we call $\phi$ \emph{strictly non-singular}.
\end{definition}

\begin{remark} 
    Our terminology for non-singular maps is based on \cite{SinghManhas93} and from there we derive the term for strictly non-singular maps. However, we note that there are alternative approaches. In the contexts of the theory of integration and Sobolev spaces, one considers the \textit{Luzin $N$ and $N^{-1}$ conditions} (for the precise definitions, an overview of their applications, and further references see e.g.~\cite{DolezalovaHrubesova21}). The Luzin $N^{-1}$ condition is equivalent to non-singularity and if a map satisfies both the Luzin $N$ and $N^{-1}$ conditions then it is strictly non-singular; however, the converse implication does not hold as there are strictly non-singular (or even measure-preserving) maps that fail the Luzin $N$ condition (a concrete example can be constructed e.g.~via the Cantor staircase).
\end{remark}

\begin{proposition} \label{PropCharLuzin}
	Let $\phi$ and $T_\phi$ be as in Definition~\ref{DefCompo}. Then the following three statements are equivalent:
	\begin{enumerate}
		\item $\phi$ is non-singular. \label{PropCharLuzin:i}
		\item $T_{\phi}$ is a well-defined mapping $\mathcal{M}(\mathcal{R}, \mu) \to \mathcal{M}(\mathcal{R}, \mu)$. \label{PropCharLuzin:ii}
		\item $T_{\phi}$ is a well-defined mapping $\mathcal{M}_0(\mathcal{R}, \mu) \to \mathcal{M}_0(\mathcal{R}, \mu)$. \label{PropCharLuzin:iib}
		\item $T_{\phi} \colon L^{\infty} \to L^{\infty}$ is a contraction. \label{PropCharLuzin:iii}
	\end{enumerate}
\end{proposition}

Note that \ref{PropCharLuzin:iii} contains implicitly the information that $T_{\phi}$ is well-defined.

\begin{remark} \label{RemarkTrivNiceProp}
	Let the measurable map $\phi \colon \mathcal{R} \to \mathcal{R}$ be non-singular. Then the operator $T_{\phi}$, defined as above, is a positive operator, i.e.~$T_{\phi}(\mathcal{M}_+) \subseteq \mathcal{M}_+$. Furthermore, it has the properties that $\lvert T_{\phi}f \rvert = T_{\phi}(\lvert f \rvert)$, that $f \leq g$ $\mu$-a.e.~implies $T_{\phi}f \leq T_{\phi}g$ $\mu$-a.e., and that $f_n \to f$ $\mu$-a.e.~implies $T_{\phi}f_n \to T_{\phi}f$ $\mu$-a.e. Consequently, \cite[Theorem~3.9]{NekvindaPesa24} implies that it is continuous on an r.i.~quasi-Banach function space $X$ if and only if we have $T_{\phi}f \in X$ for every $f \in X$.
\end{remark}

\begin{proposition} \label{PropCharInjectivity}
    Let $\phi$ and $T_\phi$ be as in Definition~\ref{DefCompo} and let $\phi$ be non-singular. Then the following are equivalent.
    \begin{enumerate}
        \item $\phi$ is strictly non-singular. \label{PropCharInjectivity:i}
        \item $T_\phi \colon \mathcal{M}(\mathcal{R}, \mu) \to \mathcal{M}(\mathcal{R}, \mu)$ is injective. \label{PropCharInjectivity:ii}
    \end{enumerate}
    Moreover, if $\phi$ is strictly non-singular, then $T_\phi f \geq 0$ $\mu$-a.e.~if and only if $f\geq 0$ $\mu$-a.e., where $f\in \mathcal{M}(\mathcal{R}, \mu)$; i.e.~the inverse operator $T_{\phi}^{-1} \colon T_\phi ( \mathcal{M} ) \to \mathcal{M}$ is under this assumption positive.
\end{proposition}

\begin{proof}
    The claims follow easily from the obvious equalities
    \begin{equation*}
        \{ x \in \mathcal{R}; \; T_\phi f(x)=0\}=\phi^{-1}\left(\{x \in \mathcal{R}; \; f(x)=0\}\right)
    \end{equation*}
    and \begin{equation*}
        \{ x \in \mathcal{R}; \; T_\phi f(x)\geq 0\}=\phi^{-1}\left(\{x\in\mathcal{R}; \; f(x)\geq 0\}\right)
    \end{equation*}
    which hold for $f\in \mathcal{M}(\mathcal{R}, \mu)$.
\end{proof}

Next, we present the natural condition on the symbol $\phi$ that characterises continuity on $L^p$, $p \in [1, \infty)$ (see \cite[Theorem~2.1.1]{SinghManhas93}).

\begin{definition}  \label{DefMeasureBounded}
	   A measurable map $\phi \colon \mathcal{R} \to \mathcal{R}$ is said to be \emph{measure bounded}, if there is some $A \in (0, \infty)$ for which it holds that
	\begin{align}
		\mu(\phi^{-1}(E)) &\leq A \mu(E) &\text{for every measurable } E \subseteq \mathcal{R}. \label{DefMeasureBounded:E1}
	\end{align}
	Furthermore, the smallest $A$ for which \eqref{DefMeasureBounded:E1} holds will be called the \emph{measure-bound of $\phi$}.

    Additionally, $\phi$ is said to be \emph{measure-bounded from below}, if there is some $C \in (0, \infty)$ for which it holds that
	\begin{align}
		 \mu(E) &\leq C \mu(\phi^{-1}(E)) &\text{for every measurable } E \subseteq \mathcal{R} \text{ with }\mu(E)<\infty. \label{DefMeasureBoundedBelow:E1}
	\end{align}
    Furthermore, the smallest $C$ for which \eqref{DefMeasureBoundedBelow:E1} holds will be called the \emph{measure-bound from below of $\phi$}.
\end{definition}

\begin{remark}
	It is clear the every measure-bounded map $\phi$ is non-singular. On the other hand, if $\phi$ is non-singular then the measure $\nu$ given by $\nu(E) = \mu(\phi^{-1}(E))$ is absolutely continuous with respect to $\mu$ and so there exists a Radon--Nikodym derivative. Measure-boundedness is then clearly equivalent to this derivative belonging to $L^{\infty}(\mu)$, with its norm being equal to the measure-bound of $\phi$. Similarly, measure-boundedness from below implies that $\mu$ has a Radon-Nikodym derivative with respect to the measure $\nu$ that belongs to $L^\infty(\nu)$.  
\end{remark}

As it turns out, measure-boundedness is very useful when examining the behaviour on the entire class of r.i.~quasi-Banach function spaces. This is a consequence of the estimate \ref{LemmaDilEst_i} by the dilation operator in the following Lemma. The second estimate will be used in the proof of Theorem \ref{ThmInjectiveAndClosedRangeQBFS} below.

\begin{lemma}\label{LemmaDilEst}
	Let $\phi \colon \mathcal{R} \to \mathcal{R}$ be measurable.
    \begin{enumerate}
        \item \label{LemmaDilEst_i} Assume that $\phi$ is measure-bounded and denote its measure-bound by $A$. Then we have for every $f \in \mathcal{M}(\mathcal{R}, \mu)$ that
	       \begin{align*}
		      (T_{\phi} f)^* &\leq D_{A^{-1}} f^* &\text{on $[0, \infty)$}.
	       \end{align*}
        \item \label{LemmaDilEst_ii} Assume that $\phi$ is measure-bounded from below and denote its measure-bound by $C$. Then we have for every $f \in \mathcal{M}(\mathcal{R}, \mu)$ that
	       \begin{align*}
		      (T_{\phi} f)^* &\geq D_{C} f^* &\text{on $[0, \infty)$}.
	       \end{align*}
    \end{enumerate}
\end{lemma}

\begin{proof}
	We only give a proof of \ref{LemmaDilEst_i} as the proof of \ref{LemmaDilEst_ii} is mutatis mutandis the same. It follows from our assumptions on $\phi$ and $A$ that we have
	\begin{equation*}
		(T_{\phi}f)_*(s) = \mu \left( \left \{ x \in \mathcal{R} ; \; \lvert f(\phi(x)) \rvert > s \right \} \right) = \mu \left( \phi^{-1} \left( \left \{ x \in \mathcal{R} ; \; \lvert f(x) \rvert > s \right \} \right) \right) \leq A f_*(s)
	\end{equation*}
	for all $s \in [0, \infty)$. Thence,
	\begin{equation*}
		(T_{\phi} f)^* (t) = \inf \{ s \in [0, \infty); \; (T_{\phi} f)_*(s) \leq t \} \leq \inf \{ s \in [0, \infty); \; Af_*(s) \leq t \} = f^* (A^{-1} t)
	\end{equation*}
	for all $t \in [0, \infty)$.
\end{proof}

\begin{theorem} \label{ThmBoundedQBFS}
	Let $\phi$ and $T_\phi$ be as in Definition~\ref{DefCompo}. Then the following three statements are equivalent:
	\begin{enumerate}
		\item $\phi$ is measure-bounded. \label{ThmBoundedQBFS:i}
		\item $T_{\phi} \colon L^1 \to L^1$ is continuous. \label{ThmBoundedQBFS:ii}
		\item $T_{\phi} \colon X \to X$ is continuous for every r.i.~quasi-Banach function space $X$. \label{ThmBoundedQBFS:iii}
	\end{enumerate}
\end{theorem}

\begin{proof}
	It is clear that \ref{ThmBoundedQBFS:iii} implies \ref{ThmBoundedQBFS:ii} which in turn implies \ref{ThmBoundedQBFS:i}. 
	
	It remains to prove that \ref{ThmBoundedQBFS:i} implies \ref{ThmBoundedQBFS:iii}. Recall that it follows from Theorem~\ref{TDRIS} that the dilation operator $D_{t}$ is for any $t \in (0, \infty)$ and any r.i.~quasi-Banach function space $X$ continuous on the canonical representation space $\overline{X}$; in particular, this holds for $t = A^{-1}$. Hence, we may use Lemma~\ref{LemmaDilEst} \ref{LemmaDilEst_i} to compute
	 \begin{equation*}
	 	\lVert T_{\phi} f \rVert_X = \lVert (T_{\phi} f)^* \rVert_{\overline{X}} \leq \lVert D_{A^{-1}} \rVert_{\overline{X} \to \overline{X}} \lVert f^* \rVert_{\overline{X}} = \lVert D_{A^{-1}} \rVert_{\overline{X} \to \overline{X}} \lVert f \rVert_X.
	 \end{equation*}
\end{proof}

It is clear that $L^1$ is not a unique space with the property that continuity of $T_{\phi}$ implies measure-boundedness of $\phi$; we chose it because it makes the implications in one direction obvious. Indeed, the same argument (i.e.~testing with characteristic functions) works for every $L^{p}$, $p \in (0, \infty)$, among others; this is of course a known fact, see e.g.~\cite[Theorem~2.1.1]{SinghManhas93}. We shall now prove the necessity of the measure-boundedness of the symbol for r.i.~quasi-Banach function spaces over completely atomic measure spaces (except $\ell^{\infty}$, of course).

\begin{theorem} \label{ThmNecessityAtomic}
	Let $(\mathcal{R}, \mu)$ be completely atomic with all atoms having the same measure and let $X$ be an r.i.~quasi-Banach function space (of sequences) over $(\mathcal{R}, \mu)$ such that $X \neq \ell^{\infty}$. If $T_{\phi} \colon X \to X$ is continuous, then $\phi$ is measure-bounded.
\end{theorem}

\begin{proof}
	For the simplicity of notation we will assume that all the atoms have measure one. Assume that $\phi$ is not measure-bounded, i.e.~that we have for every $n$ some set $E_n$ for which $\mu(\phi^{-1}(E_n)) > n \mu(E_n)$. This estimate implies that $\mu(E_n) < \infty$, whence it follows that $E_n$ contains finitely many atoms. A standard combinatorial argument now yields that there is some atom $e_n$ such that $\mu(\phi^{-1}(e_n)) > n$. Since the sequence $\lVert \chi_{e_n} \rVert_X$ is constant with respect to $n$ (because $X$ is~r.i.), assuming that $T_{\phi}\colon X \to X$ is continuous yields that $\chi_{\phi^{-1}(e_n)}^* \geq \chi_{[0, n)}$ is a bounded sequence in $\overline{X}$, the canonical representation space of $X$. Consequently, the Fatou property \ref{P3} of the canonical representation quasinorm implies that $\chi_{[0, \infty)} \in \overline{X}$, which means that $\chi_{\mathcal{R}} \in X$. Finally, when the underlying space is completely atomic, then it is clear that this property uniquely characterises $\ell^\infty$; see also~\cite[Theorem~4.16]{MusilovaNekvinda24} for further context.
\end{proof}

For non-atomic measure spaces, the situation is more complicated, as the following example shows that $L^{\infty}$ is not the only space that admits continuous composition operators induced by non-measure-bounded $\phi$. Furthermore, the space we construct is even an r.i.~Banach function space and the underlying measure space is of finite measure.

\begin{example} \label{CounterExSV}
	Consider the measure space $([0,1], \lambda)$ and the function
	\begin{align} \label{CounterExSV:E1}
		\Phi(t) &= \begin{cases}
			0 & t = 0, \\
			\frac{1}{1 - \log(t)} & t \in (0, 1), \\
			1 & t \in [1, \infty).
		\end{cases}
	\end{align}
	Then $\Phi$ is continuous and quasiconcave on $[0, \infty)$, whence the weak Marcinkiewicz endpoint space $m_{\Phi}$, induced by the functional
	\begin{align*} 
		\lVert f \rVert_{m_{\Phi}} &= \sup_{t \in [0, \infty)} \Phi(t) f^*(t), &f \in \mathcal{M}([0,1], \lambda),
	\end{align*}
	is an r.i.~quasi-Banach function space over $([0,1], \lambda)$. This follows from the abstract theory contained in e.g.~\cite[Section~4]{MusilovaNekvinda24}. It is also easy to verify that $m_{\Phi} = M_{\Phi}$, where $M_{\Phi}$ is the strong Marcinkiewicz endpoint space corresponding to $\Phi$ (see e.g.~\cite[Theorem~5.1]{GogatishviliSoudsky14}, \cite[Proposition~7.10.5]{FucikKufner13}, or \cite[Theorem~4.11]{MusilovaNekvinda24}). $M_{\Phi}$ is an r.i.~Banach function space (see e.g.~\cite[Chapter~2, Section~5]{BennettSharpley88}), but it is more convenient for our purposes to work with the weaker quasinorm. It is also perhaps worth noting that $\Phi$ is a slowly varying function in the sense of e.g.~\cite[Definition~2.15]{PesaLK} and thus $m_{\Phi}$ can also be described as the Lorentz--Karamata space $L^{\infty, \infty, \Phi}$ (with equal quasinorm); it is precisely this very slow decay of $\Phi$ which makes it interesting to our purposes. Finally, let us note that $f \in \mathcal{M}([0,1], \lambda)$ belongs to $m_{\Phi}$ if and only if $f^*(t) \lesssim 1 - \log(t)$ on $[0,1]$.
	
	Let us now consider the mapping $\phi \colon [0,1] \to [0,1]$ given by $\phi(t) = t^n$, where $n \in \mathbb{N}$, $n\geq 2$ is arbitrary. It is clear that $\phi$ is not measure-bounded, as $\phi^{-1}([0, \varepsilon]) = [0, \varepsilon^{\frac{1}{n}}]$. It is also rather clear that we have for every $E \subseteq [0,1]$ that $\lambda(\phi^{-1}(E)) \leq \lambda(\phi^{-1}([0, \lambda(E) ) ) = \lambda(E)^{\frac{1}{n}}$ (since the Radon--Nikodym derivative of the measure $\lambda(\phi^{-1}(\cdot))$ is non-increasing).
	
	Consider now a non-negative simple function $s$,
	\begin{equation*}
		s = \sum_{i=0}^{k} a_i \chi_{E_i},
	\end{equation*}
	where $a_i > 0$ and $E_k \subseteq E_{k-1} \subseteq \dots \subseteq E_1 \subseteq E_0$. We easily observe that
	\begin{equation*}
		(T_{\phi}s)^* = \sum_{i=0}^{k} a_i \chi_{[0, \lambda(\phi^{-1}(E_i)))} \leq \sum_{i=0}^{k} a_i \chi_{[0, \lambda(E)^{\frac{1}{n}})} = T_{\phi}(s^*).
	\end{equation*}
	By approximating arbitrary non-negative $f \in m_{\Phi}$ $\lambda$-a.e.~by an increasing sequence of simple functions and recalling that the non-increasing rearrangement depends only on the modulus of the given function, we conclude that for every function $f \in m_{\Phi}$
	\begin{equation*}
		(T_{\phi} f)^* \leq T_{\phi}(f^*) \lesssim T_{\phi} \left( \frac{1}{\Phi} \right) \leq n \frac{1}{\Phi} \in m_{\Phi}.
	\end{equation*}
	Here, $\frac{1}{\Phi}(t) = 1- \log(t)$ on $[0, 1]$ is the largest function belonging to the space, as observed above. Hence, $T_{\phi}(f) \in m_{\Phi}$ for every $f \in m_{\Phi}$.  The continuity follows by Remark~\ref{RemarkTrivNiceProp}.
	
	On the other hand, if we consider the mapping $\widetilde{\phi} \colon [0,1] \to [0,1]$ given by $\widetilde{\phi}(t) = \exp \left( 1-\frac{1}{t} \right)$ then
	\begin{equation*}
		T_{\widetilde{\phi}} \left( \frac{1 }{\Phi}\right) (t) = \frac{1}{t} \notin m_{\Phi}.
	\end{equation*}
    Hence, unlike $L^{\infty}$, it is not true that every non-singular $\phi$ induces a bounded operator on $m_{\Phi}$.
\end{example}

The space constructed in Example~\ref{CounterExSV} is, in a sense, very close to $L^{\infty}$. By this we mean that it is the (strong) Marcinkiewicz endpoint space for a quasiconcave function that converges to zero at zero, but does so very slowly (see \cite[Theorem~4.15]{MusilovaNekvinda24} and \cite[Chapter~2, Section~5]{BennettSharpley88} for further context; we omit the lengthy presentation as this matter is only tangential to our main results). This observation motivates the following result.

\begin{proposition} \label{PropNecessityXX'}
	Let $\lVert \cdot \rVert_X$ be an r.i.~quasi-Banach function norm satisfying \ref{P5}, let $X$ be the corresponding r.i.~quasi-Banach function space, and let $X'$ be the corresponding associate space. If $T_{\phi}$ is continuous as both $T_{\phi} \colon X \to X$ and $T_{\phi} \colon X' \to X'$, then $\phi$ is measure-bounded.
\end{proposition}

\begin{proof}
	Denote by $\nu$ the measure given by $\nu(E) = \mu(\phi^{-1}(E))$. Our assumption that $T_{\phi}$ is continuous implies that $\phi$ is non-singular, hence $\nu << \mu$ and there exists the appropriate Radon--Nikodym derivative, which we denote $h_{\nu}$. We have to show that $h_{\nu} \in L^{\infty}$.
	
	Fix arbitrary $f \in X$, $\varphi \in X'$. Our assumptions on continuity of $T_{\phi}$ guarantee that
	\begin{equation*}
		\infty > \lVert  f \rVert_X \lVert \varphi \rVert_{X'} \gtrsim \lVert T_{\phi} f \rVert_X \lVert T_{\phi} \varphi \rVert_{X'} \geq\int_{\mathcal{R}} T_{\phi}(f) T_{\phi}(\varphi) \: d\mu = \int_{\mathcal{R}} T_{\phi} (f\varphi) \: d\mu = \int_{\mathcal{R}} f \varphi h_{\nu} \: d\mu.
	\end{equation*}
    Since $(\mathcal{R}, \mu)$ is assumed to be resonant and $X$ is r.i., we get
    \begin{equation*}
        \infty > \lVert  f \rVert_X \lVert \varphi \rVert_{X'} \gtrsim \int_0^{\infty} f^* (\varphi h_{\nu})^* \: d\mu.
    \end{equation*}
	As $f \in X$ was arbitrary, Proposition~\ref{PropLandauRes} implies that $\varphi h_{\nu} \in X'$. Since $\varphi \in X'$ was also arbitrary, this implies that $h_{\nu}$ is a pointwise multiplier on $X'$. \cite[Theorem~1]{MaligrandaPersson89} then yields $h_{\nu} \in L^{\infty}$ (we note the associate space $X'$ is an r.i.~Banach function space, see Theorem~\ref{TFA}).
\end{proof}

We have already observed in Proposition~\ref{PropCharInjectivity} that strict non-singularity of the symbol characterises injectivity of the corresponding composition operator. We shall now explore this behaviour more broadly.

\begin{theorem} \label{ThmInjectiveAndClosedRangeQBFS}
	Let $\phi$ and $T_\phi$ be as in Definition~\ref{DefCompo} and let $\phi$ be measure-bounded. Then the following three statements are equivalent:
	\begin{enumerate}
		\item $\phi$ is measure-bounded from below. \label{ThmInjectiveAndClosedRangeQBFS:i}
		\item $T_{\phi} \colon L^1 \to L^1$ is injective and has closed range. \label{ThmInjectiveAndClosedRangeQBFS:ii}
		\item For every r.i.~quasi-Banach function space $X$ there is a constant $c>0$ such that
        \begin{align*}
            \lVert f \rVert_X &\leq c \lVert T_\phi f\rVert_X & \text{for every } f \in X.
        \end{align*}
        In particular, $T_{\phi} \colon X \to X$ is injective and has closed range. \label{ThmInjectiveAndClosedRangeQBFS:iii}
	\end{enumerate}
\end{theorem}

\begin{proof}
    It is clear that \ref{ThmInjectiveAndClosedRangeQBFS:iii} implies \ref{ThmInjectiveAndClosedRangeQBFS:ii}. Moreover, by an easy consequence (cf.~\cite[Corollary 8.7]{MeiseVogt97}) of the Open Mapping Theorem, if \ref{ThmInjectiveAndClosedRangeQBFS:ii} holds, there is $c>0$ such that $\lVert f\rVert_{L^1}\leq c\lVert T_\phi f\rVert_{L^1}$ for each $f\in L^1$, in particular
    \begin{equation*}
        \mu(E)=\lVert \chi_E\rVert_{L^1}\leq c \lVert T_\phi \chi_E\rVert_{L^1}=c\mu(\phi^{-1}(E))
    \end{equation*}
    for every measurable subset $E$ of $\mathcal{R}$ with finite measure. Hence, \ref{ThmInjectiveAndClosedRangeQBFS:ii} implies \ref{ThmInjectiveAndClosedRangeQBFS:i}.

    Finally, let $\phi$ be measure-bounded from below by $C>0$. As usual, let $\overline{X}$ be the canonical representation of the arbitrary r.i.~quasi-Banach function space $X$. Applying Lemma \ref{LemmaDilEst} \ref{LemmaDilEst_ii} and taking into account the continuity of the Dilation operators $D_t$ for any $t\in (0,\infty)$ (Theorem \ref{TDRIS}), we conclude that for each $f\in X$
    \begin{equation*}
        \lVert f\rVert_X=\lVert f^*\rVert_{\overline{X}}\leq \lVert D_{C^{-1}}\rVert_{\overline{X} \to \overline{X}}\lVert D_C f^*\rVert_{\overline{X}} \leq \lVert D_{C^{-1}}\rVert_{\overline{X} \to \overline{X}}\lVert (T_\phi f)^*\rVert_{\overline{X}}= \lVert D_{C^{-1}}\rVert_{\overline{X} \to \overline{X}} \lVert T_\phi f\rVert_{X},
    \end{equation*}
    so that \ref{ThmInjectiveAndClosedRangeQBFS:iii} follows from \ref{ThmInjectiveAndClosedRangeQBFS:i}.
\end{proof}

\begin{corollary} \label{CorPositiveInverse}
    Let $\phi$ and $T_\phi$ be as in Definition~\ref{DefCompo} and let $\phi$ be both measure-bounded and measure-bounded from below. Then for every r.i.~quasi-Banach function space $X$, $T_\phi$ is injective in $X$ and its inverse $T_\phi^{-1} \colon T_\phi(X) \to X$ is a linear, norm-continuous, and positive mapping, i.e.~$T_\phi^{-1} f\geq 0$ for every $f\in X, f \geq 0$.
\end{corollary}

\begin{proof}
    The claim follows immediately from the positivity of $T_\phi$, and Theorem \ref{ThmInjectiveAndClosedRangeQBFS} combined with Proposition \ref{PropCharInjectivity}, since measure-boundedness and measure-boundedness from below clearly imply strict non-singularity. 
\end{proof}

The following example shows that the composition operator induced by a measure-bounded symbol which is also measure-bounded from below need not be bijective.

\begin{example} \label{ExampleNonSurjectiveButMeasureBoundedFromBelow}
    Consider the measure space $(\mathbb{N}, m)$; i.e.~the non-negative integers equipped with the counting measure. Then $\phi \colon \mathbb{N} \to \mathbb{N}$, given as
    \begin{align*}
        \phi(n) &= \begin{cases}
            0, &\text{for } n = 0, \\
            n-1, &\text{for } n \geq 1,
        \end{cases}
    \end{align*}
    is measure-bounded as well as measure-bounded from below. By Theorem \ref{ThmInjectiveAndClosedRangeQBFS}, $T_\phi$ has closed range and is injective on any r.i.~quasi-Banach function space of sequences, but since
    \begin{equation*}
    \chi_{\{0\}}\notin T_\phi(X)=\{ f \in X; \; f(0)=f(1)\},
    \end{equation*}
    $T_\phi$ is not surjective.
\end{example}

\subsection{Power-boundedness}

Let us now move on to the iterates of $T_{\phi}$. The situation is quite similar.

\begin{remark} \label{RemarkPowerCompo}
	It follows directly from the definition that $T^n_{\phi} = T_{\phi^n}$.
\end{remark}

\begin{definition}  \label{DefPowerMeasureBounded}
	A measurable map $\phi \colon \mathcal{R} \to \mathcal{R}$ is said to be \emph{power-measure-bounded}, if there is some $A \in (0, \infty)$ such that it holds for every $n \in \mathbb{N}$, $n \geq 1$, that
	\begin{align}
		\mu(\phi^{-n}(E)) &\leq A \mu(E) &\text{for every measurable } E \subseteq \mathcal{R}. \label{DefPowerMeasureBounded:E1}
	\end{align}
	
	Furthermore, the smallest $A$ for which \eqref{DefPowerMeasureBounded:E1} holds for every $n \in \mathbb{N}$, $n \geq 1$ will be called the \emph{power-measure-bound of} $\phi$.
\end{definition}

\begin{lemma}\label{LemmaDilPowerEst}
	Let $\phi$ be power-measure-bounded and denote its power-measure-bound by $A$. Then we have for every $n \in \mathbb{N}$, $n\geq 1$, and every $f \in \mathcal{M}(\mathcal{R}, \mu)$ that
	\begin{align*}
		(T^n_{\phi} f)^* &\leq D_{A^{-1}} f^* &\text{on $[0, \infty)$}.
	\end{align*}
\end{lemma}

\begin{proof}
	It follows from our assumption of power-measure-boundedness of $\phi$ that we have for every $n \in \mathbb{N}$, $n\geq 1$, that $\phi^n$ is measure-bounded and its measure-bound $A_n$ satisfies $A_n \leq A$. Applying Lemma~\ref{LemmaDilEst} and Remark~\ref{RemarkPowerCompo}, we obtain for every such $n$ and every $f \in \mathcal{M}(\mathcal{R}, \mu)$ that
	\begin{align*}
		(T^n_{\phi}f)^* &= (T_{\phi^n} f)^* \leq D_{A_n^{-1}} f^* \leq D_{A^{-1}} f^* &\text{on $[0, \infty)$},
	\end{align*}
	where the last estimate is due to $f^*$ being non-increasing.
\end{proof}

\begin{theorem} \label{ThmPowerBoundedQBFS}
	Let $\phi$ and $T_\phi$ be as in Definition~\ref{DefCompo}. Then the following three statements are equivalent:
	\begin{enumerate}
		\item $\phi$ is power-measure-bounded. \label{ThmPowerBoundedQBFS:i}
		\item $T_{\phi} \colon L^1 \to L^1$ is power-bounded. \label{ThmPowerBoundedQBFS:ii}
		\item $T_{\phi} \colon X \to X$ is power-bounded for every r.i.~quasi-Banach function space $X$. \label{ThmPowerBoundedQBFS:iii}
	\end{enumerate}
\end{theorem}

\begin{proof}
	It is again clear that \ref{ThmPowerBoundedQBFS:iii} implies \ref{ThmPowerBoundedQBFS:ii} which in turn implies \ref{ThmPowerBoundedQBFS:i}. As for the remaining implication, we proceed analogously to the proof of Theorem~\ref{ThmBoundedQBFS} and use Lemma~\ref{LemmaDilPowerEst} to show for every $n \in \mathbb{N}$, $n\geq 1$, and every $f \in \mathcal{M}(\mathcal{R}, \mu)$ that
	\begin{equation*}
		\lVert T^n_{\phi} f \rVert_X \leq \lVert D_{A^{-1}} \rVert_{\overline{X} \to \overline{X}}  \lVert f \rVert_X.
	\end{equation*}
	The conclusion now follows, as $\lVert D_{A^{-1}} \rVert_{\overline{X} \to \overline{X}}$ does not depend on $n$.
\end{proof}

Again, the necessity of power-measure-boundedness of the symbol $\phi$ for power-boundedness of $T_{\phi}$ on a given r.i.~quasi-Banach function space is not restricted to the spaces $L^p$, $p < \infty$ (and much less $L^1$ which we chose in the theorem for convenience). The argument for the next result is almost identical to that of Theorem~\ref{ThmNecessityAtomic}.

\begin{theorem} \label{ThmNecessityAtomicPower}
	Let $(\mathcal{R}, \mu)$ be completely atomic with all atoms having the same measure and let $X$ be an r.i.~quasi-Banach function space over $(\mathcal{R}, \mu)$ such that $X \neq \ell^{\infty}$. If $T_{\phi} \colon X \to X$ is power-bounded, then $\phi$ is power-measure-bounded.
\end{theorem}

The analogue of Proposition~\ref{PropNecessityXX'} requires a small amount of extra work.

\begin{proposition} \label{PropNecessityXX'Power}
	Let $\lVert \cdot \rVert_X$ be an r.i.~quasi-Banach function norm satisfying \ref{P5}, let $X$ be the corresponding r.i.~quasi-Banach function space, and let $X'$ be the corresponding associate space. If $T_{\phi}$ is power-bounded as both $T_{\phi} \colon X \to X$ and $T_{\phi} \colon X' \to X'$, then $\phi$ is power-measure-bounded.
\end{proposition}

\begin{proof}
    Denote by $A$ some constant for which we have both
    \begin{align*}
        \lVert T_{\phi}^n f \rVert_X &\leq A \lVert f \rVert_X &\text{for every } f \in X \text{ and every } n \in \mathbb{N}, \\
        \lVert T_{\phi}^n \varphi \rVert_{X'} &\leq A \lVert \varphi \rVert_{X'} &\text{for every } \varphi \in X' \text{ and every } n \in \mathbb{N}.
    \end{align*}    
    Consider the functions $h_{\nu_n}$, Radon--Nikodym derivatives of the measures $\nu_n( \cdot ) = \mu(\phi^{-n}( \cdot ))$, obtained analogously as in the proof of Proposition~\ref{PropNecessityXX'}. 

	Fix arbitrary $f \in X$, $\lVert f \rVert_X \leq 1$, $\varphi \in X'$, and $n \in \mathbb{N}$. We compute
	\begin{equation*}
    \begin{split}
		\infty &> A^2 \lVert f \rVert_X \lVert \varphi \rVert_{X'} \\ 
        &\geq \lVert T_{\phi}^n f \rVert_X \lVert T_{\phi}^n \varphi \rVert_{X'} \\
        &\geq \int_{\mathcal{R}} T^n_{\phi}(f) T^n_{\phi}(\varphi) \: d\mu \\
        &= \int_{\mathcal{R}} T^n_{\phi} (f\varphi) \: d\mu \\
        &= \int_{\mathcal{R}} f \varphi h_{\nu_n} \: d\mu.
    \end{split}
	\end{equation*}
	Hence, $\lVert \varphi h_{\nu_n} \rVert_{X'} \leq A^2 \lVert \varphi \rVert_{X'}$. Since $\varphi \in X'$ was arbitrary, this implies that $h_{\nu_n}$ is a pointwise multiplier on $X'$ and the norm of the corresponding operator is at most $A^2$. \cite[Theorem~1]{MaligrandaPersson89} then yields $h_{\nu_n} \in L^{\infty}$, $\lVert h_{\nu_n} \rVert_{L^{\infty}} \leq A^2$. Thence, $\phi$ is power-measure-bounded.
\end{proof}

The analogue of Example~\ref{CounterExSV} is more complicated. It is obvious from the there presented arguments that the rather simple and well-behaved symbol we used does not lead to a power-bounded composition operator. We fix this issue by introducing a shift into the symbol, which ensures that the increase of mass the composition creates does not propagate through the iterations. However, this introduces new technical challenges and it also forces us to require that the underlying measure space has infinite measure.

\begin{example} \label{CounterExSVPower}
	Consider the measure space $([0, \infty), \lambda)$ and the function $\Phi \colon [0, \infty) \to [0, \infty)$ given by \eqref{CounterExSV:E1}. Since we have changed the underlying measure space, the desired weak Marcinkiewicz endpoint space $m_{\Phi}$ is now induced by a functional given by the same formula, but with a different domain:
	\begin{align*}
		\lVert f \rVert_{m_{\Phi}} &= \sup_{t \in [0, \infty)} \Phi(t) f^*(t), &f \in \mathcal{M}([0,\infty), \lambda).
	\end{align*}
	However, all the considerations presented in Example~\ref{CounterExSV} remain valid, i.e.~it is an r.i.~quasi-Banach function space that is equivalent to the r.i.~Banach function space $M_{\Phi}$ and, notably, $f \in \mathcal{M}([0,\infty), \lambda)$ belongs to $m_{\Phi}$ if and only if $f^*(t) \lesssim 1 - \log(t)$ on $[0,1]$ (since then automatically $\sup_{t \in (1, \infty)} \Phi(t) f^*(t) = f^*(1) < \infty$). This time, however, we will also need the quantitative version of the statement, that is
    \begin{align} \label{CounterExSVPower:E1}
        f^* &\leq \lVert f \rVert_{m_{\Phi}} \frac{1}{\Phi} & \text{for every } f \in m_{\Phi}.
    \end{align}
    
    As for the symbol for the composition operator, the construction gets more complicated. We consider the transformation $\phi\colon [0, \infty) \to [0, \infty)$ given by
    \begin{align*}
        \phi(t) &= \begin{cases}
            1 + t^n &\text{for } t \in [0, 1), \\
            n(t-1) + 2 &\text{for } t \in [1, \infty),
        \end{cases}
    \end{align*}
    where $n \in \mathbb{N}$, $n \geq 2$ is arbitrary. It is then clear that $\phi$ is not even measure-bounded, much less power-measure bounded, since $\phi^{-1}([1, 1+\varepsilon]) = [0, \varepsilon^{\frac{1}{n}}]$, $\varepsilon \in (0,1)$. Furthermore, we may compute the Radon--Nikodym derivative of the measure $\nu( \cdot ) = \mu(\phi^{-1}( \cdot ))$, which we denote $h_{\nu}$, obtaining
    \begin{align} \label{CounterExSVPower:E2}
        h_{\nu}(s) &= \begin{cases}
            0 &\text{for } s \in [0, 1), \\
            \frac{1}{n} (s-1)^{\frac{1}{n} - 1} &\text{for } s \in (1, 2), \\
            \frac{1}{n} &\text{for } s \in [2, \infty).
        \end{cases}
    \end{align}    
    Consider now for arbitrary $E \subseteq [0, \infty)$ the set $\widetilde{E} = E \cap [1, \infty)$. Then \eqref{CounterExSVPower:E2} implies that
    \begin{align*}
        \lambda(\phi^{-1}(E)) &= \lambda(\phi^{-1}(\widetilde{E})) \leq \lambda(\phi^{-1}([1, 1+\lambda(\widetilde{E})))) = \begin{cases}
            \lambda(\widetilde{E})^{\frac{1}{n}} &\text{ if }  \lambda(\widetilde{E}) \in [0, 1], \\
            \frac{\lambda(\widetilde{E})-1}{n} + 1 &\text{ if }  \lambda(\widetilde{E}) \in [1, \infty],
        \end{cases}
    \end{align*}
    since $h_{\nu}$ is zero on $[0, 1)$ and non-increasing on $(1, \infty)$. Consequently, given a simple function $s$ and writing it down as
	\begin{equation*}
		s = \sum_{i=0}^{k} a_i \chi_{E_i} + \overline{s},
	\end{equation*}
	where $a_i > 0$ and $E_k \subseteq E_{k-1} \subseteq \dots \subseteq E_1 \subseteq E_0 \subseteq [1, \infty)$, while $\overline{s} = s\chi_{[0,1)}$, we easily observe that
	\begin{equation*}
		(T_{\phi}s)^* = \sum_{i=0}^{k} a_i \chi_{[0, \lambda(\phi^{-1}(E_i)))} \leq \sum_{i=0}^{k} a_i \chi_{[0, \lambda(\phi^{-1}( [1, 1+\lambda(E_i)) )) )} = T_{\phi} \left(\widetilde{s} \right),
	\end{equation*}
    where
    \begin{align} \label{CounterExSVPower:E3}
        \widetilde{s}(t) &= \begin{cases}
            0 &\text{for } t \in [0, 1), \\
            (s \chi_{[1, \infty)})^*(t -1) &\text{for } t \in [1, \infty).
        \end{cases}
    \end{align}
    By approximating arbitrary non-negative $f \in m_{\Phi}$ $\lambda$-a.e.~by an increasing sequence of simple functions and recalling that the non-increasing rearrangement depends only on the modulus of the given function, we conclude that for every function $f \in m_{\Phi}$
    \begin{equation*}
        (T_{\phi}f)^* \leq T_{\phi} \left( \widetilde{f} \right),
    \end{equation*}
    where $\widetilde{f}$ is defined for a given function $f$ as in \eqref{CounterExSVPower:E3}. Consequently, considering the function $\Psi \colon [0, \infty) \to [0, \infty)$ given by the formula
    \begin{align*} 
        \Psi(t) &= \begin{cases}
            0 &\text{for } t \in [0, 1), \\
            \frac{1}{\Phi}(t-1) &\text{for } t \in [1, \infty),
        \end{cases}
    \end{align*}
    we observe that $\Psi = \widetilde{\Psi}$ and $\Psi^* = \frac{1}{\Phi}$, so $\Psi$ is at the same time the essentially largest function in $m_{\Phi}$ and optimally rearranged in order to maximise $(T_{\phi}\Psi)^*$. Thence, we get for every $f \in m_{\Phi}$ that \eqref{CounterExSVPower:E1} implies $\widetilde{f} \leq \lVert f \rVert_{m_{\Phi}} \Psi$ and thus 
	\begin{equation*}
		(T_{\phi} f)^* \leq T_{\phi} \left( \widetilde{f} \right) \leq \lVert f \rVert_{m_{\Phi}}  T_{\phi} \left( \Psi \right) \leq \lVert f \rVert_{m_{\Phi}} ( n(1- \log(\cdot)) \chi_{(0, 1)} + \chi_{[1, \infty)} )\leq \lVert f \rVert_{m_{\Phi}} n \frac{1}{\Phi} \in m_{\Phi}.
	\end{equation*}
    It follows that $T_{\phi}$ is continuous on $m_{\Phi}$ and $\lVert T_{\phi} \rVert_{m_{\Phi} \to m_{\Phi}} \leq n$.

    However, we need to show that $T_{\phi}$ is power-bounded. We will show that the norm in fact does not increase.

    Let $k \geq 2$ be fixed. Let us define the following sequence of numbers:
    \begin{align*}
        a_i &= \begin{cases}
            0& \text{for } i = 0, \\
            1& \text{for } i = 1, \\
            2& \text{for } i = 2, \\
            2 + \sum_{j = 1}^{i - 2} n^j & \text{for } i \geq 3;
        \end{cases}
    \end{align*}
    and of sets:
    \begin{equation*}
        E_i = [a_i, a_{i+1}].
    \end{equation*}
    Then it holds for $i \geq 1$ that $\phi^{-1}(E_i) = E_{i-1}$, while $\phi^{-1}(E_0) = \emptyset$. Thus, by iteration,
    \begin{align} \label{CounterExSVPower:E4}
        \phi^{-(k-1)} (E_i) &= \begin{cases}
            E_{i - (k-1)} &\text{for } i \geq k-1, \\
            \emptyset &\text{for } i \leq k-2.
        \end{cases} 
    \end{align}

    Let now $f \in m_{\Phi}$ be arbitrary and denote
    \begin{align*}
        f_0 &= f \chi_{\cup_{i = 0}^{k-2} E_i}, \\
        f_1 &= f \chi_{E_{k-1}}, \\
        f_2 &= f \chi_{\cup_{i = k}^{\infty} E_i}.
    \end{align*}
    Then of course $f = f_0 + f_1 + f_2$. Investigating those three functions separately (in the light of \eqref{CounterExSVPower:E4}), we observe that
    \begin{itemize}
        \item $T_{\phi}^{k-1} f_0 = 0$,
        \item $T_{\phi}^{k-1} f_1$ is supported on a subset of $E_0$ and thus $T_{\phi} (T_{\phi}^{k-1} f_1) = 0$,
        \item and finally $T_{\phi}^{k-1} f_2$ has support inside $[1, \infty)$, which implies that $T_{\phi}^{k-1} f_2 = D_{n^{k-1}} f_2$, where $D_{n^{k-1}}$ is a dilation operator as introduced in Definition~\ref{DDO}.
    \end{itemize}
    We recall that $D_{n^{k-1}}$ is a contraction on every r.i.~quasi-Banach function space (because $n^{k-1} \geq 1$; the argument for the values of the dilation parameter larger than one is rather trivial,  see \cite[Thorem~3.23]{NekvindaPesa24}). Thus,
    \begin{equation*}
        \lVert T_{\phi}^{k} f \rVert_{m_{\Phi}} \leq 0 + 0 + \lVert T_{\phi} ( D_{n^{k-1}} f_2 ) \rVert_{m_{\Phi}} \leq n \lVert f \rVert_{m_{\Phi}}.
    \end{equation*}
    That is, $T_{\phi}$ is power-bounded on $m_{\Phi}$.
\end{example}

\subsection{Absolute continuity of the quasinorm}

We conclude this section by observing that measure-boundedness of $\phi$ ensures that $T_{\phi}$ preserves absolute continuity of the quasinorm, while power-measure-boundedness of $\phi$ ensures that this preservation is in a way uniform for the iterates $\{ T_{\phi}^n; \; n \in \mathbb{N} \}$.

\begin{corollary} \label{CorPresXa}
	Let $\phi$ and $T_\phi$ be as in Definition~\ref{DefCompo}, let $X$ be an r.i.~quasi-Banach function space. Then:
	\begin{enumerate}
		\item \label{CorPresXa_i} If $\phi$ is measure-bounded and $f \in X$ has absolutely continuous quasinorm, then $T_{\phi}f \in X$ also has absolutely continuous quasinorm.
		\item \label{CorPresXa_ii} If $\phi$ is power-measure-bounded and $f \in X$ has absolutely continuous quasinorm, then the sequence $(T^n_{\phi}f)^* \in \overline{X}$ has uniformly absolutely continuous quasinorm, i.e. it holds for every sequence $E_k$ of subsets of $[0, \infty)$ satisfying $\chi_{E_k} \to 0$ $\lambda$-a.e.~as $k \to \infty$ that
		\begin{align*}
			\sup_{n \in \mathbb{N}} \lVert (T^n_{\phi}f)^* \chi_{E_k} \rVert_{\overline{X}} &\to 0  &\text{as } k \to \infty.
		\end{align*}
	\end{enumerate}
\end{corollary}

\begin{proof}
	\ref{CorPresXa_i} follows by combining Lemma~\ref{LemmaDilEst}, Lemma~\ref{LemDilPresXa}, and Theorem~\ref{ThmInheritanceACqN}. As for \ref{CorPresXa_ii}, we are now working with rearrangements, for which we have a uniform pointwise estimate by the function $D_{A^{-1}}f^* \in \left( \overline{X} \right)_a$ (see Lemma~\ref{LemmaDilPowerEst}, Lemma~\ref{LemDilPresXa}, and Theorem~\ref{ThmRepreACqN}), hence it follows directly from the lattice property of $\lVert \cdot \rVert_{\overline{X}}$ that
	\begin{align*}
		\sup_{n \in \mathbb{N}} \lVert (T^n_{\phi}f)^* \chi_{E_k} \rVert_{\overline{X}} \leq \lVert D_{A^{-1}}f^* \chi_{E_k} \rVert_{\overline{X}} &\to 0  &\text{as } k \to \infty.
	\end{align*}
\end{proof}

\section{A rearrangement-invariant topology on r.i.~quasi-Banach function spaces} \label{SecRITopo}

We continue working under the standing assumption that the underlying measure space $(\mathcal{R}, \mu)$ is $\sigma$-finite and resonant.

\subsection{Absolute continuity of rearrangements}\label{SectionACR}

We begin by briefly examining the properties of functions that have absolutely continuous rearrangement. These properties are analogues to those of functions having absolutely continuous quasinorm in some quasi-Banach function space, which is the motivation for our choice of terminology.

\begin{proposition} \label{PropACR}
	Let $f \in \mathcal{M}_{(ACR)}$ and assume that the sequence $E_k \subseteq \mathcal{R}$ satisfies $\chi_{E_k} \to 0$ $\mu$-a.e.~as $k \to \infty$. Then
	\begin{align*}
		\lim_{k \to \infty} (f \chi_{E_k})^*(t) &= 0 & \text{for } t \in (0, \infty).
	\end{align*}
\end{proposition}

\begin{proof}
	Fix $f$ and $E_k$ as in the assumptions and arbitrary $\varepsilon > 0$. Since $f \in \mathcal{M}_{(ACR)}$, we have for $F \subseteq \mathcal{R}$ given by
	\begin{equation*}
		F = \{ \lvert f \rvert > \varepsilon\}
	\end{equation*}
	that
	\begin{equation*}
		\mu(F) = \lambda(\{ f^* > \varepsilon\}) < \infty.
	\end{equation*}
	Hence, for any $t \in (0,\infty)$ there is an $k_0$ such that for every $k \geq k_0$ we have 
	\begin{equation*}
		\lambda(\{ (f \chi_{E_k})^* > \varepsilon \}) = \mu(F\cap E_k) < t,
	\end{equation*}
	which implies that $(f \chi_{E_k})^*(t) \leq \varepsilon$.
\end{proof}

\begin{proposition} \label{PropACRDom}
	Let $g \in \mathcal{M}_{(ACR)}$, $g \geq 0$, and assume that we have a sequence of functions $f_n \in \mathcal{M}_0$ such that $\lvert f_n \rvert \leq g$ $\mu$-a.e.~and that there is some function $f \in \mathcal{M}_0$ such that $f_n \to f$ $\mu$-a.e. Then 
	\begin{align*}
		\lim_{n \to \infty} (f - f_n)^*(t) &= 0 & \text{for } t \in (0, \infty).
	\end{align*}
\end{proposition}

\begin{proof}
	Fix $t \in (0, \infty)$ and $\varepsilon > 0$. Put
	\begin{equation*}
		E_n = \{ \lvert f - f_n \rvert > \varepsilon \}.
	\end{equation*}
	As $\lvert f - f_n \rvert \leq 2g$ $\mu$-a.e., we get
	\begin{equation*}
		\lambda \left( \left\{ (f - f_n)^* > \varepsilon \right\} \right) = \mu \left( \left \{ \lvert f - f_n \rvert \chi_{E_n} > \varepsilon \right \} \right) \leq \mu \left( \left \{ g \chi_{E_n} > \frac{\varepsilon}{2} \right \} \right) = \lambda \left( \left\{ (g \chi_{E_n})^* > \frac{\varepsilon}{2} \right\} \right).
	\end{equation*}
	
	Since $g \in \mathcal{M}_{(ACR)}$ and $\chi_{E_n} \to 0$ $\mu$-a.e.~as $n \to \infty$, by Proposition~\ref{PropACR} $\lim_{n \to \infty} (g \chi_{E_n})^*(t) = 0$. Hence, there is $n_0$ such that we have for every $n \geq n_0$ that
	\begin{equation*}
		\lambda \left( \left\{ (f - f_n)^* > \varepsilon \right\} \right) < t,
	\end{equation*}
	which implies that $(f - f_n)^*(t) \leq \varepsilon$.
\end{proof}

As with Proposition~\ref{PropDomConv}, Proposition~\ref{PropACRDom} has a corollary that is worth mentioning separately.

\begin{corollary} \label{PropACRMono}
	Let $f \in \mathcal{M}_{(ACR)}$, $f \geq 0$, and assume that we have a sequence of functions $f_n \in \mathcal{M}_0$ such that $0 \leq f_n \uparrow f$ $\mu$-a.e. Then 
	\begin{align*}
		\lim_{n \to \infty} (f - f_n)^*(t) &= 0 & \text{for } t \in (0, \infty).
	\end{align*}
\end{corollary}

\subsection{The rearrangement-invariant locally convex topology} We are now in a position to introduce the new locally convex topology generated by a family of rearrangement-invariant norms that is the key ingredient of the proofs of our main results.

\begin{definition}
	Let $\lVert \cdot \rVert_X$ be an r.i.~quasi-Banach function norm satisfying \ref{P5}, let $X$ be the corresponding r.i.~quasi-Banach function space, and let $X'$ be the corresponding associate space. We denote for every $\varphi \in X'$ by $\lvert \cdot \rvert_{\varphi}$ the functional that is defined for every $f \in X$ by
	\begin{equation*}
		\lvert f \rvert_{\varphi} = \int_0^{\infty} \varphi^* f^* \: d\lambda.
	\end{equation*}
\end{definition}

\begin{remark} \label{RemPropPhi*Norm}
	The functional $\lvert \cdot \rvert_{\varphi}$ is for every $\varphi \in X'$, $\varphi \neq 0$, a norm on $X$ that satisfies $\lvert \cdot \rvert_{\varphi} \lesssim \lVert \cdot \rVert_X$. Indeed, the estimate is due to Proposition~\ref{PAS}, the triangle inequality follows from the subadditivity of the elementary maximal function (Proposition~\ref{PropSubAdd**}) and Hardy's lemma (Lemma~\ref{LemmaHardy}), absolute homogeneity is obvious, and finally $\lvert f \rvert_{\varphi} = 0$ is clearly true if and only if $f^* = 0$ $\lambda$-a.e.~which in turns holds if and only if $f = 0$ $\mu$-a.e. This of course means that the family of norms $\{\lvert \cdot \rvert_{\varphi}; \; \varphi \in X'\}$ separates points.	
	
	Furthermore, our assumption that the measure space $(\mathcal{R}, \mu)$ is resonant guarantees (see Remark~\ref{RemResonant}) that it holds for every $\varphi \in X'$ that
	\begin{align} \label{RemPropPhi*Norm:E1}
		\lvert f \rvert_{\varphi} &= \sup_{\substack{\psi \in X' \\ \psi ^* = \varphi^*}} \left \lvert \int_{\mathcal{R}} \psi f \: d\mu \right \rvert &\text{for all } f \in X.
	\end{align}
	On the other hand, Proposition~\ref{PAS} implies that
	\begin{align} \label{RemPropPhi*Norm:E2}
		\lVert f \rVert_{X''} &= \sup_{\lVert \varphi \rVert_{X'} \leq 1} \lvert f \rvert_{\varphi} &\text{for all } f \in \mathcal{M}_0,
	\end{align}
	and in the case when $\lVert \cdot \rVert_X$ is an r.i.~Banach function norm the same equality holds for $\lVert \cdot \rVert_X$, as follows from Theorem~\ref{TFA}.
	
	Finally, let us note that we have for every $\varphi_1, \varphi_2 \in X'$ that both $\lvert \varphi_1 \rvert + \lvert \varphi_2 \rvert \in X'$ and $ (\lvert \varphi_1 \rvert + \lvert \varphi_2 \rvert)^* \geq \max \{ \varphi_1^*, \, \varphi_2^*\}$, whence $\lvert \cdot \rvert_{\lvert \varphi_1 \rvert + \lvert \varphi_2 \rvert} \geq \max \{\lvert \cdot \rvert_{\varphi_1}, \, \lvert \cdot \rvert_{\varphi_2}  \}$. That is, $\{\lvert \cdot \rvert_{\varphi}; \; \varphi \in X'\}$ is a saturated family of norms.
\end{remark}

\begin{definition}
	Let $\lVert \cdot \rVert_X$ be an r.i.~quasi-Banach function norm satisfying \ref{P5}, let $X$ be the corresponding r.i.~quasi-Banach function space, and let $X'$ be the corresponding associate space. By the \textit{rearrangement-invariant locally convex topology} on $X$, denoted $\newtopology$, we mean the locally convex Hausdorff topology induced by the saturated family of norms
	\begin{equation*}
		\left \{ \lvert \cdot \rvert_{\varphi} ; \; \varphi \in X' \right \}.
	\end{equation*}
\end{definition}

\begin{remark} \label{RemEmbeTopo}
	It follows directly from Remark~\ref{RemPropPhi*Norm} and \cite[Lemma 22.5]{MeiseVogt97} that $\newtopology$ is indeed a Hausdorff locally convex topology on $X$ and we always have $(X, \lVert \cdot \rVert_X) \hookrightarrow (X, \newtopology) \hookrightarrow (X, w')$.
\end{remark}

\begin{remark}\label{RemNormalTopo}
    There is another natural topology which, at least for sequence spaces, is studied in the literature, e.g.~in \cite{Koethe69}. The so-called \emph{normal topology} $\mathfrak{n}$ on an r.i.~Banach function space $X$ is defined by the apparently saturated family of seminorms $\{p_\varphi; \; \varphi\in X'\}$, where for every $f\in X$
    \begin{equation*}
        p_\varphi(f)=\int_{\mathcal{R}} \lvert \varphi f \rvert \: d\mu.
    \end{equation*}
    As follows immediately from the Hardy--Littlewood inequality, $(X,\newtopology)\hookrightarrow (X,\mathfrak{n})$. Although it is quite obviously true for $X=\ell^1$ (so that $X'=\ell^\infty$) that $\mathfrak{n}$---and thus also $\xi$---coincides with the norm topology on $\ell^1$, cf.~\cite[§30.2(3)]{Koethe69}, this is not true for any other r.i.~Banach function space of sequences; in every other case, $\mathfrak{n}$ is strictly coarser than $\xi$. Indeed, equip $\N$ with the counting measure $m$ and denote by $e_j$, $j\in\N$, the sequence $(\delta_{k,j})_{k\in\N}$ (Kronecker's $\delta$). Then $e_j^*=\chi_{[0,1]}$ for every $j \in \N$ and for any r.i.~Banach function space $X$ over $(\N,m)$ and $\varphi\in X'$ we get
    \begin{equation*}
        \lvert e_j \rvert_\varphi=\int_0^1 \varphi^* \: d\lambda.
    \end{equation*}
    Therefore, assuming $(X,\mathfrak{n})\hookrightarrow (X,\xi)$, for each $\varphi\in X'$ there are $C>0$ and $g=(g_k)_{k\in\N}\in X'$ such that
    \begin{align*}
        \int_0^1 \varphi^* \: d\lambda &= \lvert e_j \rvert_\varphi \leq C p_g(e_j)=C \lvert g_j \rvert &\text{for every } j \in \mathbb{N},
    \end{align*}
    which, in case of $\varphi\neq 0$, implies that $g \in X' \setminus c_0$. Using an argument similar to that of \cite[Proof of Theorem~4.16]{MusilovaNekvinda24}, this implies that $\chi_{\mathbb{N}} \in X'$, whence $X'=\ell^\infty$ and thus $X = (\ell^\infty)' = \ell^1$ (see Theorem~\ref{TFA}; in this remark we work in the context of r.i.~Banach function spaces).
\end{remark}

We commence the examination of the properties of the $\newtopology$ topology by the following result which shows that it interacts well with the canonical representation of $X$.

\begin{theorem} \label{ThmRepre(w')*}
	Let $\lVert \cdot \rVert_X$ be an r.i.~quasi-Banach function norm satisfying \ref{P5}, let $\lVert \cdot \rVert_{\overline{X}}$ be the canonical representation quasinorm, let $X$ and $\overline{X}$ be the corresponding r.i.~quasi-Banach function spaces, and let $X'$ and $\left( \overline{X} \right)'$ be the corresponding associate spaces. Then $(X, \newtopology)$ is represented by $\left( \overline{X}, \newtopology \right)$, i.e.~ it holds for every $\varphi \in X'$ that there is some $\psi \in  \left( \overline{X} \right)'$ such that we have for every $f \in X$
	\begin{equation} \label{ThmRepre(w')*:E0}
		\lvert f \rvert_{\varphi} = \lvert f^* \rvert_{\psi}
	\end{equation}
	and, conversely, it also holds for every $\psi \in  \left( \overline{X} \right)'$ that there is some $\varphi \in X'$ such that \eqref{ThmRepre(w')*:E0} holds for every $f \in X$.
\end{theorem}

We want to stress that the theorem only allows for testing the equality in \eqref{ThmRepre(w')*:E0} by $f \in X$ and its non-increasing rearrangement, i.e.~when one wishes to examine some $g \in \overline{X}$, then one needs to know a~priori that $g$ is the non-increasing rearrangement of some function in $X$.

\begin{proof}
	We first note that $\lVert \cdot \rVert_X$ satisfying \ref{P5} implies the same for $\lVert \cdot \rVert_{\overline{X}}$ (by Proposition~\ref{PropRepreP5}), hence all the spaces, norms, and topologies are well defined. Furthermore, Proposition~\ref{PropRepreAS} shows that $\varphi^* \in \left( \overline{X} \right)'$ for every $\varphi \in X'$, whence we only need to prove the converse statement. We will perform the proof separately for the two cases of resonant measure spaces.
	
	When $(\mathcal{R}, \mu)$ is non-atomic, then it follows from the classical Sierpiński theorem (see e.g.~\cite[Chapter~2, Corollary~7.8]{BennettSharpley88}), that there is for every $\psi \in \mathcal{M}([0,\infty), \lambda)$ some function $\varphi \in \mathcal{M}(\mathcal{R}, \mu)$ such that $\varphi^* = \psi^* \chi_{[0, \mu(\mathcal{R}))}$. Whence, it clearly holds for every $f \in X \subseteq \mathcal{M}(\mathcal{R}, \mu)$ that
	\begin{equation*}
		\int_0^{\infty} \varphi^* f^* \: d\lambda = \int_0^{\mu(\mathcal{R})} \psi^* f^* \: d\lambda = \int_0^{\infty} \psi^* f^* \: d\lambda \leq \lVert \psi \rVert_{\left( \overline{X} \right)'} \lVert f^* \rVert_{\overline{X}} = \lVert \psi \rVert_{\left( \overline{X} \right)'} \lVert f \rVert_{X}.
	\end{equation*}
	This shows both that $\varphi \in X'$ whenever $\psi \in \left( \overline{X} \right)'$ and that \eqref{ThmRepre(w')*:E0} holds.
	
	Assume now that $(\mathcal{R}, \mu)$ is completely atomic and all atoms have the same measure $\beta \in (0, \infty)$. Then the non-increasing rearrangement of any function from $\mathcal{M}(\mathcal{R}, \mu)$ is necessarily constant on the intervals $[n\beta, (n+1)\beta)$, $n \in \mathbb{N}$, $(n+1)\beta \leq \mu(\mathcal{R})$. It therefore holds for every $\psi \in \left( \overline{X} \right)'$ and every $f \in X$ that
	\begin{equation*}
		\int_0^{\infty} \psi^* f^* \: d\lambda = \sum_{\substack{n \in \mathbb{N} \\ (n+1) \beta \leq \mu(\mathcal{R})}} f^*(n) \int_{\beta n}^{\beta (n+1)} \psi^* \: d\lambda.
	\end{equation*}
	Note that all the integrals on the right-hand side are well defined and finite by the property \ref{P5} of $\left( \overline{X} \right)'$. Consider now some enumeration $\{e_i\}$ of atoms in $\mathcal{R}$, recalling that $(\mathcal{R}, \mu)$ is $\sigma$-finite and thus the set is at most countable. Then the function $\varphi \colon \mathcal{R} \to \mathbb{C}$ defined pointwise on $\mathcal{R}$ by
	\begin{equation} \label{ThmRepre(w')*:E1}
		\varphi(e_n) = \beta^{-1} \int_{\beta n}^{\beta (n+1)} \psi^* \: d\lambda
	\end{equation}
	clearly belongs to $\mathcal{M}(\mathcal{R}, \mu)$ and satisfies for every $f \in X$ that
	\begin{equation*}
    \begin{split}
		\int_0^{\infty} \varphi^* f^* \: d\lambda &= \sum_{\substack{n \in \mathbb{N} \\ (n+1) \beta \leq \mu(\mathcal{R})}} f^*(n) \int_{\beta n}^{\beta (n+1)} \psi^* \: d\lambda \\
        &= \int_0^{\infty} \psi^* f^* \: d\lambda \\
        &\leq \lVert \psi \rVert_{\left( \overline{X} \right)'} \lVert f^* \rVert_{\overline{X}} \\
        &= \lVert \psi \rVert_{\left( \overline{X} \right)'} \lVert f \rVert_{X}.
    \end{split}
	\end{equation*}
	As before, this shows both that the function $\varphi$ we have constructed satisfies $\varphi \in X'$ and that \eqref{ThmRepre(w')*:E0} holds.
\end{proof}

We note that the construction in \eqref{ThmRepre(w')*:E1} is the same as the one employed in \cite[Proof of Theorem~3.1]{MusilovaNekvinda24} to construct the representation functional for the case of completely atomic measure.

Next, we examine the convergence with respect to the $\newtopology$ topology. We will show, that although this convergence is weaker than the one with respect to the original quasinorm, it is strong enough to have some interesting consequences.

\begin{proposition} \label{Prop(w')*Implies*}
	Let $\lVert \cdot \rVert_X$ be an r.i.~quasi-Banach function norm satisfying \ref{P5} and let $X$ be the corresponding r.i.~quasi-Banach function space. Assume that a net $(f_{\iota})_{\iota}$ in $X$ satisfies
	\begin{align} \label{Prop(w')*Implies*:E1}
		f_{\iota} &\to 0 &\text{in } (X, \newtopology).
	\end{align}
	Then
	\begin{align*}
		f_{\iota}^*(t) &\to 0 & \text{for } t \in (0, \infty).
	\end{align*}
	Consequently, $f_{\iota} \to 0$ in measure.
\end{proposition}

\begin{proof}
	Assume the contrary, i.e.~that there is some $t \in (0, \infty)$ and some $\varepsilon > 0$  such that for every index $\iota$ there is some index $\kappa_{\iota} \geq \iota$ for which $f_{\kappa_{\iota}}^*(t) > \varepsilon$. Clearly, we can assume $t \leq \mu(\mathcal{R})$. Using this assumption (and $\sigma$-finiteness of $\mathcal{R}$), we may find some $E \subseteq \mathcal{R}$ such that $t \leq \mu(E) < \infty$. Then $\chi_E \in X'$ and we have for every $\iota$ that
	\begin{equation*}
		\int_0^{\infty} \chi_E^* f_{\kappa_{\iota}}^* \: d\lambda \geq \int_0^{t} f_{\kappa_{\iota}}^* \: d\lambda > t \varepsilon > 0,
	\end{equation*}
	which implies that \eqref{Prop(w')*Implies*:E1} fails.
	
	The remaining part is clear. 
\end{proof}

\begin{corollary} \label{CorEmbM0}
	Let $\lVert \cdot \rVert_X$ be an r.i.~quasi-Banach function norm satisfying \ref{P5} and let $X$ be the corresponding r.i.~quasi-Banach function space. Then $(X, \newtopology) \hookrightarrow (\mathcal{M}_0, \mu_{\textup{loc}})$.
\end{corollary}

The embedding $(X, \newtopology) \hookrightarrow (\mathcal{M}_0, \mu_{\textup{loc}})$ already sets the $\newtopology$ topology apart from the $w'$ topology. However, if we add absolute continuity of the rearrangement to the mix, we can prove more. 

\begin{proposition} \label{Prop(w')*DomConv}
	Let $\lVert \cdot \rVert_X$ be an r.i.~quasi-Banach function norm satisfying \ref{P5} and let $X$ be the corresponding r.i.~quasi-Banach function space. Fix $g \in X \cap \mathcal{M}_{(ACR)}$, $g \geq 0$. Consider a sequence $f_n \in \mathcal{M}_0$ such that $\lvert f_n \rvert \leq g$ $\mu$-a.e.~and that there is some function $f \in \mathcal{M}_0$ such that $f_n \to f$ $\mu$-a.e. Then all the functions $f_n, f \in X$ and
	\begin{align*}
		f_n &\to f &\text{in } (X, \newtopology).
	\end{align*}
\end{proposition}

\begin{proof}
	Fix $\varphi \in X'$. Since $f_n \leq g$ $\mu$-a.e., Proposition~\ref{PropACRDom} implies that
	\begin{align*}
		2\varphi^*(t) g^*(t) &\geq \varphi^*(t) (f - f_n)^*(t) \to 0 & \text{for } t \in (0, \infty).
	\end{align*}
	The conclusion thus follows by the classical Lebesgue dominated convergence theorem.
\end{proof}

Yet again, we single out the case of monotone convergence.

\begin{corollary} \label{CorMonotoneImplies(w')*}
	Let $\lVert \cdot \rVert_X$ be an r.i.~quasi-Banach function norm satisfying \ref{P5} and let $X$ be the corresponding r.i.~quasi-Banach function space. Fix $f \in X \cap \mathcal{M}_{(ACR)}$, $f \geq 0$. Then every sequence  $f_n \in X$ such that $0 \leq f_n \uparrow f$ $\mu$-a.e.~satisfies
	\begin{align*}
		f_n &\to f &\text{in } (X, \newtopology).
	\end{align*}
\end{corollary}

The importance of Proposition~\ref{Prop(w')*DomConv} and Corollary~\ref{CorMonotoneImplies(w')*} lies in the fact that it allows us to deduce from the \ACR property such statements for the $\newtopology$ topology, whose analogues for the quasinormed topology require the much stronger absolute continuity of the quasinorm. The first such statement is that the simple functions are sequentially dense in $(X, \newtopology)$.

\begin{corollary} \label{CorSimple(w')^*dense}
	Let $\lVert \cdot \rVert_X$ be an r.i.~quasi-Banach function norm satisfying \ref{P5} and let $X$ be the corresponding r.i.~quasi-Banach function space. Assume that $X$ has the \ACR property. Then the simple functions are sequentially dense in $(X, \newtopology)$.
\end{corollary}

\begin{proof}
	The statement follows from Proposition~\ref{Prop(w')*DomConv} once one recalls that any function $f \in X$ is $\mu$-a.e.~the pointwise limit of some sequence of simple functions $s_n$ such that
    \begin{equation*}
        \lvert s_n \rvert \leq \lvert f \rvert \in X \subseteq \mathcal{M}_{(ACR)}.
    \end{equation*}
\end{proof}

The second such statement is even more interesting. As it turns out, Corollary~\ref{CorMonotoneImplies(w')*} allows us to employ a technique similar to \cite[Chapter~1, Theorem~4.1]{BennettSharpley88} to show that the topological dual space of $(X, \newtopology)$ equals $X'$. This property is one of the two crucial components of our main ergodicity results.

\begin{theorem} \label{ThmACRDual}
	Let $\lVert \cdot \rVert_X$ be an r.i.~quasi-Banach function norm satisfying \ref{P5} and let $X$ be the corresponding r.i.~quasi-Banach function space. Assume that $X$ has the \ACR property. Then $(X, \newtopology)^* = X'$, in the sense that a linear functional $\alpha$ on $X$ is continuous with respect to the $\newtopology$ topology if and only if it can be represented by some uniquely determined function $\varphi \in X'$ via the formula
	\begin{align}
		\alpha(f) &= \int_{\mathcal{R}} \varphi f \: d\mu, &f \in X. \label{ThmACRDual:E1}
	\end{align}
\end{theorem}

\begin{proof}
	It is clear from the Hardy--Littlewood inequality (Theorem~\ref{THLI}) that every function $\varphi \in X'$ induces (via the formula \eqref{ThmACRDual:E1}) a continuous linear functional on $(X, \newtopology)$, so we focus on the converse.
	
	Fix $\alpha \in (X, \newtopology)^*$. Then, because $\{\lvert\cdot\rvert_\psi ; \; \psi\in X'\}$ is saturated, there are $\psi_{\alpha}\in X'$ and $C>0$ such that we have for every $f \in X$ that
	\begin{equation*}
		\lvert \alpha(f) \rvert \leq C \lvert f \rvert_{\psi_{\alpha}}.
	\end{equation*}		
	Fix further some sequence of pairwise disjoint sets $\mathcal{R}_n$ such that $\mu(\mathcal{R}_n) < \infty$ and $\mathcal{R} = \bigcup_n \mathcal{R}_n$. We now define for fixed $n$ and every measurable $E \subseteq \mathcal{R}_n$ the set-function $\mu_n$ given by the formula
	\begin{equation*}
		\mu_n(E) = \alpha(\chi_E).
 	\end{equation*}
 	Note that $\mu_n$ is well defined and finite for every appropriate $E$ since $\mu(\mathcal{R}_n) < \infty$. Furthermore, it is a complex measure. Indeed, the only property that requires verification is the $\sigma$-additivity, which follows from our assumption that $X$ has the \ACR property via Corollary~\ref{CorMonotoneImplies(w')*}. It is also clear that $\mu_n$ is absolutely continuous with respect to $\mu$, hence we may use the classical Radon--Nikod\'{y}m theorem (see e.g.~\cite[Theorem~6.10]{Rudin87}) to obtain a measurable function $\varphi_n \colon \mathcal{R}_n \to \mathbb{C}$ such that
 	\begin{align*}
 		\mu_n(E) &= \int_{E} \varphi_n \: d\mu &\text{for every measurable } E \subseteq \mathcal{R}_n.
 	\end{align*}
 	Since $\mathcal{R}_n$ are pairwise disjoint, we may now use the functions $\varphi_n$ in the obvious way to construct a measurable function $\varphi \colon \mathcal{R} \to \mathbb{C}$ satisfying
 	\begin{align*}
 		\int_{\mathcal{R}} \varphi \chi_E \: d\mu = \alpha(\chi_E) &\text{ for every measurable $E$ satisfying } E \subseteq \mathcal{R}_n \text{ for some $n$.}
 	\end{align*} 	
 	It remains to show that $\varphi \in X'$ and that \eqref{ThmACRDual:E1} holds for all $f \in X$ (instead of just for the characteristic functions of some sets). We begin with the former.
 	
 	The argument has to be performed separately for the real and imaginary part of $\varphi$. As it is the same in both cases, we will present only the former case, denoting the real part of $\varphi$ by $\varphi_{\operatorname{Re}}$.
 	
 	Let $g \in X$ be a non-negative simple function such that
 	\begin{align*}
 		\supp g &\subseteq \bigcup_{i = 0}^n \mathcal{R}_i & \text{for some } n \in \mathbb{N}.
 	\end{align*}
 	Then $\sgn(\varphi_{\operatorname{Re}}) g$ is also a simple function belonging to $X$ and having the same support. It follows that it can be written as a finite linear combination of functions for which \eqref{ThmACRDual:E1} holds, whence
 	\begin{equation} \label{ThmACRDual:E2}
 		\int_{\mathcal{R}} \lvert \varphi_{\operatorname{Re}}\rvert g \: d\mu = \int_{\mathcal{R}} \varphi_{\operatorname{Re}} \sgn(\varphi_{\operatorname{Re}}) g \: d\mu \leq \lvert g \rvert_{\psi_{\alpha}} \leq \lVert \psi_{\alpha} \rVert_{X'} \lVert g \rVert_X.
 	\end{equation}
 	
 	Let now $h \in X$ be arbitrary, then it is clear that there is some sequence of functions $g_k$ of the type considered in the previous step such that $g_k \uparrow \lvert h \rvert$ $\mu$-a.e. Applying monotone convergence theorem and the Fatou property \ref{P3} on the left-hand and right-hand sides of \eqref{ThmACRDual:E2}, respectively, we obtain 	
 	\begin{equation*}
 		\int_{\mathcal{R}} \lvert \varphi_{\operatorname{Re}} h \rvert\: d\mu \leq \lVert \psi_{\alpha} \rVert_{X'} \lVert h \rVert_X,
 	\end{equation*}
 	i.e.~$\varphi_{\operatorname{Re}} \in X'$ and $\lVert \varphi_{\operatorname{Re}} \rVert_{X'} \leq \lVert \psi_{\alpha} \rVert_{X'}$.
 	
 	Having established that $\varphi \in X'$, we now show that \eqref{ThmACRDual:E1} holds for all $f \in X$. 
 	 	
 	We first assume that $f$ is real-valued. As in the previous step, we consider a sequence of non-negative simple functions $g_k$ satisfying
 	\begin{align*}
 		\supp g_k &\subseteq \bigcup_{i = 0}^{n_k} \mathcal{R}_i & \text{for some } n_k \in \mathbb{N}
 	\end{align*}
 	and $g_k \uparrow \lvert f \rvert$ $\mu$-a.e. Putting $f_k = g_k \sgn(f)$, we observe that $f_k \to f$ $\mu$-a.e.~and, thanks to Proposition~\ref{Prop(w')*DomConv}, also in $(X, \newtopology)$ (here we again employ the \ACR property of $X$). Furthermore, $f_k$ is for every $k$ a linear combination of functions for which \eqref{ThmACRDual:E1} holds. Therefore, \eqref{ThmACRDual:E1} holds also for $f_k$. By applying the dominated convergence theorem on the right-hand side (with $\lvert \varphi f \rvert$ serving as the required majorant) and the continuity of $\alpha$ on the left-hand side we conclude that \eqref{ThmACRDual:E1} holds for $f$, too.
 	
    Finally, when $f \in X$ is arbitrary, then we have already proved that \eqref{ThmACRDual:E1} holds for both its real and imaginary parts, so the validity for $f$ follows.

    It remains only to consider the uniqueness of the representation, which follows immediately from the fact that $X$ contains the simple functions.
\end{proof}

\begin{corollary} \label{CorACRContinuity}
	Let $\lVert \cdot \rVert_X$ be an r.i.~quasi-Banach function norm satisfying \ref{P5} and let $X$ be the corresponding r.i.~quasi-Banach function space. Assume that $X$ has the \ACR property and an operator $T \colon (X, \newtopology) \to (X, \newtopology)$ is continuous. Then it is also continuous when considered as an operator $T \colon (X, w') \to (X, w')$ or, equivalently, there exists the associate operator $T'$ on $X'$.
	\end{corollary}

\begin{proof}
	It follows directly from Theorem~\ref{ThmACRDual} that the dual operator $T^* \colon (X, \newtopology)^* \to (X, \newtopology)^*$ can be interpreted as an operator on $X'$ and that it satisfies for every $\varphi \in X'$ that
	\begin{equation*}
		\int_{\mathcal{R}} \varphi T(f) \: d\mu = \int_{\mathcal{R}} T^*(\varphi) f \: d\mu.
	\end{equation*}
	Hence, $T^*$ is the associate operator of $T$ and the continuity $T \colon (X, w') \to (X, w')$ follows by Proposition~\ref{PropW'T'qBFS}. 
\end{proof}

The following corollary is rather immediate, however we believe it to be worth stating explicitly, as it will play an important role in the next section.

\begin{corollary} \label{CorRIEst}
	Let $\lVert \cdot \rVert_X$ be an r.i.~quasi-Banach function norm satisfying \ref{P5} and let $X$ be the corresponding r.i.~quasi-Banach function space. Assume that $X$ has the \ACR property. Assume that $T \colon (X, \newtopology) \to (X, \newtopology)$ is continuous.  Then 
	\begin{enumerate}
		\item \label{CorRIEst_i} If a linear operator $S$ defined on $X$ satisfies for every $f \in X$ that $(Sf)^* \leq (Tf)^*$ then $S \colon (X, \newtopology) \to (X, \newtopology)$ is continuous and there exists the associate operator $S'$ on $X'$.
		\item \label{CorRIEst_ii} If a linear operator $S$ defined on $X$ satisfies for every $f \in X$ and every $n \in \mathbb{N}$ that $(S^nf)^* \leq (Tf)^*$ then $S \colon (X, \newtopology) \to (X, \newtopology)$ is power-bounded.
	\end{enumerate}
	
	Furthermore, the conclusions still hold if we consider as the upper bound a continuous operator $T \colon \left( \overline{X}, \newtopology \right) \to \left( \overline{X}, \newtopology \right)$.
\end{corollary}

For the next two results, we restrict our setting to that of r.i.~Banach function spaces. This restriction is necessary in both cases. In the first one, this necessity is due to the fact that the result requires the equality $X = X''$ (see Theorem~\ref{TFA}).

\begin{theorem} \label{ThmOrtogonality}
	Let $X$ be an r.i.~Banach function space and $X'$ the corresponding associate space. Assume that both $X$ and $X'$ have the \ACR property. Let $T \colon (X, \newtopology) \to (X, \newtopology)$ be continuous and assume that the associate operator $T' \colon (X', \newtopology) \to (X', \newtopology)$ is also continuous. Then we have
	\begin{enumerate}
		\item \label{ThmOrtogonality_i} $(\operatorname{Ker} T)^{\circ} = \overline{\operatorname{Im} T'}^{\newtopology}$, $\left(\overline{\operatorname{Im} T}^{\newtopology} \right)^{\circ} = \operatorname{Ker} T'$
		\item \label{ThmOrtogonality_ii}$(\operatorname{Ker} T')^{\circ} = \overline{\operatorname{Im} T}^{\newtopology}$, $\left(\overline{\operatorname{Im} T'}^{\newtopology} \right)^{\circ} = \operatorname{Ker} T$.
	\end{enumerate}
	In the above, the (absolute) polar/annihilator symbol $( \cdot )^{\circ}$ refers to the dual pair $(X, X')$.
\end{theorem}

\begin{proof}
	By \cite[Lemma 23.31]{MeiseVogt97} we have $\left( \operatorname{Im} T' \right)^\circ = \operatorname{Ker} T$ as well as $\left( \operatorname{Im} T \right)^\circ = \operatorname{Ker} T'$, so that $\left( \operatorname{Ker} T \right)^\circ = \left( \operatorname{Im} T' \right)^{\circ\circ}$ as well as $\left( \overline{\operatorname{Im} T}^{\newtopology}  \right)^{\circ} = \operatorname{Ker} T'$. Additionally, the Bipolar Theorem \cite[Theorem~22.13]{MeiseVogt97} applied to the locally convex space $(X',\newtopology)$ gives  $(\operatorname{Ker} T)^{\circ} = \overline{\operatorname{Im} T'}^{\newtopology}$ which proves \ref{ThmOrtogonality_i}.
	
	From \ref{ThmOrtogonality_i} we conclude
	\begin{equation*}
	    (\operatorname{Ker} T')^{\circ} = \left(\overline{\operatorname{Im} T}^{\newtopology}\right)^{\circ\circ},\quad \left(\overline{\operatorname{Im} T'}^{\newtopology} \right)^{\circ} = \left(\operatorname{Ker} T\right)^{\circ\circ}.
	\end{equation*}
	\ref{ThmOrtogonality_ii} now follows from
	\begin{equation*}
	    \left(\overline{\operatorname{Im} T}^{\newtopology}\right)^{\circ\circ}=\overline{\operatorname{Im} T}^{\newtopology}, \quad \left(\operatorname{Ker} T\right)^{\circ\circ}=\overline{\operatorname{Ker} T}^{\newtopology}=\operatorname{Ker} T,
	\end{equation*}
	which holds by invoking the Bipolar Theorem once more (for the locally convex space $(X,\newtopology)$).
\end{proof}

We now come to the second of the two crucial properties of the $\newtopology$ topology, namely completeness. This property is a consequence of the embedding $(X, \newtopology) \hookrightarrow (\mathcal{M}_0, \mu_{\textup{loc}})$ established in Corollary~\ref{CorEmbM0} and we believe that it is interesting in its own right; recall that an infinite-dimensional normed space is never complete with respect to its (classical) weak topology (see e.g.~\cite[Theorem~3]{Kaplan1952} or \cite[Exercise~II.3]{Diestel1984}).

\begin{theorem} \label{Thm(w')*Complete}
    Let $X$ be an r.i.~Banach function space. Then it is $\newtopology$-complete.
\end{theorem}

\begin{proof}
	Let $(f_{\iota})_{\iota}$ be a Cauchy net in $(X, \newtopology)$. Then Corollary~\ref{CorEmbM0} implies that it is also a Cauchy net in $(\mathcal{M}_0, \mu_{\textup{loc}})$, which is a complete metric space (see Proposition~\ref{PropM0CompMetr}). This implies that there is some function $f \in \mathcal{M}_0$ such that $f_{\iota} \to f$ in $(\mathcal{M}_0, \mu_{\textup{loc}})$ (see \cite[Remark 22.18 (d)]{MeiseVogt97}). This is the natural candidate for the limit of $(f_{\iota})_{\iota}$, we need to show that it belongs to $X$ and that the net indeed converges in $(X, \newtopology)$.
	
	Since $X$ is assumed to be an r.i.~Banach function space, Proposition~\ref{PropLandauRes} together with Theorem~\ref{TFA} imply that $f \in X$ if and only if it satisfies
	\begin{equation*}
		\int_0^{\infty} \varphi^* f^* \: d\lambda < \infty
	\end{equation*}
	for all $\varphi \in X'$, so let us fix such a $\varphi$. The fact that $(\mathcal{M}_0, \mu_{\textup{loc}})$ is metrisable allows us to construct a sequence $\iota_n$ of indices for which it holds both that $f_{\iota_n} \to f$ in $(\mathcal{M}_0, \mu_{\textup{loc}})$ as $n \to \infty$ and that $(f_{\iota_n})_n$ is Cauchy with respect to the norm $\lvert \cdot \rvert_{\varphi}$. We stress that this sequence of indices depends on $\varphi$. By further passing to a subsequence if necessary, we may assume that $f_{\iota_n} \to f$ $\mu$-a.e. Whence, it follows from the properties of non-increasing rearrangement (see e.g.~\cite[Proposition~1.7]{BennettSharpley88}) and the classical Fatou's lemma that
	\begin{equation*}
		\int_0^{\infty} \varphi^* f^* \: d\lambda \leq \int_0^{\infty} \varphi^* \liminf_{n \to \infty} f_{\iota_n}^* \: d\lambda \leq \liminf_{n \to \infty} \int_0^{\infty} \varphi^* f_{\iota_n}^* \: d\lambda,
	\end{equation*}
	where the right-hand side is clearly finite since $f_{\iota_n}$ was constructed to be Cauchy with respect to the norm $\lvert \cdot \rvert_{\varphi}$. As $\varphi \in X'$ was arbitrary, we have established $f \in X$.
	
	Let us now move to the convergence in $(X, \newtopology)$, hence we fix $\varphi \in X'$ and $\varepsilon > 0$. Using the assumption that $(f_{\iota})_{\iota}$ is $\newtopology$-Cauchy, we now find some index $\iota_0$ such that it holds for every $\iota, \kappa \geq \iota_0$ that
	\begin{equation*} \label{Thm(w')*Complete:E1}
		\int_0^{\infty} \varphi^* (f_{\iota} - f_{\kappa})^* \: d\lambda < \varepsilon.
	\end{equation*}
	We now fix arbitrary $\iota \geq \iota_0$. Next, using the triangle inequality of $\lvert \cdot \rvert_{\varphi}$, we observe that 
	\begin{equation*}
		\int_0^{\infty} \varphi^* (f_{\iota} - f)^* \: d\lambda < \infty,
	\end{equation*}
	since both $f_{\iota}, f \in X$ and $\varphi \in X'$. Consequently, we may find $N, \delta \in (0, \infty)$ such that $N \geq \delta$, $N \in \operatorname{Im} \mu$, and
	\begin{align*}
		\int_N^{\infty} \varphi^* (f_{\iota} - f)^* \: d\lambda &< \varepsilon, \\
		\int_0^{\delta} \varphi^* (f_{\iota} - f)^* \: d\lambda &< \varepsilon. 
	\end{align*}
	Next, since we assume that $(\mathcal{R}, \mu)$ is resonant and that $N \in \operatorname{Im} \mu$, we know that $\varphi^* \chi_{[0, N)}$ is the non-increasing rearrangement of some function in $\mathcal{M}(\mathcal{R}, \mu)$ (either by the classical Sierpiński theorem when $(\mathcal{R}, \mu)$ is non-atomic, see e.g.~\cite[Chapter~2, Corollary~7.8]{BennettSharpley88}, or by the construction presented in \eqref{ThmRepre(w')*:E1} when it is completely atomic), which necessarily belongs to $X'$ thanks to Corollary~\ref{Corollary*<*=>||<||}, and so it follows from \eqref{DefResonant:E1} that we may find some $\psi_0 \in X'$ satisfying both $\psi_0^* = \varphi^* \chi_{[0, N)}$ and
	\begin{equation*}
		\int_0^{N} \varphi^* (f_{\iota} - f)^* \: d\lambda < \int_{\mathcal{R}} \lvert \psi_0 \rvert \lvert f_{\iota} - f \rvert \: d\mu + \varepsilon.
	\end{equation*}
	We then put $\mathcal{R}_0 = \supp \psi_0$ (i.e.~the support of an arbitrary but fixed representative of the equivalence class $\psi_0$) and note that $\mu(\mathcal{R}_0) = N$.	Finally, using the fact that $f_{\iota} \to f$ in $(\mathcal{M}_0, \mu_{\textup{loc}})$, we find some $\kappa \geq \iota_0$ such that the set $E_{\kappa}$ given by
	\begin{equation*}
		E_{\kappa} = \left \{ x \in \mathcal{R}_0; \; \lvert f_{\kappa}(x) - f(x) \rvert > \varepsilon \left( \int_0^{N} \varphi^* \: d\lambda \right)^{-1} \right \} 
	\end{equation*}
	satisfies $\mu(E_{\kappa}) < \delta$ (recall that $\varphi^* \in \left( \overline{X} \right)'$ which implies that it is locally integrable, see Proposition~\ref{PropRepreAS} and Theorem~\ref{TFA}).
	
	We are now suitably equipped to compute the required estimate (using also the Hardy--Littlewood inequality (Theorem~\ref{THLI}) multiple times, as well as the triangle inequality of $\lVert \cdot \rVert_{L^1}$, subadditivity of the elementary maximal function (Proposition~\ref{PropSubAdd**}), and the Hardy's lemma (Lemma~\ref{LemmaHardy})):
	\begin{equation*}
		\begin{split}
			\int_0^{\infty} \varphi^* (f_{\iota} - f)^* \: d\lambda &\leq \int_0^{N} \varphi^* (f_{\iota} - f)^* \: d\lambda + \int_N^{\infty} \varphi^* (f_{\iota} - f)^* \: d\lambda \\
			&\leq \int_{\mathcal{R}} \lvert \psi_0 \rvert \lvert f_{\iota} - f \rvert \: d\mu + 2\varepsilon \\
			&\leq \int_{\mathcal{R}} \lvert \psi_0 \rvert \lvert f_{\iota} - f_\kappa \rvert \: d\mu + \int_{\mathcal{R}} \lvert \psi_0 \rvert \lvert f_{\kappa} - f \rvert \: d\mu + 2\varepsilon \\
			&\leq \int_0^{N} \varphi^* (f_{\iota} - f_{\kappa})^* \: d\lambda + \int_{E_{\kappa}} \lvert \psi_0 \rvert \lvert f_{\kappa} - f \rvert \: d\mu +3\varepsilon \\
			&\leq \int_0^{\delta} \varphi^* (f_{\kappa} - f)^* \: d\lambda + 4\varepsilon \\
			&\leq \int_0^{\delta} \varphi^* (f_{\kappa} - f_{\iota})^* \: d\lambda + \int_0^{\delta} \varphi^* (f_{\iota} - f)^* \: d\lambda + 4\varepsilon \\
			&< 6 \varepsilon.
		\end{split}
	\end{equation*}
\end{proof}

In contrast to the previous theorem, as follows from our next result, the $\newtopology$-completeness never holds for r.i.~quasi-Banach function spaces that have the \ACR property but fail to be Banach function spaces. This justifies our claim that the restriction to r.i.~Banach function spaces in the previous theorem was necessary.

\begin{proposition}
	Let $\lVert \cdot \rVert_X$ be an r.i.~quasi-Banach function norm satisfying \ref{P5} and let $X$ be the corresponding r.i.~quasi-Banach function space. Assume that $X$ has the \ACR property. Then $X''$, the second associate space of $X$, also has the \ACR property and $X$ is $\newtopology$-dense in $X''$. Consequently, if $X \neq X''$, then it is not $\newtopology$-complete.
\end{proposition}

\begin{proof}
	We first recall that $X'$ is a Banach function space and thus $X''' = X'$ (see Theorem~\ref{TFA}), whence the norms defining the $\newtopology$ topology in $X$ are induced by the same functions as those on $X''$, i.e.~they are just restrictions. Furthermore, the fact that $X''$ also has the \ACR property follows from \cite[Theorem~4.16 and Theorem~4.26]{MusilovaNekvinda24}. Thus, Theorem~\ref{ThmACRDual} implies that $(X'', \newtopology)^* = X'$. However, $X$ contains the simple function, so the only function in $X'$ that induces the zero functional via the relation \eqref{ThmACRDual:E1} is the zero function, and so the classical Hahn--Banach theorem implies that $X$ is $\newtopology$-dense in $X''$.
	
	Finally, since the $\newtopology$ topologies on $X$ and $X''$ are both induced by $X'$, it follows that any net in $X$ that is $\newtopology$-convergent to a point in $X''$ is $\newtopology$-Cauchy in $X$, which shows that $X$ is not $\newtopology$-complete.
\end{proof}

The next result shows, that the $\newtopology$ topology is almost never metrisable---and when it is, then it is not interesting, as it coincides with the subspace topology inherited from $X''$.

\begin{theorem} \label{ThmEquivMetrNorm}
	Let $\lVert \cdot \rVert_X$ be an r.i.~quasi-Banach function norm satisfying \ref{P5} and let $X$ and $X'$ be, respectively, the corresponding r.i.~quasi-Banach function space and its associate space. Then the following statements are equivalent:
	\begin{enumerate}
        \item The topological space $(X, \newtopology)$ is first-countable. \label{ThmEquivMetrNorm_0}
		\item The topological space $(X, \newtopology)$ is metrizable. \label{ThmEquivMetrNorm_i}
		\item The topological space $(X, \newtopology)$ is normable. \label{ThmEquivMetrNorm_ii} 
		\item There is some $\psi \in X'$ such that $\lvert \cdot \rvert_{\psi} \approx \lVert \cdot \rVert_{X''}$ on $\mathcal{M}_0(\mathcal{R}, \mu)$. \label{ThmEquivMetrNorm_iii}
		\item There is some $\psi \in X'$ for which it holds that an arbitrary function $\varphi \in \mathcal{M}_0$ belongs to $X'$ if and only if there is some constant $C_{\varphi}$ for which it holds that $\varphi \prec C_{\varphi} \psi$.  \label{ThmEquivMetrNorm_iv}
	\end{enumerate}
\end{theorem}

\begin{proof}
	We start by showing that \ref{ThmEquivMetrNorm_0} implies \ref{ThmEquivMetrNorm_iv}. Assume that $(X, \newtopology)$ is first-countable, i.e.~that there is a countable basis of neighbourhoods of zero. Furthermore, $\{\lvert \cdot \rvert_{\varphi}; \; \varphi \in X'\}$ is a saturated family of norms (see Remark~\ref{RemPropPhi*Norm}) and so we may assume that there is a sequence $\psi_n \in X'$, $\lVert \psi_n \rVert_{X'} \leq 1$, and a corresponding sequence $\varepsilon_n \in (0, \infty)$ such that the sets
	\begin{equation} \label{ThmEquivMetrNorm:E1}
		\left \{ f \in X; \; \lvert f \rvert_{\psi_n} < \varepsilon_n \right \}
	\end{equation}
	form a basis of neighbourhoods of zero. Now, if we denote by $\lVert \cdot \rVert_{\overline{X}}$ the canonical representation quasinorm, by $\lVert \cdot \rVert_{\left( \overline{X} \right)'}$ the corresponding associate norm, and finally by $\left( \overline{X} \right)'$ the corresponding associate space, then Proposition~\ref{PropRepreAS} ensures that $\lVert \psi_n^* \rVert_{\left(\overline{X}\right)'} \leq 1$ for every $n$ and thus the function $\widetilde{\psi} \in \mathcal{M}([0, \infty), \lambda)$ given by
	\begin{equation*}
		\widetilde{\psi} = \sum_{n=0}^{\infty} 2^{-n} \psi_n^*
	\end{equation*}
	satisfies $\widetilde{\psi} \in \left(\overline{X}\right)'$ (as $\left(\overline{X}\right)'$ is a Banach space, see Theorem~\ref{TFA}). Theorem~\ref{ThmRepre(w')*} then yields $\psi \in X'$ such that we have for every $f \in X$ that
	\begin{equation*}
		\lvert f \rvert_{\psi} = \sum_{n=0}^{\infty} 2^{-n} \int_0^{\infty} \psi_n^* f^* \: d\lambda.
	\end{equation*}
	Now, $\lvert \cdot \rvert_{\psi}$ is of course a norm (as observed in Remark~\ref{RemPropPhi*Norm}) and since \eqref{ThmEquivMetrNorm:E1} is a basis of neighbourhoods of zero, it follows that the topology induced by this norm is equivalent to the $\newtopology$ topology. (This proves \ref{ThmEquivMetrNorm_ii}, but we need to work a little bit further to obtain \ref{ThmEquivMetrNorm_iv}.) Consequently, the norms $\lvert \cdot \rvert_{\varphi}$, $\varphi \in X'$, are continuous with respect to the norm $\lvert \cdot \rvert_{\psi}$ (since they are continuous with respect to the $\newtopology$ topology), i.e.~we have the estimates
	\begin{align*} 
		\lvert \cdot \rvert_{\varphi} &\lesssim \lvert \cdot \rvert_{\psi} &\text{on } X, \text{ for } \varphi \in X'.
	\end{align*}
	By testing this estimate by characteristic functions of measure $t$, we get that for every $\varphi \in X'$ there is some constant $C_{\varphi}$ such that we have
	\begin{align*}
		\int_0^{t} \varphi^* \: d\lambda &\leq C_{\varphi} \int_0^{t} \psi^* \: d\lambda &\text{for } t \text{ in the range of } \mu.
	\end{align*}
	Moreover, when $\mu(\mathcal{R}) < \infty$ then the integrals on both sides are, as a functions of $t$, constant on $(\mu(\mathcal{R}), \infty)$ and when $(\mathcal{R}, \mu)$ is completely atomic then the same functions are piecewise linear (as the rearrangements are piecewise constant). Hence, this estimate extends to the whole of $[0, \infty)$, which means that we have
	\begin{align*} \label{ThmEquivMetrNorm:E2}
		\varphi &\prec C_{\varphi} \psi &\text{for every } \varphi \in X'.
	\end{align*}
	Since the Hardy--Littlewood--P\'{o}lya principle holds for $X'$ (because it is an r.i.~Banach function space, see Theorem~\ref{TFA} and e.g.~\cite[Chapter~2, Theorem~4.6]{BennettSharpley88}), this establishes \ref{ThmEquivMetrNorm_iv}.
	
	Let us now show that \ref{ThmEquivMetrNorm_iv} implies \ref{ThmEquivMetrNorm_iii}. Since $X'$ is non-trivial (see Theorem~\ref{TFA}), we observe that $\psi \neq 0$, where $\psi$ is the function from our assumption \ref{ThmEquivMetrNorm_iv}. Thus, it is not difficult to observe that $\lvert \cdot \rvert_{\psi}$ is in fact an r.i.~Banach function norm. Indeed, it has already been observed in Remark~\ref{RemPropPhi*Norm} that it is a norm, the properties \ref{P2} and \ref{P3} are direct consequences of the properties of non-increasing rearrangement (see e.g.~\cite[Chapter~2, Proposition~1.7]{BennettSharpley88}), \ref{P4} follows from the property \ref{P5} of $\lVert \cdot \rVert_{X'}$ (see Theorem~\ref{TFA}), and finally \ref{P5} can be shown via a rather standard argument, see e.g.~\cite[Proposition~3.2]{PesaRepreACqN}. Hence, recalling Theorem~\ref{TEQBFS} and Proposition~\ref{PropLandauRes} makes it clear that in order to show that $\lvert \cdot \rvert_{\psi} \approx \lVert \cdot \rVert_{X''}$ it suffices to show that we have for every function $f \in \mathcal{M}_0(\mathcal{R}, \mu)$ that
	\begin{equation*}
		\int_0^{\infty} \psi^* f^* \: d\lambda < \infty
	\end{equation*}
	implies
	\begin{align*}
		\int_0^{\infty} \varphi^* f^* \: d\lambda &< \infty &\text{for all } \varphi \in X'.
	\end{align*}
	However, this follows by combining the Hardy's lemma (Lemma~\ref{LemmaHardy}) with our assumption that $\varphi \prec C_{\varphi} \psi$ for every $\varphi \in X'$.
	
	That \ref{ThmEquivMetrNorm_iii} implies \ref{ThmEquivMetrNorm_ii} is rather clear, since \eqref{RemPropPhi*Norm:E2} shows that the subspace topology (on $X$) inherited from $(X'', \lVert \cdot \rVert_{X''})$ is stronger than $\newtopology$ (see also Proposition~\ref{PESSAS}). Finally, it is trivial that $\ref{ThmEquivMetrNorm_ii} \implies \ref{ThmEquivMetrNorm_i} \implies \ref{ThmEquivMetrNorm_0}$.
\end{proof}

Finally, let us show that the condition imposed on $X'$ in \ref{ThmEquivMetrNorm_iv} is rather restrictive.

\begin{proposition}
	Let $X$ be an r.i.~Banach function space such that there is some $g \in X$ for which it holds that an arbitrary function $f \in \mathcal{M}_0$ belongs to $X$ if and only if there is some constant $C_{f}$ for which it holds that $f \prec C_{f} g$. Then there exists a quasiconcave function $\Phi$ such that $M_{\Phi}$, the (strong) Marcinkiewicz endpoint space corresponding to $\Phi$ (considered over $([0,\infty), \lambda)$), is a representation space for $X$, up to equivalence of norms.
\end{proposition}

We stress that we do not claim that $M_{\Phi}$ will be the canonical representation space. Indeed, the space constructed in the proof is in some cases distinct from it; we shall discuss this in more detail after the proof.

\begin{proof}
	Let $g$ be as in the statement of the theorem and consider the function 
	\begin{align*}
		\Phi(t) &= \begin{cases}
			\frac{t}{\int_0^{t} g^* \: d\lambda}, &\text{for } t \in (0, \infty), \\
			0, &\text{for } t = 0.
		\end{cases}
	\end{align*}
	Then $\Phi$ is clearly quasiconcave and a simple calculation shows that we have for $f \in \mathcal{M}(\mathcal{R}, \mu)$
	\begin{equation*}
		\sup_{t \in [0, \infty)} \Phi(t) f^{**}(t) < \infty \iff f \prec C_{f} g \text{ for some } C_{f} \in (0, \infty).
	\end{equation*}
	Consequently, $f^* \in M_{\Phi} \iff f \in X$. That this fact implies $\lVert f^* \rVert_{\M_{\Phi}} \approx \lVert f \rVert_X$ for every $f \in \mathcal{M}(\mathcal{R}, \mu)$ is then proved by applying two times the standard approach presented e.g.~in \cite[Chapter~1, Proof of Theorem~1.8]{BennettSharpley88} or \cite[Proof of Theorem~3.9]{NekvindaPesa24} (for the estimate ``$\gtrsim$'' one needs to use also the properties of the elementary maximal function, notably Proposition~\ref{PropSubAdd**}); we leave the details to the reader.
\end{proof}

Let us now provide the promised detail about the relationship of $M_{\Phi}$ and $\overline{X}$, the canonical representation space of $X$. As it turns out, they coincide (up to equivalence of norms) if and only if $(\mathcal{R}, \mu)$ is non-atomic and of infinite measure (i.e.~when the representation is unique). Indeed, it is not difficult to verify that for an atomic measure space $(\mathcal{R}, \mu)$, the functions $\overline{f} \in M_{\Phi}$ (which are defined on $([0, \infty), \lambda)$) satisfy $\overline{f}^* \chi_{[0,1]} \in L^{\infty}$, while in the case of finite measure the same functions $\overline{f}$ satisfy $\overline{f}^* \chi_{(1, \infty)} \in L^1$; in both cases, those are much stronger conditions than those applying to functions in $\overline{X}$.

Let us further note that this result is far from optimal, we just do not wish to spend time proving something that is not really necessary for our purposes. Namely, the result also holds for r.i.~quasi-Banach function spaces, one only needs to show that the assumptions guarantee that $X$ satisfies \ref{P5}, which can be done as in \cite[Lemma~2.24]{Pesa22}. The converse implication also holds, but only for quasiconcave functions $\Phi$ that satisfy $\lim_{t \to 0_+} t \Phi^{-1}(t) = 0$, as then the least concave majorant of $t \Phi^{-1}(t)$ (see e.g.~\cite[Chapter~2, Proposition~5.10]{BennettSharpley88}) will be equal to an integral of a non-increasing function and we will be able to obtain the desired $g \in X$ via the construction used in \cite[Theorem~3.1]{MusilovaNekvinda24} (to show that this function has the desired properties one employs an argument similar to that of \cite[Proposition~3.3]{PesaRepreACqN}).

\section{Mean ergodicity of composition operators on r.i.~Banach function spaces} \label{SectionErgodic}

In this section we continue working under the standing assumption that the underlying measure space $(\mathcal{R}, \mu)$ is $\sigma$-finite and resonant and further restrict ourselves to the case of r.i.~Banach function spaces. Our motivation for this restriction is natural: We have characterised the symbols $\phi$ for which the operator $T_{\phi}$ is power-bounded on every r.i.~quasi-Banach function space and it is well known that power-bounded operators are mean ergodic on reflexive Banach spaces. However, r.i.~Banach function spaces have the property that $X = X''$ (see Theorem~\ref{TFA}), which can often be used in lieu of reflexivity, and we are working with a specific type of power-bounded operators. Hence, the question appears naturally whether we could employ those facts to obtain some form of mean ergodicity even for the non-reflexive r.i.~Banach function spaces. As will be shown in this section, the answer is affirmative: by using the $\newtopology$ topology, we obtain several results that provide different forms of mean ergodicity, always with much weaker assumptions than that of reflexivity.

However, our first observation in this section holds for arbitrary r.i.~quasi-Banach function spaces over $([0,\infty),\lambda)$.

\begin{proposition} \label{PropDil(w')*Cont}
	Let $\lVert \cdot \rVert_X$ be an r.i.~quasi-Banach function norm on $\mathcal{M}([0, \infty), \lambda)$ satisfying \ref{P5} and let $X$ be the corresponding r.i.~quasi-Banach function space. Then the dilation operator $D_t \colon  (X, \newtopology) \to (X, \newtopology)$ is continuous for every $t \in (0, \infty)$.
\end{proposition}

\begin{proof}
	Let $f \in X$ and $\varphi \in X'$. Then the fact that the dilation operator commutes with non-increasing rearrangement and a simple change of variables show, that
	\begin{equation*}
		\int_0^{\infty} (D_tf)^* \varphi^* \: d\lambda = t^{-1} \int_0^{\infty} f^* (D_{t^{-1}}\varphi)^* \: d\lambda,
	\end{equation*}
	where $D_{t^{-1}}\varphi \in X'$ as follows from Theorem~\ref{TDRIS}.
\end{proof}

Next step is to use the uniform estimate by the dilation operator, obtained from the power-measure-boundedness assumption on the symbol in Lemma~\ref{LemmaDilPowerEst}, to obtain $\newtopology$-power-boundedness for both $T_{\phi}$ and $T_{\phi}'$. To avoid technicalities that are irelevant to our main results, we assume $X$ to be an r.i.~Banach function space.

\begin{proposition} \label{PropTsigma(w')*Bounded}
	Let $\lVert \cdot \rVert_X$ be an r.i.~Banach function norm and let $X$ be the corresponding r.i.~Banach function space. Assume that $X$ has the \ACR property. Then
	\begin{enumerate}
		\item If $\phi \colon \mathcal{R} \to \mathcal{R}$ is measure-bounded, then $T_{\phi} \colon  (X, \newtopology) \to (X, \newtopology)$ is continuous and the associate operator $T_{\phi}'$ both exists and is continuous as an operator $T_{\phi}' \colon (X', \newtopology) \to (X', \newtopology)$.
		\item If $\phi \colon \mathcal{R} \to \mathcal{R}$ is power-measure-bounded, then $T_{\phi} \colon  (X, \newtopology) \to (X, \newtopology)$ is power-bounded. Furthermore, the associate operator $T_{\phi}' \colon (X', \newtopology) \to (X', \newtopology)$ is in this case power-bounded as well.
	\end{enumerate}
\end{proposition}

\begin{proof}
	The continuity and power-boundedness of $T_{\phi} \colon  (X, \newtopology) \to (X, \newtopology)$ follow by combining Lemmata~\ref{LemmaDilEst} and \ref{LemmaDilPowerEst}, respectively, with Proposition~\ref{PropDil(w')*Cont} and Corollary~\ref{CorRIEst}. The same is true for the existence of the associate operator $T' \colon X' \to X'$.
	
	As for the power-boundedness of $T_{\phi}' \colon (X', \newtopology) \to (X', \newtopology)$ under the appropriate assumptions, we begin by computing (for arbitrary $f \in X'' = X$ and $\varphi \in X'$)
	\begin{equation} \label{PropTsigma(w')*Bounded:E1}
		\begin{split}
			\int_0^{\infty} ((T_{\phi}^n)' \varphi)^* f^* \: d\lambda &= \sup_{g^* = f^*} \left \lvert \int_{\mathcal{R}} (T_{\phi}^n)' ( \varphi ) g \: d\mu \right \rvert \\
			&= \sup_{g^* = f^*} \left \rvert \int_{\mathcal{R}} \varphi T_{\phi}^n (g) \: d\mu \right \rvert \\ 
			&\leq \sup_{g^* = f^*} \int_0^{\infty} \varphi^* (T_{\phi}^n g)^* \: d\lambda \\
			&\leq \int_0^{\infty} \varphi^* D_{A^{-1}} f^* \: d\lambda,
		\end{split}
	\end{equation}
	where we use (in this order) the definition of resonant measure spaces (see Definition~\ref{DefResonant} and Remark~\ref{RemResonant}), the definition of the associate operator (Definition~\ref{DefAssocOperator}), the Hardy--Littlewood inequality (Theorem~\ref{THLI}), and Lemma~\ref{LemmaDilPowerEst}.
	
	Now, $f \in X = (X')'$, and thus Proposition~\ref{PropRepreAS} shows that $f^* \in \left( \overline{X'} \right)'$. Theorem~\ref{TDRIS} then ensures that also $D_{A^{-1}} f^* \in \left( \overline{X'} \right)'$, whence the conclusion follows via Theorem~\ref{ThmRepre(w')*}.
\end{proof}

The computation \eqref{PropTsigma(w')*Bounded:E1} uses the assumption that $X$ is an r.i.~Banach function space (and thus $X = X''$) so that it can work with $T_{\phi}^n$ from the second line on. Without this assumption, one would have to work with $(T_{\phi}^n)'' \colon X'' \to X''$ which might in principle be distinct from $T_{\phi}^n$ on $X'' \setminus X$. The assumption that $\phi$ is power-measure-bounded prevents this behaviour, but the argument would significantly complicate the proof with no real benefit, as we do not require the stronger result. 

We now arrive to the first of our main results. The properties of $T_{\phi}$ with respect to the $\newtopology$ topology together with the properties of the topology itself allow us to use an argument similar to the classical one from \cite{Lorch39} to show that $T_{\phi}$ is mean ergodic on $(X, \newtopology)$ for power-measure-bounded $\phi$, under the rather weak additional assumption that both $X$ and $X'$ have the \ACR property. We stress once more that this assumption is significantly weaker than reflexivity of $X$, which is equivalent to both $X$ and $X'$ having absolutely continuous norms (see Corollary~\ref{CorReflexBFS}). We also recall that our assumptions on the underlying measure space is just $\sigma$-finiteness and resonance; we make no assumption on the finiteness of the measure.

\begin{theorem} \label{ThmErgodic(w')*}
	Let $X$ be an r.i.~Banach function space and let $X'$ be the corresponding associate space. Assume that both $X$ and $X'$ have the \ACR property. Then for every power-measure-bounded mapping $\phi \colon \mathcal{R} \to \mathcal{R}$ the operator $T_{\phi}$ is mean ergodic on $(X, \newtopology)$, i.e.~there is some continuous operator $T \colon (X, \newtopology) \to (X, \newtopology)$ such that we have for every $f \in X$ that
	\begin{align} \label{ThmErgodic(w')*:E0}
		\frac{1}{n}\sum_{i = 0}^{n-1} T_{\phi}^i f &\to Tf & \text{in } (X, \newtopology)  \text{ as } n \to \infty.
	\end{align}
	
	Furthermore, the above defined operator $T$ is positive and continuous as $T \colon (X, w') \to (X, w')$ and $T \colon (X, \lVert \cdot \rVert_X) \to (X, \lVert \cdot \rVert_X)$. Similarly, the associate operator $T'$ on $X'$ is also continuous as an operator $T' \colon (X', \lVert \cdot \rVert_{X'}) \to (X', \lVert \cdot \rVert_{X'})$, $T' \colon (X', \newtopology) \to (X', \newtopology)$, and $T' \colon (X', w') \to (X', w')$ and it is also the $\newtopology$-ergodic-limit of $T_{\phi}'$, the associate operator of $T_{\phi}$; i.e.~it satisfies for every $\varphi \in X'$ that
	\begin{align} \label{ThmErgodic(w')*:E0.1}
		\frac{1}{n}\sum_{i = 0}^{n-1} (T_{\phi}')^i \varphi &\to T'\varphi & \text{in } (X', \newtopology)  \text{ as } n \to \infty.
	\end{align}
	Finally, the operators $T$ and $T'$ are projections onto the spaces $\operatorname{Ker} (I - T_{\phi})$ and $\operatorname{Ker} (I - T_{\phi}')$, respectively; the operators $I - T$ and $I - T'$ are projections onto $\overline{\operatorname{Im} (I - T_{\phi})}^{\newtopology}$ and $\overline{\operatorname{Im} (I - T_{\phi}')}^{\newtopology}$, respectively; and we have
	\begin{align}
		X &= \operatorname{Ker} (I - T_{\phi}) \bigoplus \overline{\operatorname{Im} (I - T_{\phi})}^{\newtopology}, \label{ThmErgodic(w')*:E0.2}\\
		X'&= \operatorname{Ker} (I - T_{\phi}') \bigoplus \overline{\operatorname{Im} (I - T_{\phi}')}^{\newtopology}. \nonumber
	\end{align}
\end{theorem}

\begin{proof}
	Assume first that $f \in \operatorname{Ker} (I - T_{\phi})$. Then it is obvious that 
	\begin{align*}
		\frac{1}{n}\sum_{i = 0}^{n-1} T_{\phi}^i f &= f &\text{for } n \in \mathbb{N}.
	\end{align*}
	
	Assume now that $f \in \operatorname{Im} (I - T_{\phi})$, i.e.~that there is some $g \in X$ such that $f = (I - T_{\phi})(g)$. Since $T_{\phi} \colon  (X, \newtopology) \to (X, \newtopology)$ is power-bounded (see Proposition~\ref{PropTsigma(w')*Bounded}), it follows via a simple computation that we have for every $\varphi \in X'$ that
	\begin{align} \label{ThmErgodic(w')*:E1}
		\left \lvert \frac{1}{n}\sum_{i = 0}^{n-1} T_{\phi}^i f \right \rvert_{\varphi} &\leq \frac{1}{n} ( \left \lvert g \right \rvert_{\varphi} + \left \lvert T_{\phi}^{n}g \right \rvert_{\varphi}) \to 0 &\text{as } n \to \infty.
	\end{align}
	
	Now, the power-boundedness of $T_{\phi} \colon  (X, \newtopology) \to (X, \newtopology)$ implies that the sequence of operators
	\begin{align} \label{ThmErgodic(w')*:E1.1}
		\frac{1}{n}\sum_{i = 0}^{n-1} T_{\phi}^i,& &n \in \mathbb{N},
	\end{align}
	is equicontinuous on $(X, \newtopology)$. A standard argument thus shows that \eqref{ThmErgodic(w')*:E1} implies
	\begin{align} \label{ThmErgodic(w')*:E2}
		\frac{1}{n}\sum_{i = 0}^{n-1} T_{\phi}^i f &\to 0 & \text{in } (X, \newtopology)  \text{ as } n \to \infty
	\end{align}
	for every $f \in \overline{\operatorname{Im} (I - T_{\phi})}^{\newtopology}$.
	
	Note, that both $\operatorname{Ker} (I - T_{\phi})$ and $\overline{\operatorname{Im} (I - T_{\phi})}^{\newtopology}$ are $\newtopology$-closed linear subspaces of $X$ and that \eqref{ThmErgodic(w')*:E2} implies that 
	\begin{equation*}
		\operatorname{Ker} (I - T_{\phi}) \cap \overline{\operatorname{Im} (I - T_{\phi})}^{\newtopology} = \{0\}.
	\end{equation*} 
	We have thus established that the sequence of operators given by \eqref{ThmErgodic(w')*:E1.1} converges pointwise (with respect to the $\newtopology$ topology) on the direct sum of these two subspaces. Hence, it remains to show the this direct sum is $\newtopology$-dense in $X$, i.e.~that
	\begin{equation} \label{ThmErgodic(w')*:E3}
		\overline{\operatorname{Ker} (I - T_{\phi}) \bigoplus \overline{\operatorname{Im} (I - T_{\phi})}^{\newtopology}}^{\newtopology} = X,
	\end{equation}
	and use again the fact that the sequence is equicontinuous to extend the convergence to the whole $X$ via the same standard argument we have used above. We note that this extension, unlike that to $\overline{\operatorname{Im} (I - T_{\phi})}^{\newtopology}$, requires $X$ to be $\newtopology$-complete, which we have established in Theorem~\ref{Thm(w')*Complete}.
	
	Consider thus the associate operator $T_{\phi}' \colon (X', \newtopology) \to (X', \newtopology)$.  Since clearly $(I - T_{\phi}') =  (I - T_{\phi})'$, Theorem~\ref{ThmOrtogonality} implies that
	\begin{equation*}
		\begin{split}
			\left( \operatorname{Ker} (I - T_{\phi}) \bigoplus \overline{\operatorname{Im} (I - T_{\phi})}^{\newtopology} \right)^{\circ} &= (\operatorname{Ker} (I - T_{\phi}))^{\circ} \cap \left(\overline{\operatorname{Im} (I - T_{\phi})}^{\newtopology} \right)^{\circ} \\
			&= \overline{\operatorname{Im} (I - T_{\phi}')}^{\newtopology} \cap \operatorname{Ker} (I - T_{\phi}').
		\end{split}
	\end{equation*}	
	However, since $T_{\phi}' \colon (X', \newtopology) \to (X', \newtopology)$ is also power-bounded (see Proposition~\ref{PropTsigma(w')*Bounded}), it follows that all the above established facts about $T_{\phi}$ are also valid for $T_{\phi}'$ (recall that $X'' = X$ by Theorem~\ref{TFA}), specifically we have
	\begin{equation*}
		\operatorname{Ker} (I - T_{\phi}') \cap \overline{\operatorname{Im} (I - T_{\phi}')}^{\newtopology} = \{0\}.
	\end{equation*}
	Finally, we observe that the left-hand side of \eqref{ThmErgodic(w')*:E3} is a $\newtopology$-closed linear subspace of $X$ and that we have just proved that the only functional in $(X, \newtopology)^* = X'$ (see Theorem~\ref{ThmACRDual}) that vanishes on the whole of this subspace is the zero functional. Hence, the equality in \eqref{ThmErgodic(w')*:E3} follows from the classical Hahn--Banach theorem.
	
	We have thus shown that the sequence \eqref{ThmErgodic(w')*:E1.1} converges pointwise on $X$ with respect to the $\newtopology$ topology. It is evident that the pointwise limit is a linear operator which we denote $T$. That $T \colon (X, \newtopology) \to (X, \newtopology)$ is continuous follows from the equicontinuity of the sequence \eqref{ThmErgodic(w')*:E1.1} via the usual argument. Since the operators $T_{\phi}^n$ are all positive and $\newtopology$-convergence implies existence of a $\mu$-a.e.~convergent subsequence (see Corollary~\ref{CorEmbM0}), it follows that $T$ is positive.
	
	We note that a virtually identical argument shows that the sequence of Ces\`aro means
	\begin{align} \label{ThmErgodic(w')*:E1.2}
		\frac{1}{n}\sum_{i = 0}^{n-1} (T_{\phi}')^i,& &n \in \mathbb{N},
	\end{align}
	converges pointwise on $X'$ with respect to the $\newtopology$ topology to \textit{some} linear continuous operator on $X'$.
	
	Now, Corollary~\ref{CorACRContinuity} and Proposition~\ref{PropT'toNormBFS} yield the continuity of $T$ with respect to the $w'$ and $\lVert \cdot \rVert_X$ topologies and the existence of the associate operator $T'$ on $X'$. Proposition~\ref{PropT'toNormBFS} and Proposition~\ref{PropW'T'BFS} then imply the continuity of $T'$ with respect to the $\lVert \cdot \rVert_{X'}$ and $w'$ topologies on $X'$. Finally, it is easy to verify via direct computation that the pointwise $\newtopology$-limit of the sequence \eqref{ThmErgodic(w')*:E1.2} satisfies the definition of the associate operator of $T$, i.e.~that it coincides with $T'$ (recall that $\newtopology$-convergence implies $w'$-convergence as noted in Remark~\ref{RemEmbeTopo}), from which it follows that $T'$ is also continuous with respect to the $\newtopology$ topology.
	
	It remains to show that the limit operators are the appropriate projections and that the direct sum in \eqref{ThmErgodic(w')*:E3} is not just dense but in fact covers the entire space. The argument is the same for the operator $T$ and space $X$ as it is for $T'$ and $X'$, so we perform only the former.
	
	It is clear from the construction that $T$ is equal to identity on $\operatorname{Ker} (I - T_{\phi})$ and constantly zero on $\overline{\operatorname{Im} (I - T_{\phi})}^{\newtopology}$. Let now $f \in X$ be arbitrary. Then by \eqref{ThmErgodic(w')*:E3} there is some net
	\begin{equation} \label{ThmErgodic(w')*:E4}
		f_{\iota} \in \operatorname{Ker} (I - T_{\phi}) \bigoplus \overline{\operatorname{Im} (I - T_{\phi})}^{\newtopology}
	\end{equation}
	such that $f_{\iota} \to f$ in $(X, \newtopology)$. Using \eqref{ThmErgodic(w')*:E4}, we find for each $\iota$ the unique functions $\widetilde{f_{\iota}} \in \operatorname{Ker} (I - T_{\phi})$ and $\widehat{f_{\iota}} \in \overline{\operatorname{Im} (I - T_{\phi})}^{\newtopology}$ such that $f_{\iota} = \widetilde{f_{\iota}} + \widehat{f_{\iota}}$. Then necessarily
	\begin{align*}
		\widetilde{f_{\iota}} = \widetilde{f_{\iota}} + 0 = T\left(\widetilde{f_{\iota}} \right) + T \left( \widehat{f_{\iota}} \right) = T(f_{\iota}) &\to Tf &\text{in } (X, \newtopology).
	\end{align*}
	Since $\operatorname{Ker} (I - T_{\phi})$ is closed in $(X, \newtopology)$, we obtain that $Tf \in \operatorname{Ker} (I - T_{\phi})$. It then follows that
	\begin{align*}
		\widehat{f_{\iota}} = f_{\iota} - \widetilde{f_{\iota}} &\to f - Tf &\text{in } (X, \newtopology).
	\end{align*}
	Again, $\overline{\operatorname{Im} (I - T_{\phi})}^{\newtopology}$ is closed in $(X, \newtopology)$ and so $f - Tf \in \overline{\operatorname{Im} (I - T_{\phi})}^{\newtopology}$. This proves \eqref{ThmErgodic(w')*:E0.2} and it follows that $T$ is a projection onto $\operatorname{Ker} (I - T_{\phi})$ while $I - T$ is a projection onto $\overline{\operatorname{Im} (I - T_{\phi})}^{\newtopology}$.
\end{proof}

\begin{remark}
    We recall that the convergence in \eqref{ThmErgodic(w')*:E0} implies convergence in measure, as per Proposition~\ref{Prop(w')*Implies*}.
\end{remark}

\begin{remark}
	The proof of Theorem~\ref{ThmErgodic(w')*} in principle holds for a wider class of operators than just the composition ones. Indeed, what we need is $\newtopology$-power-boundedness of both the original operator and its associate operator (which are not equivalent to each other). We chose to stick to the composition operators, as they are the topic of this paper and so are also the motivation for this result. Furthermore, they also naturally satisfy not only the assumptions required for this proof, but also for the proofs of the theorems below, which allows us to keep our assumptions simple.
\end{remark}

Let us now move on to the question of $\lVert \cdot \rVert_X$-convergence of the Ces\`aro means. The first step is to establish uniform estimates by the dilation operator for said means and their limit.

\begin{proposition} \label{PropHLPEstimate}
	Let $\phi \colon \mathcal{R} \to \mathcal{R}$ be a power-measure-bounded mapping and consider the composition operator $T_{\phi}$. Then there is some constant $B \in (0, 1]$ such that the sequence of Ces\`aro means satisfies
	\begin{align} 
		\left( \frac{1}{n}\sum_{i = 0}^{n-1} T_{\phi}^i f \right)^{*} &\prec D_{B} f^* &\text{for all } f \in \mathcal{M}(\mathcal{R}, \mu) \text{ and } n \in \mathbb{N}. \label{PropHLPEstimate:E1}
	\end{align}
	
	Furthermore, let $X$ be an r.i.~Banach function space and let $X'$ be the corresponding associate space such that both $X$ and $X'$ have the \ACR property and consider the operator $T$ defined as in \eqref{ThmErgodic(w')*:E0}. Then the operator $T$ satisfies (for the same constant $B$ as above)
	\begin{align}
		(Tf)^* &\prec D_{B} f^* &\text{for all } f \in X. \label{PropHLPEstimate:E2}
	\end{align}
\end{proposition}

Lemma~\ref{LemmaDilPowerEst} gives a uniform pointwise estimate for the rearrangements of $T_{\phi}^n$ by a dilation operator. We could use $\eqref{EqAlmostSubAdd*}$ to obtain similar pointwise estimates for the rearrangements of the Ces\`aro means, but they would no longer be uniform, in the sense that the dilation parameter would converge to $0$ as $n \to \infty$. If we want a uniform estimate, then we must replace the pointwise estimate for rearrangements by the the Hardy--Littlewood--P\'{o}lya relation (as presented in Definition~\ref{DefHLPR}). 

\begin{proof}
	We recall that our assumptions and Lemma~\ref{LemmaDilPowerEst} guarantee that we have for every function $f \in \mathcal{M}(\mathcal{R}, \mu)$ and every $i \in \mathbb{N}$, $i \geq 1$ that
	\begin{align*}
		(T^i_{\phi} f)^* &\leq D_{A^{-1}} f^* &\text{on $[0, \infty)$},
	\end{align*}
	where $A$ is the power-measure-bound of $\phi$. As $T_{\phi}^0$ is the identity operator, it is obvious that $(T^0_{\phi} f)^* = D_{1} f^*$. Whence, if we put $B = \min \{1, A^{-1}\}$ then we get for every $f \in \mathcal{M}(\mathcal{R}, \mu)$ and every $i \in \mathbb{N}$ that
	\begin{align*}
		(T^i_{\phi} f)^* &\leq D_{B} f^* &\text{on $[0, \infty)$}.
	\end{align*}	
	Combining this with Proposition~\ref{PropSubAdd**}, we observe that we also have for every $n \in \mathbb{N}$ that
	\begin{align*} 
		\left( \frac{1}{n}\sum_{i = 0}^{n-1} T_{\phi}^i f \right)^{**} &\leq \frac{1}{n}\sum_{i = 0}^{n-1} (T_{\phi}^i f)^{**} \leq (D_{B} f^*)^{**} &\text{on } [0, \infty),
	\end{align*}
	which establishes \eqref{PropHLPEstimate:E1}.
	
	Furthermore, Theorem~\ref{ThmErgodic(w')*} implies that our assumptions on $X$ and $X'$ guarantee that the Ces\`aro means converge to $Tf$ in $(X, \newtopology)$, so it follows from Proposition~\ref{Prop(w')*Implies*} that there is a subsequence converging $\mu$-a.e. Denoting its indices $n_k$, it follows from the properties of non-increasing rearrangement (see e.g.~\cite[Proposition~1.7]{BennettSharpley88}) and the classical Fatou's lemma that
	\begin{align*}
		(Tf)^{**} &\leq \liminf_{k \to \infty} \left( \frac{1}{n_k}\sum_{i = 0}^{n_k-1} T_{\phi}^i f \right)^{**} \leq (D_{B} f^*)^{**} &\text{on $[0, \infty)$},
	\end{align*}
	which establishes \eqref{PropHLPEstimate:E2}.
\end{proof}

It is worth noting that while the pointwise estimate for rearrangements yielded uniform absolute continuity of the quasinorms of $(T^n_{\phi}(f))^* \in \overline{X}$ (see Corollary~\ref{CorPresXa}), this weaker uniform estimate via the Hardy--Littlewood--P\'{o}lya relation does not. Indeed, the estimate is sufficient to cover sets $E_k = [0, k^{-1}]$, but it fails to provide useful information if we consider the sets $E_k = [k, \infty)$. However, since it is still strong enough to provide the estimate \eqref{PropHLPEstimate:E2},  we still get, thanks to Theorem~\ref{ThmInheritanceACqN}, that the limit operator $T$ preserves absolute continuity of the quasinorm.

\begin{corollary} \label{CorTPresXa}
    Let $\phi \colon \mathcal{R} \to \mathcal{R}$ be a power-measure-bounded mapping and consider the composition operator $T_{\phi}$. Let $X$ be an r.i.~Banach function space and let $X'$ be the corresponding associate space such that both $X$ and $X'$ have the \ACR property and consider the operator $T$ defined as in \eqref{ThmErgodic(w')*:E0}. Then $T(X_a) \subseteq X_a$.
\end{corollary}

\begin{proof}
    Let $f \in X_a$ and $B$ be as in Proposition~\ref{PropHLPEstimate}, then Theorem~\ref{ThmRepreACqN} implies that $f^* \in \overline{X}_a$ and so also $D_{B} f^* \in \overline{X}_a$ (by Lemma~\ref{LemDilPresXa}). Recalling that the Hardy--Littlewood--P\'{o}lya principle holds for $\lVert \cdot \rVert_X$ because it is an r.i.~Banach function norm, see e.g.~\cite[Chapter~2, Theorem~4.6]{BennettSharpley88}, and therefore also for $\lVert \cdot \rVert_{\overline{X}}$ by Proposition~\ref{PropRepreHLPP}, we thus obtain from Proposition~\ref{PropHLPEstimate} and Theorem~\ref{ThmInheritanceACqN} that $(Tf)^* \in \overline{X}_a$. Finally, applying again Theorem~\ref{ThmRepreACqN} yields $Tf \in X_a$.
\end{proof}

This is the key information that allows us to show that the Ces\`aro means converge to $T$ on $X_a$ not just in the $\newtopology$ topology, but also in the $\lVert \cdot \rVert_X$ topology.

\begin{theorem} \label{ThmErgodicNormLocalised}
	Let $X$ be an r.i.~Banach function space and let $X'$ be the corresponding associate space. Assume that both $X$ and $X'$ have the \ACR property. Let $\phi \colon \mathcal{R} \to \mathcal{R}$ be a power-measure-bounded mapping and consider the composition operator $T_{\phi}$ and the operator $T$ defined as in \eqref{ThmErgodic(w')*:E0}. Then we have for every $f \in X_a$ that $Tf \in X_a$ and that
	\begin{align}
		\frac{1}{n}\sum_{i = 0}^{n-1} T_{\phi}^i f &\to Tf & \text{in } (X, \lVert \cdot \rVert_X)  \text{ as } n \to \infty. \label{ThmErgodicNormLocalised:E0}
	\end{align}
\end{theorem}

We will need the following lemma:

\begin{lemma} \label{LemmaCapXaDense}
	Let $X$ be an r.i.~Banach function space and $X'$ the corresponding associate space. Assume that both $X$ and $X'$ have the \ACR property and that the simple functions have absolutely continuous norms in $X$. Let $\phi$ be measure-bounded. Then
	\begin{equation} \label{LemmaCapXaDense:E1}
		\overline{\operatorname{Im} (I - T_{\phi}) \cap X_a}^{\newtopology} = \overline{\operatorname{Im} (I - T_{\phi})}^{\newtopology}.
	\end{equation}
\end{lemma}

\begin{proof}
	That the left-hand side of \eqref{LemmaCapXaDense:E1} is included in the right-hand side is clear. Hence, we consider $f \notin \overline{\operatorname{Im} (I - T_{\phi}) \cap X_a}^{\newtopology}$. Recalling Theorem~\ref{ThmACRDual}, we may use the classical Hahn--Banach theorem to find $\varphi \in X'$ such that
	\begin{equation} \label{LemmaCapXaDense:E2}
		\int_{\mathcal{R}} \varphi f \: d\mu > 0, 
	\end{equation}
	while it holds for every $h \in \overline{\operatorname{Im} (I - T_{\phi}) \cap X_a}^{\newtopology}$ that
	\begin{equation*}
		\int_{\mathcal{R}} \varphi h \: d\mu = 0.
	\end{equation*}
	Consider now a simple function $g$. It is our assumption that $g \in X_a$, from which it follows via Corollary~\ref{CorPresXa} that $(I - T_{\phi})g \in \operatorname{Im} (I - T_{\phi}) \cap X_a$. Thence,
	\begin{equation*}
		0 = \int_{\mathcal{R}} \varphi (I - T_{\phi})(g) \: d\mu = \int_{\mathcal{R}} (I - T_{\phi}')(\varphi) g \: d\mu,
	\end{equation*}
	where $T_{\phi}'$ is the associate operator of $T_{\phi}$ that both exists and is continuous as an operator $T_{\phi}' \colon (X', \newtopology) \to (X', \newtopology)$ by Proposition~\ref{PropTsigma(w')*Bounded}. Since $g$ was an arbitrary simple function, it follows that $(I - T_{\phi}')(\varphi) = 0$ $\mu$-a.e., i.e.~that $\varphi \in \operatorname{Ker} (I - T_{\phi}')$. Combining this fact with \eqref{LemmaCapXaDense:E2}, we observe that
	\begin{equation*}
		f \notin \left( \operatorname{Ker} (I - T_{\phi}') \right)^{\circ} = \overline{\operatorname{Im} (I - T_{\phi})}^{\newtopology},
	\end{equation*}
	where the last equality follows from Theorem~\ref{ThmOrtogonality}.
\end{proof}

\begin{proof}[Proof of Theorem~\ref{ThmErgodicNormLocalised}]
	We first recall that we know from Corollary~\ref{CorTPresXa} that $Tf \in X_a$. Then we move on to \eqref{ThmErgodicNormLocalised:E0}. 
    
    In light of Theorem~\ref{ThmErgodic(w')*}, we only have to prove that the assumption $f \in X_a$ implies that the convergence to the limit operator $T$ in \eqref{ThmErgodic(w')*:E0} holds also with respect to the $\lVert \cdot \rVert_X$ topology. Considering Proposition~\ref{PropXaAlternative}, we see that either $X_a$ contains the simple functions or we have $X_a = \{0\}$. Since the validity of \eqref{ThmErgodicNormLocalised:E0} in the latter case is trivial, we focus on the former case, i.e.~we assume that $X_a$ contains the simple functions.
	
	Since $T_{\phi}$ is power-bounded on $(X, \lVert \cdot \rVert_X)$ (by Theorem~\ref{ThmPowerBoundedQBFS}), we may use virtually the same arguments as in the proof of Theorem~\ref{ThmErgodic(w')*} to show that we have for $g \in \operatorname{Ker} (I - T_{\phi})$ that
	\begin{align*}
		\frac{1}{n}\sum_{i = 0}^{n-1} T_{\phi}^i g & \to g = Tg &\text{in } (X, \lVert \cdot \rVert_X)  \text{ as } n \to \infty,
	\end{align*}
	while for $g \in \overline{\operatorname{Im} (I - T_{\phi})}^{\lVert \cdot \rVert_X}$ it holds that
	\begin{align*}
		\frac{1}{n}\sum_{i = 0}^{n-1} T_{\phi}^i g & \to 0 = Tg &\text{in } (X, \lVert \cdot \rVert_X)  \text{ as } n \to \infty,
	\end{align*}
	where we employ the fact that $(X, \lVert \cdot \rVert_X) \hookrightarrow (X, \newtopology)$ and thus the $\lVert \cdot \rVert_X$-closure is a subset of the $\newtopology$-closure.
	
	Recalling the decomposition \eqref{ThmErgodic(w')*:E0.2} and the fact that $T$ is a projection onto $\operatorname{Ker} (I - T_{\phi})$ while $I - T$ is a projection onto $\overline{\operatorname{Im} (I - T_{\phi})}^{\newtopology}$, it is clear that \eqref{ThmErgodicNormLocalised:E0} will follow once we show that 
	\begin{equation} \label{ThmErgodicNormLocalised:E1}
		(I - T)f \in \overline{\operatorname{Im} (I - T_{\phi})}^{\lVert \cdot \rVert_X}.
	\end{equation} 
	We will first need to prepare some tools, however.
	
	With a slight abuse of notation, we will denote by $(X_a, \lVert \cdot \rVert_X)$ and $(X_a, \newtopology)$ the linear space $X_a$ equipped with the subspace topology inherited from $(X, \lVert \cdot \rVert_X)$ and $(X, \newtopology)$, respectively. Proposition~\ref{PropXaOrdId} yields that $X_a$ is $\lVert \cdot \rVert_X$-closed in $X$, hence $(X_a, \lVert \cdot \rVert_X)$ is a Banach space, and Theorem~\ref{ThmACNDual} shows that $(X_a, \lVert \cdot \rVert_X)^* = X'$ (since we assume that $X_a$ contains the simple functions). On the other hand, the assumption that $X_a$ contains the simple functions also implies that $X_a$ is dense in $(X, \newtopology)$ (see Corollary~\ref{CorSimple(w')^*dense}), whence $(X, \newtopology)^* \subseteq (X_a, \newtopology)^*$, in the sense that there is a one-to-one correspondence between the functionals in $(X, \newtopology)^*$ and their restrictions in $(X_a, \newtopology)^*$. Putting this together, we obtain
    \begin{equation*}
        X' = (X, \newtopology)^* \subseteq (X_a, \newtopology)^* \subseteq (X_a, \lVert \cdot \rVert_X)^* = X',
    \end{equation*}
    where the first equality is due to Theorem~\ref{ThmACRDual} and the second inclusion is due to the $\lVert \cdot \rVert_X$ topology being finer than $\newtopology$ (see Remark~\ref{RemEmbeTopo}). Thus, we have shown that all the three duals can be identified with $X'$, from which we obtain via the classical Hahn--Banach theorem that the closures in $(X_a, \lVert \cdot \rVert_X)$ and $(X_a, \newtopology)$ coincide.
    
    Having prepared the necessary tools, we now return to \eqref{ThmErgodicNormLocalised:E1}. As $(I - T)f \in \overline{\operatorname{Im} (I - T_{\phi})}^{\newtopology}$, Lemma~\ref{LemmaCapXaDense} shows that there is a net $(g_{\iota})_{\iota}$ in $\operatorname{Im} (I - T_{\phi}) \cap X_a$ such that $g_{\iota} \to (I - T)f$ in $(X, \newtopology)$. As we also have that $(I - T)f \in X_a$ (since $Tf \in X_a$, as established above), it follows that $g_{\iota} \to (I - T)f$ in $(X_a, \newtopology)$, which means that $(I - T)f$ belongs to the closure of $\operatorname{Im} (I - T_{\phi}) \cap X_a$ in $(X_a, \newtopology)$. Since this closure coincides with the closure in $(X_a, \lVert \cdot \rVert_X)$, we get that there is a sequence of functions $(h_n)_{n\in\N}$ in $\operatorname{Im} (I - T_{\phi}) \cap X_a$ such that $h_n \to (I - T)f$ in $(X_a, \lVert \cdot \rVert_X)$. However, this clearly means that we also have $h_n \to (I - T)f$ in $(X, \lVert \cdot \rVert_X)$, which establishes \eqref{ThmErgodicNormLocalised:E1}.
\end{proof}

\begin{remark}
	Again, the proof of Theorem~\ref{ThmErgodicNormLocalised} is not limited to composition operators, however the assumptions start to pile up. Aside from the $\newtopology$-power-boundedness of both the original operator and its associate operator, which are inherited from Theorem~\ref{ThmErgodic(w')*}, we also require that the original operator is $\lVert \cdot \rVert_X$-power-bounded (which does not follow from $\newtopology$-power-boundedness), that it preserves the space $X_a$, and that the ergodic limit also preserves $X_a$.
\end{remark}

Let us now present the proof of Theorem~~\ref{ThmErgodicLite}, which follows quite easily from the already presented results.

\begin{proof}[Proof of Theorem~\ref{ThmErgodicLite}]
    Most of the claims of the theorem are already contained in Theorems~\ref{ThmErgodic(w')*} and \ref{ThmErgodicNormLocalised}, we only have to prove that
    \begin{equation*}
         \overline{\operatorname{Im} (I - T_{\phi})}^{\newtopology} =  \overline{\operatorname{Im} (I - T_{\phi})}^{\lVert \cdot \rVert_X}
    \end{equation*}
    whenever $X$ has absolutely continuous norm. However, this assumption implies that $(X, \lVert \cdot \rVert_X)^* = X'$ (see Theorem~\ref{ThmACNDual}) which together with the fact that also $(X, \newtopology)^* = X'$ (see Theorem~\ref{ThmACRDual}) yields the desired conclusion via the classical Hahn--Banach theorem.
\end{proof}

The following concrete example covers, in particular, translation operators on r.i.~Banach functions spaces over open subsets of euclidean spaces.

\begin{example}\label{example:C^1}
    Let $U\subset\R^d$ be open and $\phi:U\rightarrow U$ be $C^1$, one-to-one, and $\operatorname{det}J\phi(x)\neq 0$ for every $x\in U$, where by $J\phi$ we denote the Jacobian of $\phi$. Moreover, let $\mu=g\,d\lambda^d$, i.e.~$\mu$ is the Borel measure on $U$ which has the Lebesgue density $g\in\mathcal{M}_+(U,\lambda^d|_U)$. By the change-of-variables formula, it follows easily that for $n\in\N$ the image measure $\mu\left(\phi^{-n}(\cdot)\right)$ of $\mu$ under $\phi^n$ has Lebesgue density
    \begin{equation*}
        g_n(x):=\chi_{\phi^n(U)}(x)g\left(\phi^{-n}(x)\right)\left(\prod_{j=0}^{n-1}\lvert \operatorname{det}J\phi^{-1} \rvert\circ \phi^{-j}\right)(x).
    \end{equation*}

    Hence, for the special case $g\equiv 1$, i.e.~$\mu=\lambda^d|_U$, by Theorem \ref{ThmErgodicLite}, $T_\phi$ is mean ergodic on every r.i.~Banach function space $X$ over $(U,\lambda^d|_U)$ which has absolutely continuous norm and such that $X'$ has the \ACR whenever there is $C>0$ such that
    \begin{align*}
        \chi_{\phi^n(U)}\left(\prod_{j=0}^{n-1} \lvert \operatorname{det}J\phi^{-1} \rvert \circ \phi^{-j}\right) & \leq C &\lambda^d\text{-a.e.~for all }n\in\N.
    \end{align*}
    Specializing further to $\phi(x)=Ax+b$ for $A\in\operatorname{Gl}(\R^d)$ with $\lvert \operatorname{det}A \rvert \geq 1$, $b\in\R^d$, this holds for every $U$ with $\phi(U)\subset U$.
\end{example}

We also cover the shift operator on sequence spaces, as we shall see next.

\begin{example}\label{example:shifts}
    Let $\mathcal{R}\in\{\N,\Z\}$ and let $\mu$ be the counting measure over $\mathcal{R}$. Clearly, $\phi(n)=n+1$, $n\in\mathcal{R}$, is power-measure-bounded. Thus, by Theorem~\ref{ThmErgodicLite}, the composition operator $T_\phi$, called \emph{unilateral shift} or \emph{bilateral shift} in case $\mathcal{R}=\N$ or $\mathcal{R}=\Z$, respectively, is mean ergodic on every r.i.~Banach function space $X$ (of sequences) over $\N$ or $\Z$ which has absolutely continuous norm and for which $X'$ has the \ACR property; a notable non-reflexive examples are the Lorentz sequence spaces $\ell^{p,1}$, $p \in (1, \infty)$. Furthermore, other notable non-reflexive spaces are covered by Theorems~\ref{ThmErgodic(w')*} and \ref{ThmErgodicNormLocalised}, e.g.~the Lorentz sequence spaces $\ell^{p, \infty}$, $p \in (1, \infty)$, also known as the weak Lebesgue spaces. (For the precise definition of said spaces and a modern overview of their applications we suggest the recent paper \cite{DolezalovaVybiral20}.)
\end{example}

We close the section with two counterexamples showing that the assumption in Theorem~\ref{ThmErgodic(w')*} that both $X$ and $X'$ have the \ACR property is necessary in general. In order to avoid technicalities that would only obscure the core ideas, we present the counterexamples only for the case $(\mathcal{R}, \mu) = (\mathbb{R}, \lambda)$; we believe it to be quite trivial to modify the construction for the cases of $(\mathbb{R}^n, \lambda^n)$ or a completely atomic space of infinite measure (note that the infiniteness of the measure is crucial, as otherwise the \ACR property cannot fail). For arbitrary non-atomic measure spaces of infinite measure, the construction can be replicated as long as one can construct the appropriate measure-preserving mappings.

We start with the case of the \ACR condition not holding for $X'$. Informally speaking, this corresponds to $X$ being equivalent to $L^1$ near infinity.

\begin{example} \label{CounterExampleL1}
	Let $X$ be an r.i.~Banach function space over $(\mathbb{R}, \lambda)$ such that $X'$ does not have the \ACR property. By \cite[Theorem~4.16]{NekvindaPesa24}, this implies that $\chi_{\mathbb{R}} \in X'$.
		
	Now, consider the measure-preserving mapping $\phi \colon \mathbb{R} \to \mathbb{R}$, given for $t \in \mathbb{R}$ by
	\begin{equation*}
		\phi(t) = t-1,
	\end{equation*}
	and the function $f = \chi_{[0,1)} \in X$. It is evident that it holds for every $i \in \mathbb{N}$ that $T_{\phi}^i f = \chi_{[i, i+1)}$ and thus we have for every $n \in \mathbb{N}$ that
	\begin{equation*}
		\frac{1}{n}\sum_{i = 0}^{n-1} T_{\phi}^i f = \frac{1}{n} \chi_{[0, n)}.
	\end{equation*}
	As $\frac{1}{n} \chi_{[0, n)} \to 0$ in measure as $n \to \infty$, Proposition~\ref{Prop(w')*Implies*} tells us that the only candidate for the $\newtopology$-limit of the Ces\`aro means is $0$. However, we have observed above that $\chi_{\mathbb{R}} \in X'$ and of course we have
	\begin{equation}
		\left \lvert \frac{1}{n} \chi_{[0, n)} \right \rvert_{\chi_{\mathbb{R}}} = \frac{1}{n} \int_0^{\infty} \chi_{[0, n)}^* \: d\lambda = 1.
	\end{equation}
\end{example}

Next is the case of the \ACR condition not holding for $X$. Informally speaking, it corresponds to $X$ being equivalent to $L^{\infty}$ near infinity.

\begin{example} \label{CounterExampleLInfty}
	Let $X$ be an r.i.~Banach function space over $(\mathbb{R}, \lambda)$ that does not have the \ACR property. By \cite[Theorem~4.16]{NekvindaPesa24}, this implies that the function
	\begin{equation*}
		f = \chi_{[0, \infty)}
	\end{equation*}
	satisfies $f \in X$ (since it is equimeasurable with $\chi_{\mathbb{R}} \in X$).
	
	Now, consider the same measure-preserving mapping $\phi \colon \mathbb{R} \to \mathbb{R}$ as above, i.e.~$\phi$ is given for $t \in \mathbb{R}$ by
	\begin{equation*}
		\phi(t) = t-1.
	\end{equation*}
	Evidently $T_{\phi}^i f = \chi_{[i,\infty)}$ holds and thus we have for every $n \in \mathbb{N}$ that
	\begin{equation*}
		\frac{1}{n}\sum_{i = 0}^{n-1} T_{\phi}^i f = \chi_{[n-1, \infty)} + \sum_{i = 0}^{n-2} \frac{i+1}{n} \chi_{[i, i+1)}.
	\end{equation*}
	Hence, the Ces\`aro means converge to $0$ $\lambda$-a.e., so we again get from Proposition~\ref{Prop(w')*Implies*} that $0$ is the only candidate for their $\newtopology$-limit (as convergence in measure implies existence of a subsequence converging a.e.). However, we clearly have
	\begin{equation*}
		\chi_{[0, \infty)} \geq \left( \frac{1}{n}\sum_{i = 0}^{n-1} T_{\phi}^i f \right)^* \geq \chi_{[n-1, \infty)}^* = \chi_{[0, \infty)},
	\end{equation*}
	whence the Ces\`aro means do not converge to $0$ when measured with any of the norms $\lvert \cdot \rvert_{\varphi}$, $\varphi \in X'$, $\varphi \neq 0$.
\end{example}

\section{Pointwise ergodicity of composition operators on r.i.~Banach function spaces}\label{SectionPointwise}

We now turn our attention to pointwise ergodicity. We shall show that norm convergence of the Ces\`aro means implies pointwise convergence, provided that the symbol $\phi$ satisfies some additional assumptions. The key tool to obtain pointwise convergence is, as usual, an appropriate version of the Maximal Ergodic Theorem, that in turn is obtained via interpolation.

\subsection{An elementary interpolation theorem}

To prove our version of the Maximal Ergodic Theorem, we need the weak-type interpolation result that we present below. This result is rather simple, or, to put it more precisely, if follows in a quite simple fashion from several know facts. However, we have not been able to locate it in literature, hence we present it here with a full proof.

The result is naturally formulated for sublinear operators, which is also the setting we need (since the maximal ergodic operator, as defined in Definition~\ref{DefMaxErgodicOp} is clearly not linear). We recall that the appropriate definitions can be found in Section~\ref{SecSublinearOps}.

\begin{theorem} \label{ThmInterpolation}
    Let $S$ be a sublinear operator that is bounded both as $S \colon L^{\infty} \to L^{\infty}$ and as $S \colon L^1 \to L^{1, \infty}$, where all spaces are considered to be equipped with the standard (quasi)normed topologies. Then $S$ satisfies for every $f \in L^1 + L^{\infty}$ that $(Sf)^* \lesssim f^{**}$ on $[0, \infty)$.

    Consequently, it holds for every r.i.~Banach function space $X$, every $f \in X$, and every $s \geq  0$ that
    \begin{equation} \label{ThmInterpolation:E1}
        s \varphi_X \Big ( \mu( \{ \lvert Sf \rvert > s \} ) \Big ) \lesssim \lVert f \rVert_X,
    \end{equation}
    where $\varphi_X$ is the fundamental function of $X$ and where the hidden constant depends only on $\lVert \cdot \rVert_X$, not on $f$ or $s$.
\end{theorem}

Note that the left-hand side of \eqref{ThmInterpolation:E1} is well defined since $\mu( \{ \lvert Sf \rvert > s \} )$ is clearly in the range of $\mu$ for every s. Furthermore, the theorem is on purpose formulated in a way that avoids stating that $S$ is bounded from $X$ to $m_{\varphi_X}$. While this is indeed the correct intuition, such a statement would lead to some technical issues in the cases when the underlying measure space is not $([0, \infty), \lambda)$. Given that we will only need the estimate \eqref{ThmInterpolation:E1}, we chose to formulate the statement in a way that avoids these issues.

\begin{proof}
    That the assumed boundedness properties of the operator $S$ guarantee $(Sf)^* \lesssim f^{**}$ on $[0, \infty)$ for every $f \in L^1 + L^{\infty}$ is known, although the task of providing a specific reference has turned out to be more difficult than expected. The result certainly follows immediately from the knowledge of the $K$-functionals for the couples $(L^1, L^{\infty})$ (see e.g.~\cite[Chapter~2, Theorem~6.2]{BennettSharpley88}) and $(L^{1,\infty},L^{\infty})$ (see e.g.~\cite[Commentary after Corollary~3.1]{Bennett73} and the references therein). Alternatively, a more elementary approach using only the $K$-functional of $(L^1, L^{\infty})$ can be used; see e.g.~\cite[Chapter~3, Proof of Theorem~3.8]{BennettSharpley88} (the proof is performed for the Hardy--Littlewood maximal operator, but the estimate we need uses only the boundedness properties we assume and subadditivity).

    Let us now consider $\overline{X}$, the canonical representation space of $X$, and $\varphi_{\overline{X}}$, its fundamental function. Then we immediately obtain for every $f \in \mathcal{M}(\mathcal{R}, \mu)$
    \begin{equation*}
        \lVert (Sf)^* \rVert_{m_{\varphi_{\overline{X}}}} = \sup_{t \in [0, \infty)} \varphi_{\overline{X}}(t) (Sf)^*(t) \leq \sup_{t \in [0, \infty)} \varphi_{\overline{X}}(t) f^{**}(t) = \lVert f^* \rVert_{M_{\varphi_{\overline{X}}}} \lesssim \lVert f^* \rVert_{\overline{X}},
    \end{equation*}
    where the last estimate follows from the embedding $\overline{X} \hookrightarrow M_{\varphi_{\overline{X}}}$ (see e.g.~\cite[Proposition~5.9]{BennettSharpley88}). Now, the right-hand side of the estimate is obviously equal to the right-hand side of \eqref{ThmInterpolation:E1}, while the left-hand side of the estimate is larger than or equal to the left-hand side of \eqref{ThmInterpolation:E1}, as follows from \cite[Lemma~4.14]{MusilovaNekvinda24}, the fact that $Sf$ and $(Sf)^*$ are equimeasurable, and the fact that the functions $\varphi_X$ and $\varphi_{\overline{X}}$ coincide on the range of $\mu$.
\end{proof}

\subsection{A maximal ergodic theorem} The theorem we require is a statement about the boundedness of the following maximal ergodic operator:

\begin{definition} \label{DefMaxErgodicOp}
    Let $\phi \colon \mathcal{R} \to \mathcal{R}$ be non-singular. We then define the maximal ergodic operator $T^{ \#}_{\phi} \colon \mathcal{M}(\mathcal{R}, \mu) \to \mathcal{M}(\mathcal{R}, \mu)$ by the formula
    \begin{align*}
        T^{ \#}_{\phi}(f) &= \sup_{\substack{n \in \mathbb{N} \\ n \geq 1}} \frac{1}{n} \sum_{i = 0}^{n-1} \lvert T_{\phi}^i f \rvert, &f \in \mathcal{M}(\mathcal{R}, \mu),
    \end{align*}
    where both the sum and the supremum are to be interpreted pointwise.
\end{definition}

\begin{remark} \label{RemMaxErgodicOp}
    Let $\phi \colon \mathcal{R} \to \mathcal{R}$ be non-singular. Then $T^{ \#}_{\phi}$ is clearly a sub-linear operator that is well defined on $\mathcal{M}(\mathcal{R}, \mu)$ (we stress that we make no claims about finiteness almost everywhere). Furthermore, it is also clear that it is a contraction on $L^{\infty}$. Finally, it is clear that we have $T^{ \#}_{\phi}(f) = T^{ \#}_{\phi}(\lvert f \rvert)$ for every $f \in \mathcal{M}$, whence any further boundedness properties can be established by examining just the non-negative functions.
\end{remark}

Let us now present the assumptions under which we are able to prove the required boundedness of the maximal ergodic operator.

\begin{definition} \label{Def(I)}
    We say, that $\phi \colon \mathcal{R} \to \mathcal{R}$ satisfies the condition (I), if it is power-measure-bounded and further satisfies at least one of the following conditions:
    \begin{enumerate}[label=(I\arabic*)]
        \item \label{I1} Its measure-bound $A$ satisfies $A \leq 1$.
        \item \label{I3} There is a constant $C > 0$ such that we have for every $i \in \mathbb{N}$ and every measurable $E \subseteq \mathcal{R}$
            \begin{equation*}
                C \mu (E) \leq \mu ( \phi^{-i} ( E )).
            \end{equation*}
    \end{enumerate}
\end{definition}

It is evident that the assumptions \ref{I1} and \ref{I3} go in quite different direction from each other. Furthermore, even the proofs are completely different for the two cases (as evidenced in Proof of Theorem~\ref{ThmMaxErgodicL1} below). Hence, it would be very interesting to know whether there is some common denominator for these conditions.

We are now prepared to prove our version of the Maximal Ergodic Theorem. We start by establishing a weak-type boundedness on $L^1$ and then employ the interpolation principle established in Theorem~\ref{ThmInterpolation}.

\begin{theorem} \label{ThmMaxErgodicL1}
    Let $\phi \colon \mathcal{R} \to \mathcal{R}$ satisfy the condition (I). Then $T^{ \#}_{\phi} \colon L^1 \to L^{1, \infty}$ is correctly defined and bounded.
\end{theorem}

\begin{proof}
    When $\phi$ satisfies \ref{I1}, then $T_{\phi}$ is a contraction on $L^1$, while it is also trivially a contraction on $L^{\infty}$. Whence, we may apply the classical result of Dunford and Schwartz (see e.g.~\cite[Section~3, Lemma~7]{DunfordSchwartz56}) on $\lvert f \rvert$ (so that our definitions of the maximal ergodic operator coincide) and obtain the desired boundedness.

    When $\phi$ satisfies \ref{I3}, we obtain the desired boundedness via the Calderón transference principle, in a way analogous to that of \cite[Section~2.6.3]{EinsiedlerWard11}:

    Fix arbitrary $\alpha > 0$. Denote by $M$ the version of a discrete maximal operator that is given by the formula (for an arbitrary sequence $(a(i))_{i \in \mathbb{Z}}$)
    \begin{align*}
        Ma(j) &= \sup_{\substack{n \in \mathbb{N} \\ n \geq 1}} \frac{1}{n} \sum_{i = 0}^{n-1} \lvert a(j+i) \rvert &\text{for } j \in \mathbb{Z}.
    \end{align*}
    It is well known that $M \colon \ell^1 \to \ell^{1, \infty}$ is a bounded sub-linear operator, see e.g.~\cite[Lemma~2.29]{EinsiedlerWard11} (applied on $\lvert f \rvert$). We will further need the $K$-approximation of $M$, for $K \in \N$:
    \begin{align*}
        M_K a(j) &= \sup_{\substack{n \in \mathbb{N} \\ 1 \leq n \leq K}} \frac{1}{n} \sum_{i = 0}^{n-1} \lvert a(j+i) \rvert &\text{for } j \in \mathbb{Z},
    \end{align*}
    as well as the corresponding approximation $T^{ \#}_{\phi, K}$ of $T^{ \#}_{\phi}$: 
    \begin{align*}
        T^{ \#}_{\phi, K}(f) &= \sup_{\substack{n \in \mathbb{N} \\ 1 \leq n \leq K}} \frac{1}{n} \sum_{i = 0}^{n-1} \lvert T_{\phi}^i f \rvert.
    \end{align*}
    We now fix $K$ and also $J \in \N$, $J \geq 2K$. For fixed $x \in \mathcal{R}$, we define the sequence $(a_x(i))_{i \in \mathbb{Z}}$ by
    \begin{align*}
        a_x(i) &= \begin{cases}
            T_{\phi}^i f(x), &\text{for } 0 \leq i \leq J, \\
            0 &\text{otherwise}.
        \end{cases}
    \end{align*}
    Clearly, if $0 \leq j \leq J - K$, then 
    \begin{equation} \label{ThmMaxErgodicL1:E1}
        M_K a_x(j) = T^{ \#}_{\phi, K}(f)(\phi^j(x)).
    \end{equation}
    Furthermore, since $ M_K \leq M \colon \ell^1 \to \ell^{1, \infty}$, we get
    \begin{equation*}
        \alpha m \left( \left \{ j \in \mathbb{Z}; \; 0 \leq j \leq J-K, \, M_K a_x(j) > \alpha \right \} \right) \leq \alpha m \left( \left \{ j \in \mathbb{Z}; \; M a_x(j) > \alpha \right \} \right) \lesssim \lVert a_x \rVert_{\ell^1},
    \end{equation*}
    where we recall that $m$ is the counting measure on $\mathbb{Z}$. Using \eqref{ThmMaxErgodicL1:E1}, this may be reformulated as
    \begin{equation} \label{ThmMaxErgodicL1:E2}
        \alpha \sum_{j = 0}^{J-K} \chi_{ \phi^{-j}( \{ T^{ \#}_{\phi, K}(f) > \alpha\} ) } (x) \lesssim \sum_{j = 0}^{J} T_{\phi}^j \lvert f(x) \rvert.
    \end{equation}
    Since $\phi$ is power-measure-bounded (and thus power-bounded on $L^1$, see Theorem~\ref{ThmPowerBoundedQBFS}) and assumed to satisfy \ref{I3}, we obtain after integrating \eqref{ThmMaxErgodicL1:E2} over $x \in \mathcal{R}$ that
    \begin{equation*}
        (J-K+1) \alpha \mu \left( \left \{ T^{ \#}_{\phi, K}(f) > \alpha \right \} \right) \lesssim (J+1) \lVert f \rVert_{L^1}.
    \end{equation*}
    We note here that the hidden constant depends only on the bound of $M \colon \ell^1 \to \ell^{1, \infty}$, the power-measure-bound of $\phi$ and the constant $C$ from \ref{I3}. Since $J \geq 2K$ was arbitrary, it follows that
    \begin{equation*}
        \alpha \mu \left( \left \{ T^{ \#}_{\phi, K}(f) > \alpha \right \} \right) \lesssim \lVert f \rVert_{L^1},
    \end{equation*}
    for every $K$, where the hidden constant does not depend on $K$. Letting $K \to \infty$ thus concludes the proof.
\end{proof}

\begin{remark}
    It is worth noting, that the condition \ref{I3} is closely related, but not quite the same as the assumptions of the recent result \cite[Theorem~3.1]{Martin-ReyesDelaRosa22}. There, the authors assume that the operator in question is a positive and invertible Lamperti operator whose inverse is positive and such that both the operator and its inverse are Ces\`aro bounded on $L^1$ and $L^{\infty}$. As it turns out, if the symbol $\phi$ is power-measure-bounded and satisfies \ref{I3}, then the only thing that stops $T_{\phi}$ from qualifying is the possible lack of surjectivity. Indeed, composition operators are a subclass of the Lamperti ones; $T_{\phi}$ is positive, power-bounded on $L^1$ and $L^{\infty}$ (by Theorem~\ref{ThmPowerBoundedQBFS}), and injective (by Proposition~\ref{PropCharInjectivity}); and the inverse is positive (by Proposition~\ref{PropCharInjectivity}) and continuous on its domain (by Theorem~\ref{ThmInjectiveAndClosedRangeQBFS}). Finally, if the operator were surjective, and thus invertible, then one could obtain power-boundedness of the inverse by modifying the proof of Theorem~\ref{ThmPowerBoundedQBFS} (in a similar way in which the proof of Theorem~\ref{ThmBoundedQBFS} is modified to obtain Theorem~\ref{ThmInjectiveAndClosedRangeQBFS}).
\end{remark}

By combining Theorem~\ref{ThmInterpolation} with Theorem~\ref{ThmMaxErgodicL1} (and Remark~\ref{RemMaxErgodicOp} for the bound $T^{ \#}_{\phi} \colon L^{\infty} \to L^{\infty}$) we obtain the general weak-type boundedness on every r.i.~Banach function space that we require for our pointwise ergodic result.

\begin{corollary} \label{CorMaxErgodicX}
    Let $\phi \colon \mathcal{R} \to \mathcal{R}$ satisfy condition (I) and let $X$ be an r.i.~Banach function space. Then $T^{ \#}_{\phi}$ satisfies for every $f \in X$ and every $s \in [0, \infty)$
    \begin{equation} \label{CorMaxErgodicX:E1}
        s \varphi_X \Bigg ( \mu \left( \left \{ T^{ \#}_{\phi} f > s \right \} \right) \Bigg ) \lesssim \lVert f \rVert_X,
    \end{equation}
    where $\varphi_X$ is the fundamental function of $X$ and where the hidden constant depends only on $\lVert \cdot \rVert_X$, not on $f$ or $s$.
\end{corollary}

We note that we have left out the modulus around $T^{ \#}_{\phi} f$ because it is always non-negative.

\subsection{The pointwise ergodic theorem}

We may now finally prove the main result of this section, i.e.~our version of the Pointwise Ergodic Theorem in its full strength. The proof follows the standard approach, see e.g.~\cite[Section~2.6]{EinsiedlerWard11}, where we use Corollary~\ref{CorMaxErgodicX} to close the gap between norm and pointwise convergence. Of course, the fact that we work with a general r.i.~Banach function space $X$ instead of $L^1$ adds some technical difficulties that have to be taken care of.

\begin{theorem} \label{ThmErgodicPointwise}
	Let $X$ be an r.i.~Banach function space and let $X'$ be the corresponding associate space. Assume that both $X$ and $X'$ have the \ACR property. Let $\phi \colon \mathcal{R} \to \mathcal{R}$ satisfy the condition (I). Consider the composition operator $T_{\phi}$ and the operator $T$ defined as in \eqref{ThmErgodic(w')*:E0}. Then we have for every $f \in X_a$ that 
	\begin{align}
		\frac{1}{n}\sum_{i = 0}^{n-1} T_{\phi}^i f &\to Tf & \mu\text{-a.e.~as } n \to \infty. \label{ThmErgodicPointwise:E0}
	\end{align}
\end{theorem}

\begin{proof}
    Recalling Proposition~\ref{PropXaAlternative}, we see that we may assume without loss of generality that simple functions are contained and $\lVert \cdot \rVert_X$-dense in $X_a$. We thus begin by showing the pointwise convergence for the simple functions and then extend to general functions in $X_a$.

    Let $f \in X_a$ be a simple function and fix arbitrary $\varepsilon, \varepsilon_0 > 0$. Then $\lVert f \rVert_{L^{\infty}} = f^*(0) < \infty$ and Theorem~\ref{ThmErgodicNormLocalised} guarantees that there is an $M \in \mathbb{N}$ such that
    \begin{equation*}
        \left \lVert Tf - \frac{1}{M}\sum_{i = 0}^{M-1} T_{\phi}^i f \right \rVert_X < \varepsilon \varepsilon_0.
    \end{equation*}
    Consequently, Corollary~\ref{CorMaxErgodicX} yields
    \begin{equation} \label{ThmErgodicPointwise:E1}
        \varepsilon \varphi_X \left( \mu \left( \left \{ T^{ \#}_{\phi} \left( Tf - \frac{1}{M}\sum_{i = 0}^{M-1} T_{\phi}^i f \right) > \varepsilon  \right \} \right)  \right) \lesssim \varepsilon \varepsilon_0, 
    \end{equation}
    where $\varphi_X$ is the fundamental function of $X$ and the hidden constant depends only on $X$.

    Now, $Tf$ is invariant with respect to $\phi$ (i.e.~$T_{\phi}(Tf) = Tf$), so we have for every $n \in \mathbb{N}$, $n\geq 1$
    \begin{equation} \label{ThmErgodicPointwise:E2}
        \frac{1}{n}\sum_{j = 0}^{n-1} T_{\phi}^j (Tf) = Tf.
    \end{equation}
    On the other hand, we get for $n \geq 2M$ that
    \begin{equation} \label{ThmErgodicPointwise:E3}
        \frac{1}{n}\sum_{j = 0}^{n-1} T_{\phi}^j \left( \frac{1}{M}\sum_{i = 0}^{M-1} T_{\phi}^i f \right) = \frac{1}{n}\sum_{j = 0}^{n-1} T_{\phi}^j f + \frac{1}{nM}\sum_{i = 1}^{M-1} \sum_{j = 0}^{i-1} \left( T_{\phi}^{n+j}f -  T_{\phi}^j f \right).
    \end{equation}
    We further observe that our assumption $\lVert f \rVert_{L^{\infty}} < \infty$ implies
    \begin{align} \label{ThmErgodicPointwise:E4}
         \left \lvert \frac{1}{nM}\sum_{i = 1}^{M-1} \sum_{j = 0}^{i-1} \left( T_{\phi}^{n+j}f -  T_{\phi}^j f \right) \right \rvert < \frac{1}{n} \frac{2M^2}{M} \lVert f \rVert_{L^{\infty}} &\to 0 &\text{as } n \to \infty.
    \end{align}
    Finally, we get the following estimate, where the first equality follows by putting \eqref{ThmErgodicPointwise:E2}, \eqref{ThmErgodicPointwise:E3}, \eqref{ThmErgodicPointwise:E4} together:
    \begin{equation*}
    \begin{split}
        \limsup_{n \to \infty} \left\lvert Tf - \frac{1}{n}\sum_{j = 0}^{n-1} T_{\phi}^j f \right \rvert &= \limsup_{n \to \infty} \left\lvert \frac{1}{n}\sum_{j = 0}^{n-1} T_{\phi}^j \left( Tf - \frac{1}{M}\sum_{i = 0}^{M-1} T_{\phi}^i f \right) \right \rvert \\
        &\leq T^{ \#}_{\phi} \left( Tf - \frac{1}{M}\sum_{i = 0}^{M-1} T_{\phi}^i f \right).
    \end{split}
    \end{equation*}
    When we combine this estimate with \eqref{ThmErgodicPointwise:E1}, we obtain (we recall that $\varphi_X$ is non-decreasing)
    \begin{equation*}
        \varphi_X \left( \mu \left( \left \{ \limsup_{n \to \infty} \left\lvert Tf - \frac{1}{n}\sum_{j = 0}^{n-1} T_{\phi}^j f \right \rvert > \varepsilon  \right \} \right)  \right) \lesssim \varepsilon_0.
    \end{equation*}
    Since $\varepsilon_0$ was arbitrary and $\varphi_X$ is zero only at zero, we conclude that
    \begin{align*}
        \limsup_{n \to \infty} \left\lvert Tf - \frac{1}{n}\sum_{j = 0}^{n-1} T_{\phi}^j f \right \rvert &\leq \varepsilon &\mu\text{-a.e.}
    \end{align*}
    As $\varepsilon$ was also arbitrary, this means that \eqref{ThmErgodicPointwise:E0} holds for $f$.

    Let now $f \in X_a$ be arbitrary. Assuming that $f$ is not the zero function (otherwise both this step and the validity of \eqref{ThmErgodicPointwise:E0} are trivial), we may use Proposition~\ref{PropXaAlternative} to find a sequence $f_k$ of simple functions such that $f_k \to f$ in $(X, \lVert \cdot \rVert_X)$. We again fix arbitrary $\varepsilon, \varepsilon_0 > 0$ and find some $k_0$ such that
    \begin{equation} \label{ThmErgodicPointwise:E5}
        \lVert f - f_{k_0} \rVert_X < \varepsilon \varepsilon_0.
    \end{equation}
    Since Theorem~\ref{ThmErgodic(w')*} asserts that $T$ is continuous on $(X, \lVert \cdot \rVert_X)$, \eqref{ThmErgodicPointwise:E5} implies that we also have
    \begin{equation} \label{ThmErgodicPointwise:E6}
        \lVert Tf - Tf_{k_0} \rVert_X \lesssim \varepsilon \varepsilon_0,
    \end{equation}
    where, of course, the hidden constant does not depend on the values of $\varepsilon, \varepsilon_0, k_0$. Let us now establish the following notation:
    \begin{align*}
        E_0 &= \left \{ \limsup_{n \to \infty} \left \lvert Tf - \frac{1}{n}\sum_{j = 0}^{n-1} T_{\phi}^j f \right \rvert > 3\varepsilon  \right \}, \\
        E_1 &= \left \{ \lvert Tf - Tf_{k_0} \rvert > \varepsilon  \right \}, \\
        E_2 &= \left \{ \limsup_{n \to \infty} \left \lvert Tf_{k_0} - \frac{1}{n}\sum_{j = 0}^{n-1} T_{\phi}^j f_{k_0} \right \rvert > \varepsilon  \right \}, \\
        E_3 &= \left \{ \limsup_{n \to \infty} \left \lvert \frac{1}{n}\sum_{j = 0}^{n-1} T_{\phi}^j f_{k_0} - \frac{1}{n}\sum_{j = 0}^{n-1} T_{\phi}^j f \right \rvert > \varepsilon  \right \}.
    \end{align*}
    Then the triangle inequality, subadditivity of limes superior, and a standard combinatorial argument yield that $E_0 \subseteq E_1 \cup E_2 \cup E_3$ and thus
    \begin{equation} \label{ThmErgodicPointwise:E6.1}
        \mu(E_0) \leq \mu(E_1) + \mu(E_2) + \mu(E_3).
    \end{equation}

    Now we must deal with some technicalities concerning the domain of $\varphi_X$; we do so in the following way: Consider $\varphi_{\overline{X}}$, the fundamental function of $\overline{X}$, the canonical representation space of $X$. As $([0, \infty), \lambda)$ is non-atomic and of infinite measure, the domain of $\varphi_{\overline{X}}$ is the entire interval $[0, \infty)$ and the function is therefore subadditive on this interval (this follows immediately from the definition by considering disjoint sets of given measures). Furthermore, it is non-decreasing and coincides with $\varphi_X$ on the range of $\mu$ (i.e.~the domain of $\varphi_X$). Thence, we may compute (using \eqref{ThmErgodicPointwise:E6.1})
    \begin{equation} \label{ThmErgodicPointwise:E7}
    \begin{split}
        \varphi_X \left( \mu \left( \left \{ \limsup_{n \to \infty} \left\lvert Tf - \frac{1}{n}\sum_{j = 0}^{n-1} T_{\phi}^j f \right \rvert > 3\varepsilon  \right \} \right)  \right) &\leq \varphi_{\overline{X}} ( \mu(E_1) + \mu(E_2) + \mu(E_3) ) \\
        &\leq \varphi_X ( \mu(E_1) ) + \varphi_X ( \mu(E_2) ) + \varphi_X ( \mu(E_3) ).
    \end{split}
    \end{equation}
    (The reason for using $\varphi_{\overline{X}}$ in the intermediate step is that the sum of the measures may lie outside the domain of $\varphi_X$.)

    We now have to estimate the three terms on the right-hand side of \eqref{ThmErgodicPointwise:E7}. $\varphi_X ( \mu(E_2) )$ is the simplest, as $f_{k_0}$ is a simple function, whence
    \begin{equation}
        \varphi_X ( \mu(E_2) ) = \varphi_X ( 0 ) = 0,
    \end{equation}
    as proven in the first step. 
    
    As for $\varphi_X ( \mu(E_1) )$, it follows from the fact that $(Tf - Tf_{k_0})^*$ is equimeasurable with $Tf - Tf_{k_0}$, the characterisation of $\lVert \cdot \rVert_{m_{\varphi_{\overline{X}}}}$ obtained in \cite[Lemma~4.14]{MusilovaNekvinda24}, the embedding $\overline{X} \hookrightarrow m_{\varphi_{\overline{X}}}$, and \eqref{ThmErgodicPointwise:E6} that
    \begin{equation*}
    \begin{split}
        \varepsilon \varphi_X \left( \mu \left( \left \{ \lvert Tf - Tf_{k_0} \rvert > \varepsilon  \right \} \right) \right) &= \varepsilon \varphi_{\overline{X}} \left( \lambda \left( \left \{ (Tf - Tf_{k_0})^*  > \varepsilon  \right \} \right) \right)\\
        &\leq \lVert (Tf - Tf_{k_0})^* \rVert_{m_{\varphi_{\overline{X}}}} \\
        &\lesssim \lVert (Tf - Tf_{k_0})^* \rVert_{\overline{X}} \\
        &= \lVert Tf - Tf_{k_0} \rVert_{X} \\
        &\lesssim \varepsilon \varepsilon_0,
    \end{split}
    \end{equation*}
    where, of course, the hidden constant does not depend on the values of $\varepsilon, \varepsilon_0, k_0$.

    Finally, considering $\varphi_X ( \mu(E_3) )$, we immediately observe that
    \begin{equation*}
        \limsup_{n \to \infty} \left \lvert \frac{1}{n}\sum_{j = 0}^{n-1} T_{\phi}^j f_{k_0} - \frac{1}{n}\sum_{j = 0}^{n-1} T_{\phi}^j f \right \rvert \leq T^{ \#}_{\phi} (  f - f_{k_0} ).
    \end{equation*}
    Therefore, Corollary~\ref{CorMaxErgodicX} together with \eqref{ThmErgodicPointwise:E5} imply
    \begin{equation*}
        \varepsilon \varphi_X \left( \mu \left( \left \{  \limsup_{n \to \infty} \left \lvert \frac{1}{n}\sum_{j = 0}^{n-1} T_{\phi}^j f_{k_0} - \frac{1}{n}\sum_{j = 0}^{n-1} T_{\phi}^j f \right \rvert > \varepsilon  \right \} \right) \right) \lesssim \varepsilon \varepsilon_0.
    \end{equation*}
    Again, the hidden constant does not depend on the values of $\varepsilon, \varepsilon_0, k_0$.

    Plugging the above-obtained three estimates into \eqref{ThmErgodicPointwise:E7}, we conclude that
    \begin{equation*}
        \varphi_X \left( \mu \left( \left \{ \limsup_{n \to \infty} \left\lvert Tf - \frac{1}{n}\sum_{j = 0}^{n-1} T_{\phi}^j f \right \rvert > 3\varepsilon  \right \} \right)  \right) \lesssim 2 \varepsilon_0.
    \end{equation*}
    As above, it follows from the facts that both $\varepsilon_0$ and $\varepsilon$ we arbitrary and that $\varphi_X$ is zero only at zero that \eqref{ThmErgodicPointwise:E0} holds for $f$.
\end{proof}

\bibliographystyle{dabbrv}
\bibliography{bibliography}
\end{document}